\documentclass[11pt,a2chapter]{article}

\usepackage[T1]{fontenc} 
\usepackage{lmodern}

\usepackage[utf8]{inputenc}
\usepackage{amsfonts} 
\usepackage{dsfont}
\usepackage{amsmath}
\usepackage{amsmath,latexsym}
\usepackage{amsmath,amssymb}
\usepackage{amsmath,amsthm}
\usepackage[english]{babel}
\usepackage{moreverb}
\usepackage[]{graphics}
\usepackage[]{latexsym} 
\usepackage{amscd}
\usepackage{color}
\usepackage{lineno}
\usepackage{ae}
\usepackage{enumerate}
\usepackage{multicol}
 \usepackage[figuresright]{rotating}
\usepackage{bigints}
 \graphicspath{{FIGURES/}}
 \makeatletter
\newcommand{\xrightharpoonup}[2][]{\ext@arrow 0359\rightharpoonupfill@{#1}{#2}}
\makeatother
\usepackage{hyperref}
\hypersetup{
	colorlinks=true,
	linkcolor=blue,
	filecolor=magenta,      
	urlcolor=cyan,
}

\usepackage{graphics}
\usepackage{graphicx} 
 
\newcommand{\ve}{\varepsilon}  
\newcommand{\ds}{\displaystyle}

\def\R{\mathbb{R}}

\def\P{\mathbb{P}}
\def\E{\mathbb{E}}

\def\N{{\rm I\hspace{-0.50ex}N} }
\def\N{\mathbb{N}}

\newtheorem{lem}{Lemma}[section]
\newtheorem{thm}[lem]{Theorem}

\newtheorem{prop}[lem]{Proposition}

\newtheorem{rmq}[lem]{Remark}

\title{\bf Quantization-based approximation of reflected BSDEs with extended upper bounds for recursive quantization}
\author{\textsc{Rancy El Nmeir} \thanks{Sorbonne Universit\'e, Laboratoire de Probabilit\'e, Statistique et Mod\'elisation, Campus Pierre et Marie Curie, case 158, 4, pl. Jussieu, F-75252 Paris Cedex 5, France.} \thanks{ Universit\'e Saint-Joseph de Beyrouth, Laboratoire de Math\'ematiques et Applications, Unit\'e de recherche Math\'ematiques et mod\'elisation, B.P. 11-514 Riad El Solh Beyrouth 1107
2050, Liban.}
\& \textsc{Gilles Pag\`es}\footnotemark[1]% \thanks{Sorbonne Universit\'e, Laboratoire de Probabilit\'e, Statistique et Mod\'elisation, Campus Pierre et Marie Curie, case 158, 4, pl. Jussieu, F-75252 Paris Cedex 5, France.}
}
\begin{document}
	\date{}
	\maketitle
	\begin{abstract}
	We establish upper bounds for the $L^p$-quantization error, $p\in (1,2+d)$, induced by the recursive Markovian quantization of a $d$-dimensional diffusion discretized via  the Euler scheme. We introduce a {\em hybrid} recursive quantization scheme, easier to implement in the high-dimensional framework, and establish upper bounds to the corresponding $L^p$-quantization error. To take advantage of these extensions, we propose a time discretization scheme and a recursive quantization-based discretization scheme associated to a Reflected Backward Stochastic Differential Equation and estimate $L^p$-error bounds induced by the space approximation. We will explain how to numerically compute the solution of the reflected BSDE relying on the recursive quantization and compare it to other types of quantization. 
	\end{abstract}
	\bigskip 
	%\fi
	%\chapter{Quantization-based approximation of reflected BSDEs with extended upper bounds for recursive quantization}
	%\label{RBSDE+QR}
	%\paragraph{Abstract}
	%We establish upper bounds for the $L^p$-quantization error, $p\in (1,2+d)$, induced by the recursive Markovian quantization of a $d$-dimensional diffusion discretized via  the Euler scheme. We introduce a {\em hybrid} recursive quantization scheme, easier to implement in the high-dimensional framework, and establish upper bounds of the corresponding $L^p$-quantization error. To take advantage of these extensions, we propose a time discretization scheme and a recursive quantization-based discretization scheme associated to a Reflected Backward Stochastic Differential Equation and estimate $L^p$-error bounds induced by the space approximation. We explain how to numerically compute the solution of the reflected BSDE relying on recursive quantization and compare it to others types of quantization. 
\section{Introduction}
We are interested in the discretization and the computation of the solution of the following reflected backward stochastic differential equation RBSDE with maturity $T$
\begin{equation}
\label{BSDE}
Y_t=g(X_T)+\int_t^T f(s,X_s,Y_s,Z_s)ds+K_T-K_t -\int_t^T Z_s . dW_s\, , \qquad t \in [0,T],
\end{equation}
\begin{equation}
Y_t \geq h(t,X_t) \quad \mbox {and} \quad \int_0^T \big(Y_s-h(s,X_s)\big)dK_s =0.
\end{equation}
$(X_t)_{t \geq 0}$ is a Brownian diffusion process taking values in $\R^d$ and solution to the SDE 
\begin{equation}
\label{SDE}
X_t=X_0+\int_0^t b(s,X_s)ds +\int_0^t \sigma(s,X_s)dW_s\, , \qquad X_0=x_0 \in \R^d,
\end{equation}
where the drift coefficient $b: [0,T] \times \R^d \rightarrow \R^d$ and the matrix diffusion coefficient $\sigma : [0,T] \times \R^d \rightarrow \mathcal{M}(d,q)$ are Lipschitz continuous in $(t,x)$ so that $b(.,0)$ and $\sigma(.,0)$ are bounded on $[0,T]$ and satisfy the linear growth condition 
$$\big\|\sigma(.,x)\big\|+\big\|b(.,x)\big\| \leq L_{b,\sigma} (1+\|x\|)$$
with $L_{b,\sigma}=\max\big([b]_{\rm Lip},[\sigma]_{\rm Lip}, \big\|b(.,0)\big\|_{\sup},\big\|\sigma(.,0)\big\|_{\sup}\big)$ and $\|\cdot\|$ denoting any norm on $\R^d$. $(W_t)_{t \geq 0}$ is a $q$-dimensional Brownian motion defined on the probability space $(\Omega, \mathcal{A},\P)$ equipped with its augmented natural filtration $(\mathcal{F}_t)_{t \geq 0}$ where $\mathcal{F}_t=\sigma(W_s, s\le t,\, \mathcal{N}_{\P})$, ${\cal N}_{\P}$ denotes the class of all $\P$-negligible sets of ${\cal A}$. The solution of this equation is defined as a $\R \times \R^d \times \R_+$-valued triplet $(Y_t, Z_t, K_t)$ of $\mathcal{F}_t$-progressively measurable square integrable processes. 
%such that $(Z_t)_{t \geq 0}$ is an $L^p$-integrable progressively measurable process taking values in $\R^q$ and 
$K_t$ is continuous, non-decreasing, such that $K_0 = 0$ and grows exclusively on $\{t:Y_t = h(t, X_t)\}$. The driver $f(t,x,y,z): [0,T] \times \R^d \times \R \times \R^d \rightarrow \R$ is $[f]_{\rm Lip}$-Lipschitz continuous with respect to $(t,x,y,z)$, $g(X_T)$ is the terminal condition where $g:\R^d \rightarrow \R$ is $[g]_{\rm Lip}$-Lipschitz continuous and $h:[0,T] \times \R^d \to \R$ is $[h]_{\rm Lip}$-Lipschitz continuous such that $g \geq h$ for every $t$ and $x$.
Under these assumptions on $b, \sigma ,h,g$ and $f$, the RBSDE $(\ref{BSDE})$ and the SDE $(\ref{SDE})$ admit both a unique solution. The existence of a process $(Y_t,Z_t,K_t)$, solution of $(\ref{BSDE})$, was established in \cite{KarPeQue97a} where the authors also showed that this solution satisfies the following property
\begin{equation}
\label{assumptionRBSDE}
\Big\|\sup_{t \in [0,T]}|Y_t|\Big\|_{2p} \vee \|K_T\|_{2p} \vee \Big\|\int_0^T |Z_t|^2 dt \Big\|_p < \gamma_0
\end{equation}
for a finite constant $\gamma_0$ (see also $\cite{BaPaSPA}$). In general, these solutions admit no closed form. Approximation schemes are needed to approximate them. In the literature, many authors studied different types of RBSDEs, for example, in $\cite{BaPaSPA,CvMa01,KarPeQue97a, MaZhang05, MaWang09}$ and many approximation schemes were investigated: Feynman-Kac type representation formula were given in $\cite{MaZhang05}$ for the solutions of RBSDEs, a four step algorithm was developed in \cite{MaPrYong94} to solve FBSDEs, a random time scheme in \cite{Bally97}. We can also cite the study of BSDE for quasi-linear PDEs in \cite{delarue08} and quadratic BSDEs in \cite{ChRi16}. In this paper, we start by a time discretization scheme of the forward process $(X_t)_{t \in [0,T]}$, the Euler scheme, with the uniform mesh $t_k=k\Delta$, $k \in \{0,\ldots,n\}$, with $\Delta=\frac{T}{n}$. The discrete time Euler scheme $(\bar{X}_{t_k})_{0 \leq k \leq n}$ associated to the process $(X_t)_{t \in [0,T]}$ is recursively defined by
\begin{equation}
\label{SDEeulerint}
\bar{X}_{t_{k+1}}^n=\bar{X}_{t_k}^n+\Delta b(t_k,\bar{X}_{t_k}^n)+\sigma(t_k,\bar{X}_{t_k}^n)\Delta W_{t_{k+1}}, \qquad \bar{X}_{t_0}^n=X_0=x_0\in \R^d, 
\end{equation}
where $\Delta W_{t_{k+1}}=W_{t_{k+1}}-W_{t_k}$, for every $k \in \{0.\ldots,n-1\}$. This leads to consider the time discretization scheme $(\bar{Y}_t^n,\bar{\zeta}_t^n)$ associated to $(Y_t,Z_t)$ given by the following backward recursion
\begin{align}
\label{YbarTint}
\bar{Y}_{T}^n & = g(\bar{X}_{T}^n)\\
\label{Ytildekint}
\widetilde{Y}_{t_k}^n &=\E(\bar{Y}_{t_{k+1}}^n|\mathcal{F}_{t_k})+\Delta f\big(t_k,\bar{X}_{t_k}^n,\E(\bar{Y}_{t_{k+1}}^n|\mathcal{F}_{t_k}),\bar{\zeta}_{t_k}^n\big)\, , \quad k=0,\ldots,n-1,\\
\label{Zetabarkint}
\bar{\zeta}_{t_{k}}^n&=\frac{1}{\Delta}\E\big(\bar{Y}_{t_{k+1}}^n(W_{t_{k+1}}-W_{t_k})\,|\,\mathcal{F}_{t_k}\big)\, , \quad k=0,\ldots,n-1,\\
\label{Ybarkint}
\bar{Y}_{t_{k}}^n& = \widetilde{Y}_{t_{k}}^n \vee h(t_k,\bar{X}_{t_k}^n)\, , \quad k=0,\ldots,n-1.
\end{align} 
It is important to notice that, in this scheme, the conditional expectation is applied directly to $\bar{Y}_{t_{k+1}}^n$ inside the driver function $f$ depending itself on the process $Z_t$ (or $\bar \zeta_{t_k}^n$). This is slightly different of what have been already introduced and investigated in the literature. In fact, such schemes were considered for BSDE (without reflection) in $\cite{PaSa18}$ and for doubly reflected BSDE in $\cite{theseilland}$, whereas in most papers in the literature, the expectation is usually applied to the driver $f$ from the outside. In some of these papers devoted to time(-space) discretization of RBSDE, the driver does not depend on the process $Z_t$, (see \cite{BaPaSPA,BaPaBern, BoTo04,MaZhang05} for example).\\

After the time discretization, the solution of the scheme $(\ref{YbarTint})-(\ref{Ytildekint})-(\ref{Zetabarkint})-(\ref{Ybarkint})$ still admits no closed form since it involves the computation of conditional expectations which cannot be obtained analytically. Therefore, we are led to devise a space discretization scheme to approximate it. In the literature, we can find various approaches: one can cite, among others, regression methods with Monte Carlo simulations (see \cite{BoTo04,GLTV16}), the multi-step schemes methods (see \cite{BeDe07}), a hybrid approach combining Picard iterates with a decomposition in Wiener chaos (see \cite{BrLa14}), a connection with the semi-linear PDE associated to the BSDE (see \cite{HLaTaTo14}) and Monte Carlo simulations with Malliavin calculus (see \cite{BoTo04, CrMaTo10, HuNuSo11, GoTu16}). Another approach is optimal quantization introduced for RBSDEs in $\cite{BaPaPr01}$ and then developed in a series of papers ($\cite{BaPaSPA, BaPaBern, theseilland, PaSa18}$ for example), quantization-based discretization schemes have also been used in \cite{delarue} for fully coupled Forward-Backward SDEs .  In this paper, we will rely on the recursive quantization of the time-discretized Euler scheme $(\bar{X}_{t_k}^n)_{0 \leq k \leq n}$. This method, originally introduced in $\cite{PaPhPr03}$ and then studied deeply in $\cite{McW18}$ and $\cite{PaSa14}$ for one-dimensional diffusions, consists in building a Markov chain having values into a {\em grid (or quantizer)} $\Gamma_k$ of the discrete Euler scheme $\bar{X}_{t_k}$ at time $t_k$. The grids $\Gamma_k$ can be optimized in a recursive way as a kind of {\em embedded} procedure.\\

In order to explain the principle of this recursive Markovian quantization, let us first recall briefly what optimal quantization is. Assume that $\R^d$ is equipped with a norm $\|\cdot\|$ (usually the canonical Euclidean norm for our purpose). Let $X\!\in L^p_{\R^d}(\Omega,{\cal A}, \P)$ and let $N\geq 1$ be a {\em quantization level}. The aim of $L^p$-optimal quantization is to find the best approximation of $X$ in $L^p(\P)$ by a  random vector $Y$ defined on $(\Omega, {\cal A}, \P)$ taking at most $N$ values. As a first step, we may consider the grid (or quantization grid)  $\Gamma^N= Y(\Omega)=  \{x_1,\ldots,x_{_N}\}$ (with possibly repeated elements). One easily checks that, $\Gamma^N$ being fixed, the best possible choice is given by a (Borel) nearest neighbor projection of $X$ on $\Gamma^N$. It is called a Vorono\"i quantization of $X$ defined by 
\begin{equation} 
	\widehat{X}^{\Gamma^N} =\text{Proj}_{\Gamma^N}(X) := \sum_{i=1}^{N}x_i \mathds{1}_{C_{i}(\Gamma^N)}(X)
\end{equation}
where $\big(C_i(\Gamma^N)\big)_{1\leq i\leq N}$ is a Borel partition of $\R^d$ satisfying  \begin{equation}
	\label{Voronoicells}
	C_{i}(\Gamma^N)  \subset\{\xi \in \mathbb{R}^d : \|\xi-x_i\|\leq  \min_{j\neq i} \|\xi-x_j\|\}, \qquad i =1,\ldots,N.
	\end{equation} The $N$-tuple $\big(C_i(\Gamma^N)\big)_{1\le i\le N}$ is called the {\em Vorono\"i partition} induced by $\Gamma^N$. The induced $L^p$-quantization error associated to the grid $\Gamma^N$ is defined by
	\begin{equation} 
	\label{quanterror}
	e_p(\Gamma^N,X)   = \|X-\widehat{X}^{\Gamma^N}\|_p
	\end{equation}
	where $\|.\|_p$ denotes the $L^p(\P)$-norm. The optimal quantization problem boils down to finding the grid $\Gamma^N$ that minimizes this error i.e. solving the problem
\[
e_{p,N}(X) := \inf_{\Gamma, |\Gamma|\le N} e_p(\Gamma, X).
\]
where $|\Gamma|$ denotes the cardinality of the grid $\Gamma$. A solution to this problem exists, as established in $\cite{GraLu00,Pages15,Pages18}$ for example, and is called an $L^p$-optimal quantization grid of (the distribution of) $X$. The corresponding quantization error converges to $0$ as $N$ goes to $+\infty$ and its rate of convergence is given by two well known results exposed in the following theorem.
	\begin{thm} 
		\label{Zadoretpierce}
		\noindent $(a)$ {\rm Zador's Theorem (see \cite{Zador82})}:  
		Let $X \in L_{\mathbb{R}^d}^{p+\eta}(\P)$, $\eta>0$, 
		with distribution $P$ having the following decomposition $P=h.\lambda_d+\nu$  where $\lambda_d$ denotes the Lebesgue measure on $(\R^d, {\cal B}or(\R^d))$ and $\nu \perp \lambda_d$ (singular). 
		Then, 
		\begin{equation}
		\label{Zador}
		\lim_{N\rightarrow +\infty} N^{\frac{1}{d}} e_{p,N}(X) = \tilde{J}_{p,d} \|\varphi\|^{\frac{1}{p}}_{L^{\frac{p}{p+d}}(\lambda_d)}
		\end{equation}
		where $\tilde{J}_{p,d} = \ds \inf_{N \geq 1} N^{\frac{1}{d}} e_{p,N}(\mathcal{U}([0,1]^d)) \in (0,+ \infty)$.  \\
		\noindent $(b)$ {\rm Extended Pierce's Lemma (see \cite{LuPa08,Pages18})}:
		Let $p,\eta >0 $. There exists a constant $\kappa_{d,p,\eta} \in (0,+\infty)$  such that, for any random vector $X: (\Omega, \mathcal{A},\P)\rightarrow \R^d$,
		\begin{equation}
		\label{Pierce}
		\forall N\geq 1,\quad e_{p,N}(X) \leq \kappa_{d,p,\eta} \sigma_{p+\eta} (X) N^{- \frac{1}{d}}
		\end{equation}
		where, for every $p \in (0, + \infty), \, \sigma_p(X) = \ds \inf_{a \in \mathbb{R}^d} \|X-a\|_p$ is the $L^p$-(pseudo-)standard deviation of $X$.
\end{thm}
\noindent An important property, shared by quadratic optimal quantizers, is the stationarity property: an $L^2$-optimal quantizer $\Gamma^N$ is said to be stationary if 
	\begin{equation}
	\label{RBSDE:stationary}
	\E(X|\widehat{X}^{\Gamma^N})= \widehat{X}^{\Gamma^N}.
	\end{equation}
Let us now explain what recursive quantization is. If we define the Euler operator with step $\Delta$ by $$\mathcal{E}_k(x,\ve_{k+1})=x+\Delta b(t_k,x) +\sqrt{\Delta} \sigma(t_k,x) \ve_{k+1}$$
where $(\ve_k)_{0 \leq k \leq n}$ is an i.i.d. sequence of random variables with distribution $\mathcal{N}(0,I_q)$, then the recursive quantization $(\widehat{X}_{t_k})_{0 \leq k \leq n}$ of $(\bar{X}_{t_k}^n)_{0 \leq k \leq n}$ is defined by  $\widehat{X}_{t_0}=\bar{X}_{t_0}^n=x_0$ and
\begin{equation}
\label{quantifrecursiveint}
\left\{
\begin{array}{rl}
\widetilde{X}_{t_{k}} &=\mathcal{E}_{k-1}(\widehat{X}^{\Gamma_{k-1}}_{t_{k-1}},\varepsilon_{k}),  \\
\smallskip
\widehat{X}_{t_k}^{\Gamma_{k}} & =\mbox{Proj}_{\Gamma_k}(\widetilde{X}_{t_k}), \qquad \quad \forall k= 1, \ldots, n
\end{array}
\right .
\end{equation}
where $(\Gamma_k)_{0 \leq k \leq n}$ is a sequence of optimal quantizers of $(\widetilde{X}_{t_k})_{0 \leq k \leq n}$ of size $N_k$, $k=0,\ldots,n$. The optimal quantizers $(\Gamma_k)_{1 \leq k \leq n}$ can be either quadratic or $L^p$-optimal quantizers, we will detail the difference between these two frameworks later in the paper. The main advantage of this method is that it preserves the Markov property of the Euler scheme with respect to the  filtration $({\cal F}_{t_k})_{0\le k \le n}$, the process $\widehat{X}_{t_k}$ is $\mathcal{F}_{t_k}$-measurable for every $k \in \{0,\ldots,n\}$. In fact, the transition matrices $(p_{ij}^k)_{1 \leq i,j\leq N_k}$ where $p_{ij}^k= \P\big( \widehat{X}_{t_{k+1}} \in C_j(\Gamma_{k+1}) \,|\, \widehat{X}_{t_k} \in C_i(\Gamma_{k}) \big)$ and the initial distribution characterize the distribution of the Markov chain $(\widehat{X}_{t_k})_{k \geq 0}$, which was not the case with the optimal quantization in $\cite{PaSa14}$ for example. This Markov property will bring much help to carry on computations of the weights $p_i^k$ of the Vorono\"i cells and the transition weights $p_{ij}^k$, as well as with the quantized scheme of the RBSDE itself.\\

Going back to our problem, we consider, in this paper, the recursive quantization scheme associated to $(\ref{YbarTint})$-$(\ref{Ytildekint})$-$(\ref{Zetabarkint})$-$(\ref{Ybarkint})$ based on the recursive quantization $(\widehat{X}_{t_k})_{0 \leq k \leq n}$ of the Euler scheme $(\bar{X}_{t_k}^n)_{0 \leq k \leq n}$. It is defined recursively in a backward way as follows:
\begin{align}
\label{YchapTint}
\widehat{Y}_{T}^n & = g(\widehat{X}_{T})\\
\label{Zetachapkint}
\widehat{\zeta}_{{t_k}}^n&=\frac{1}{\Delta}\E\, \big(\widehat{Y}_{t_{k+1}}^n(W_{t_{k+1}}-W_{t_k})\,|\, \mathcal{F}_{t_k}\big)\, , \quad  k=0,\ldots,n-1,\\
\label{Ychapkint}
\widehat{Y}_{t_{k}}^n& = \max\left(h_k(\widehat{X}_{t_k})\, ,\, \E\, \big(\widehat{Y}_{t_{k+1}}^n\,|\, \mathcal{F}_{t_k}\big)+\Delta f\big(t_k,\widehat{X}_{t_k},\E\big(\widehat{Y}^n_{t_{k+1}}\,|\,\mathcal{F}_{t_k}\big),\widehat{\zeta}^n_{t_k}\big) \right)\, , \quad k=0,\ldots,n-1,
\end{align} 
where $(\widehat{X}_{t_k})_{0 \leq k \leq n}$ is the recursively quantized Euler scheme associated to $(\bar{X}_{t_k}^n)_{0 \leq k \leq n}$ given by $(\ref{quantifrecursiveint})$. \\
As a preliminary step, we are interested in estimating the $L^p$-quantization error $\|\widehat{X}_{t_k}-\bar{X}_{t_k}^n\|_p$, not only for $p=2$ like in $\cite{PaSa14}$ but for any $p\in (1,2+d)$. The fact that we are limited to $p < 2+d$ will become clear later in the paper, as well as the type of optimal quantizers $\Gamma_k$ of $\widetilde X_{t_k}$ needed to obtain satisfactory upper bounds for the $L^p$-quantization error. Note that in the quadratic case $p=2$, the proof was based on a Pythagoras property which cannot be applied in a general framework. Furthermore, we introduce a kind of {\em hybrid} recursive quantization where the white noise $(\varepsilon_k)_{0\leq k \leq n}$ is replaced by its (already computed) quantized version $(\widehat{\ve}_k)_{0 \leq k \leq n}$.\\

In a second part, we will proceed with the time and space discretization of the RBSDE $(\ref{BSDE})$, as explained briefly before, and give more details about these schemes. We establish a priori estimates for the time discretization error $\|Y_{t_k}-\bar{Y}_{t_k}^n\|_{_2}$ in a quadratic case. Although time discretization have already been studied in the literature (see \cite{BaPaSPA,BoTo04,MaZhang05, PaSa18, Zhang04}), our approach is still different mostly because of the combination of the reflection in the backward SDE and the conditional expectation applied directly to $\bar Y_{t_k}^n$ and $\widehat Y_{t_k}^n$ inside the driver $f$ depending itself on the process $Z_t$ (or its approximations). Likewise, estimates for the space discretization error $\|\bar{Y}_{t_k}^n-\widehat{Y}_{t_k}^n\|_p$ in $L^p$ for $p \in (1,2+d)$ will be established. To illustrate these theoretical results, we detail the numerical techniques available to compute the recursive quantization $\widehat{X}_{t_k}^n$ of $\bar X_{t_k}^n$, for every $k \in \{1, \ldots,\ldots n\}$, their distributions and the corresponding transition weight matrices. Moreover, we will explain how to compute numerically the solution of the discretized scheme $(\ref{YchapTint})$-$(\ref{Zetachapkint})$-$(\ref{Ychapkint})$ associated to the RBSDE $(\ref{BSDE})$. These computations will be useful to carry on numerical tests and experiments illustrating the above error bounds.  One of the most important applications of these quantization-based discretizations is the pricing of American options for which the driver $f$ is equal to $0$, among other examples (with a non-zero driver) that will be presented at the end of this paper. This link between BSDEs and the pricing of financial options have been first introduced in $\cite{KarPeQue97b}$. \\

Throughout this paper, we will replace, for convenience, the indices $t_k$ by $k$ for $k \in \{0,\ldots,n\}$, i.e. we will use, for example, $\widehat{X}_{k}$ instead of $\widehat{X}_{t_k}$. Also, we will replace $f(t_k,x,y,z)$ by $f_k(x,y,z)$, $b(t_k,.)$ by $b_k(.)$ and $\sigma(t_k,.)=\sigma_k(.)$. And, we will omit the $n$ in $\bar{Y}^n_{k+1}, \bar X^n_{k+1},$ etc.\\

This paper is organized as follows: In section $\ref{QR}$, we provide some short background on recursive quantization and establish the new $L^p$-error bounds for $p\in (1,2+d)$, of the recursive quantization error as well as those of the {\em hybrid} recursive quantization error. Section $\ref{timedisc}$ is devoted to the time discretization of the RBSDE and to the estimation of the corresponding error. The space disretization of the RBSDE will be treated in Section $\ref{spacedisc}$. In Section $\ref{algorithmics}$, we will present the numerical techniques to compute the recursive quantizers and the solution of the RBSDE. Finally, Section $\ref{examples}$ is devoted to several numerical examples. 
\section{Recursive Quantization: background, $L^p$-error bounds and hybrid schemes.}
\label{QR}
In this section, we study the discretization of the forward process $(X_t)_{t \geq 0}$. It is a Brownian diffusion process taking values in $\R^d$, solution to the SDE $(\ref{SDE})$ given in the introduction and recalled below
\begin{equation*}
X_t=X_0+\int_0^t b(s,X_s)ds +\int_0^t \sigma(s,X_s)dW_s\, , \qquad X_0=x_0 \in \R^d.
\end{equation*}
First, we start by the time discretization and we present the Euler scheme $(\bar{X}_{t_k})_{0 \leq k \leq n}$, with uniform mesh $t_k=k\Delta$ for $k \in \{0,\ldots,n\}$ and $\Delta=\frac{T}{n}$, associated to the process $(X_t)_{t \in [0,T]}$ which is recursively given by
\begin{equation}
\label{RBSDE:SDEeuler}
\bar{X}_{t_{k+1}}=\bar{X}_{t_k}+\Delta b_k(\bar{X}_{t_k})+\sigma_k(\bar{X}_{t_k})( W_{t_{k+1}}-W_{t_k}), \qquad \bar{X}_0=X_0=x_0, 
\end{equation}
where $W_{t_{k+1}}-W_{t_k}=\sqrt{\Delta} \ve_{k+1}$, for every $k \in \{0,\ldots,n-1\}$ and $(\ve_k)_{0 \leq k \leq n}$ is a sequence of i.i.d. random variables with distribution $\mathcal{N}(0,I_q)$. Its continuous counterpart, the {\it genuine Euler scheme}, is given by 
\begin{equation}
\label{genuineEuler}
d\bar{X}_t=b(\underline{t},\bar{X}_{\underline{t}})dt+\sigma(\underline{t},\bar{X}_{\underline{t}})dW_t
\end{equation} 
where $\underline{t}=t_k$ when $t \in [t_k,t_{k+1})$. This process satisfies for every $p \in (0,+\infty)$ and every $n \geq 1$, (see \cite{BoLe93})
$$\Big\|\sup_{t\in [0,T]}X_t\Big\|_p+\sup_{n \geq 1} \Big\|\sup_{t\in [0,T]}\bar{X}_t\Big\|_p \leq C_{b,T,\sigma}(1+|x_0|) \qquad \mbox{and} \qquad \Big\|\sup_{t \in [0,T]} |X_t-\bar{X}_t|\Big\|_p \leq C_{b,T,\sigma}\, \sqrt{\Delta} (1+|x_0|)$$
where $C_{b,T,\sigma}$ is a positive constant depending on $p,T,b$ and $\sigma$.\\  

After the time discretization, one must proceed with space discretization schemes. As introduced, we consider in this paper the approximation of the Euler scheme $(\bar{X}_{t_k})_{0 \leq k \leq n}$ by recursive quantization.% and we establish $L^p$-error bounds induced by this type of discretization for $p \in  (1,2+d)$, these upper estimates are new when we consider $p \neq 2$, the case $p=2$ was handled in $\cite{PaSa14}$. 
\subsection{Background}
Our aim is to design, for $k \!\in \{0,\ldots,n\}$, optimal quantizers $\Gamma_k$ of size $N_k$ of a function of the discrete Euler scheme $(\bar{X}_k)_{0 \leq k \leq n}$. In other words, we want to find the grid $\Gamma_k$ that minimizes the $L^p$-distortion function $G_k^p(\Gamma)=\E\Big[\mbox{dist}\big(\mathcal{E}_{k-1}(\bar{X}_{k-1},\ve_{k}),\Gamma \big)^p\Big]$ corresponding to $\mathcal{E}_{k-1}(\bar{X}_{k-1},\ve_{k})$ 
where $$\mathcal{E}_{k-1}(x,\ve_{k})=x+\Delta b_k(x) +\sqrt{\Delta} \sigma_k(x) \ve_{k}$$
and $(\ve_k)_k$ is an i.i.d. sequence of $\mathcal{N}(0,I_q)$-distributed random vectors independent from $X_0$.\\

Since $\bar{X}_0=X_0=x_0$ is fixed, its quantizer is given by $\Gamma_0=\{x_0\}$. Then, we compute $\widetilde{X}_1=\mathcal{E}_0(\widehat{X}_0^{\Gamma_0},\ve_1)$ and we build an optimal quantization grid $\Gamma_1$ of size $N_1$ that minimizes $G^p_1(\widetilde{X}_1,\Gamma)$ on the set of grids $\Gamma$ of size $N_1$ (see Section $\ref{algorithmics}$). Doing so, we are able to define the quantization of $\bar{X}_1$ by  $\widehat{X}_1^{\Gamma_1}=\mbox{Proj}_{\Gamma_1}(\widetilde{X}_1)$. Repeating this procedure, we define a(n optimized) recursive quantization of $(\bar{X}_k)_{0 \leq k \leq n }$ by the following recursion: $\widehat{X}_0=\bar{X}_0=x_0$ and
\begin{equation}
\label{RBSDE:quantifrecursive}
\left\{
\begin{array}{rl}
\widetilde{X}_{k} &=\mathcal{E}_{k-1}(\widehat{X}^{\Gamma_{k-1}}_{k-1},\ve_{k}),  \\
\smallskip
\widehat{X}_{k}^{\Gamma_{k}} & =\mbox{Proj}_{\Gamma_k}(\widetilde{X}_k), \qquad \quad \forall k= 1, \ldots, n.
\end{array}
\right .
\end{equation}
In practice, we ask the grids $\Gamma_k$ to share some optimality properties, typically to be $L^p$-optimal or in higher dimension to be a product grid with optimal marginals, etc. 
For that purpose, the following identities play a crucial role: the $L^p$-distortion function associated to $\Gamma_k=(x_1^k, \ldots, x_{N_k}^k)$ is approximated by
\begin{equation}
\label{distfunction}
G_k^p(x_1^k, \ldots, x_{N_k}^k)=\E[\mbox{dist}(\widetilde{X}_k,\{x_1^k, \ldots, x_{N_k}^k\})^p]=\sum_{i=1}^{N_k} \E[\mbox{dist} (\mathcal{E}_{k-1}(x_i^{k-1},\ve_{k}),x_i^k)^p] \P\big(\widehat{X}_k^{\Gamma_k}\in C_i(\Gamma_k)\big)
\end{equation}
where $\P\big(\widehat{X}_k^{\Gamma_k}\in C_i(\Gamma_k)\big)$ is the weight of the Vorono\"i cell of centroid  $x_i^k \in \Gamma_k$. Note that one can write the distortion function as a function of the grid $\Gamma_k$ but writing it as a function of an $N_k$-tuple is needed in order to talk about its differentiability. In fact, if the $N_k$-tuple $(x_1^k, \ldots, x^k_{N_k})$ has pairwise distinct components and the boundaries of the Vorono\"i diagram $\big(\partial C_i(\Gamma_k)\big)_{1 \leq i \leq N_k}$ are negligible w.r.t. the distribution of $\widetilde X_{k}$, then the gradient of the differentiable $L^p$-distortion function is given by 
$$\nabla G_k^p(x_1^k, \ldots, x_{N_k}^k)=p \Big(\E\big[ \mathds{1}_{ \widetilde X_k\in C_i(\Gamma_k)}(x_i^k-\widetilde X_k)^{p-1}\big] \Big)_{1 \leq i \leq N_k}.$$
Note that since the grid $\Gamma_k$ has pairwise distinct components for every $k \in \{0,\ldots,n\}$, the distribution of $\widetilde X_k$ exists as soon as $\sigma \sigma^*$ is invertible. 
From now on, we denote $\widehat{X}_k$ instead of $\widehat{X}_k^{\Gamma_k}$ for simplicity.
\subsection{$L^p$-error bounds for recursive quantization }
Our aim is to establish $L^p$-upper bounds for the recursive quantization error $\|\bar X_{t_k} -\widehat X_{t_k}\|_p$ for $p \in (1,2+d)$ and $k \in \{0,\ldots,n\}$. As explained, the recursive quantization schemes of $\bar X_{t_k}$ are based on optimal quantization sequences of $\widetilde X_{t_k}$ which can be either quadratic or $L^p$-quantization sequences, $p\neq 2$. The more interesting case is when we rely on $L^2$-optimal quantization because, from an algorithmic point of view, one has direct access to optimal quadratic quantizers since they are stationary and the algorithms used to produce optimal quantizers are either directly based on the stationarity property or easier to manage in a quadratic framework. Nevertheless, establishing an upper bound for the error $\|\bar{X}_{t_k}-\widehat{X}_{t_k}\|_p$ where $\widehat{X}_{t_k}$ is itself an $L^p$-optimal quantizer of $\widetilde{X}_{t_k}$ still seems a natural track to consider.\bigskip\\
%As explained previously, to obtain the recursive quantizer of $\bar{X}_k$, one needs an optimal quantizer of $\widetilde{X}_k$. And, even if we are interested in computing an $L^p$-error bound, this optimal quantizer can be a quadratic or an $L^p$-quantizer for $p\neq 2$ (in higher dimensions, we consider product grids of optimal quantizers of the marginal distributions of each $\widetilde X_k$). In fact, the two situations that may occur are the following. On one hand, we compute the error bound $\|\bar{X}_k-\widehat{X}_k\|_p$ where $\widehat{X}_k$ is a quadratic quantizer of $\widetilde{X}_k$ for every $k \in \{1,\ldots,n\}$, because, from an algorithmic point of view, one has direct access to optimal quadratic quantizers since they are stationary and the algorithms used to produce optimal quantizers are either directly based on the stationarity property or easier to manage in a quadratic framework. But these quantizers are sometimes
%used in other settings from which the interest of computing the corresponding $L^p$-error. The second approach is to find a bound for the error $\|\bar{X}_k-\widehat{X}_k\|_p$ where $\widehat{X}_k$ is itself an $L^p$-optimal quantizer of $\widetilde{X}_k$.\\ 
%. We will see that this latter approach will not give any interesting results.\\
\noindent {\bf $L^2$-optimal quantization}\smallskip \\
We consider the case where, for every $k \in \{1, \ldots, n\}$, $\widehat{X}_{t_k}$ is a quadratic optimal quantization of $\widetilde{X}_{t_k}$, hence it is stationary in the sense of $(\ref{RBSDE:stationary})$ (see \cite{Pages18}). In the following, we assume that $\Delta \in [0,\Delta_{\max}), \, \Delta_{\max}>0$. Note that for the Euler scheme, one can have $\Delta_{\max}=\frac{T}{n_0}$ if we consider schemes with step $\Delta=\frac{T}{n}$ and a number of steps $n>n_0$ for some $n_0>0$. %The recursive quantization error is bounded, for every $k \in \{1,\ldots,n\}$, as stated in the following Theorem.
\begin{thm}
	\label{erreurXQR}
	Let $p \in (1,2+d)$, $(\bar X_{k})_{0 \leq k \leq n}$ defined by $(\ref{RBSDE:SDEeuler})$ and $(\widehat{X}_{k})_{0 \leq k \leq n}$ the corresponding recursive quantization sequence defined by $(\ref{RBSDE:quantifrecursive})$. Assume that, for every $k \in \{0,\ldots,n\}$, $\widehat X_k$ is a stationary quadratic optimal quantization of $\widetilde{X}_k$ of size $N_k$ in the sense of $(\ref{RBSDE:stationary})$, with $\widehat{X}_0=\bar{X}_0=x_0 \in \R^d$.
	For every $k\in \{1, \ldots, n\}$ and every $\delta \in (0,1]$, 
	$$\|\bar{X}_{k}-\widehat{X}_{k}\|_p \leq \big(\widetilde{K}_{d,2,2+\delta,p} \vee \kappa_{d,2,\delta}\big) \sum_{l=1}^{k} [\mathcal{E}_k]_{\rm Lip}^{k-l} C_{2+\delta,b,\sigma,T}(l)^{\frac{1}{2+\delta}} N_l^{-\frac{1}{d}} $$
	where	$\kappa_{d,2,\delta}$ is the constant from Pierce's Lemma $\ref{Zadoretpierce}(b)$, 
	$$  \widetilde{K}_{d,2,2+\delta,p} \leq 2^{\frac{p(2+\delta)}{(2+d)^2-dp}} V_d^{-\frac{1}{2+d}}\kappa_{X,2}^{\frac{1}{2+d}}\min_{\varepsilon \in (0,\tfrac13)}\Big[(1+\ve)\varphi_2(\varepsilon)^{-\frac{1}{d+2}}\Big] \left(\int_{\R^d}\big(1 \vee \|x\| \big)^{-\frac{(d+2-p)(2+\delta)}{p}}dx \right)^{\frac{1}{2+d}}$$
	 with $\kappa_{X,r}$ a finite positive constant independent from $N_l$, $V_d$ the volume of the hyper-unit ball and $\varphi_2(u)=\big(\frac{1}{3^2}-u^2\big)u^d$,
	$$[\mathcal{E}_k]_{\rm Lip}=
	\left\{
	\begin{array}{ll}
	e^{\Delta \big(s[b]_{\rm Lip}+c_s^{(1)}+c^{(3)}_{s,\Delta_{\max},\ve_{k+1}}  [\sigma]_{\rm Lip}^s\big)/p} & \mbox{ if } \; p \in (1,2)\\
	e^{\Delta \big(p[b]_{\rm Lip}+c_p^{(1)}+ c^{(3)}_{p,\Delta_{\max},\ve_{k+1}}[\sigma]_{\rm Lip}^p\big)/p}& \mbox{ if }\; p \in [2,2+d)
	\end{array}	
	\right.
	$$
	with $s=p+1>2$, $c_p^{(1)}=2^{(p-3)_+}\frac{(p-1)(p-2)}{2}$ and $c^{(3)}_{p,\Delta_{\max},\ve_{k+1}}=2^{(p-3)_+}(p-1) \E|\ve_{k+1}|^p \big(1 +\frac{p}{2} \Delta_{\max}^{\frac{p}{2}-1}\big)$
	and $$ C_{2+\delta,b,\sigma,T}(l)= e^{t_k (C_1+C_2)} |x_0|^{2+\delta}+ \frac{C_3}{C_1+C_2} \left(e^{t_{k-1} (C_1+C_2)}-1  \right)$$ where $C_1,C_2$ and $C_3$ are defined in Lemma $\ref{Xtildequad}$.
\end{thm}
Before sharing the proof, we need to present some a priori useful results, mainly the distortion mismatch problem and two lemmas. We reconsider the notations where we replace the indices $t_k$ by $k$ to alleviate notations. \bigskip\\
\noindent {\sc  {$(L^r,L^s)$-problem or distortion mismatch problem}}\smallskip\\
Let $r,s \in (0,+\infty)$, the $(L^r,L^s)$-problem, also called distortion mismatch problem, consists in determining whether the optimal rate of $L^r$-optimal quantizers holds for $L^s$-quantizers for $s \neq r$, i.e. whether an $L^r$-optimal quantizer $\Gamma_N$ of size $N$ of a random vector $X$ has an $L^s$-optimal convergence rate for $s\neq r$. For $s <r$, it is clear that an $L^r$-optimal quantizer is $L^s$-rate optimal due to the monotony of $r \rightarrow \|.\|_r$. When $s$ becomes higher than $r$, we do not have such direct results. This problem was first introduced and treated in \cite{mismatch08,GraLu00} for radial density distributions on $\R^d$ and then generalized in \cite{PaSa18} for all random vectors satisfying a certain moment condition. In the following theorem, we sum up this result and give a universal non-asymptotic Pierce type optimality result (in the sense of $(\ref{Pierce})$).
%, contrarily to the asymptotic Zador type results (in the sense of $(\ref{Zador})$) already obtained in the literature. \bigskip \\
%\noindent {\sc {$(L^r,L^s)$-problem or distortion mismatch problem}}\medskip\\
%Let $r,s \in (0,+\infty)$. The $(L^r,L^s)$-problem or distortion mismatch problem consists in determining whether an $L^r$-optimal quantizer $\Gamma_N$ of size $N$ of a random vector $X$ has an $L^s$-optimal convergence rate for $s\neq r$. In other words, determining the conditions for which a sequence of $L^r$-optimal quantizers satisfies 
%%the limit given by Zador Theorem $(\ref{Zadoretpierce})$
%%\begin{equation}
%%\label{zador}
%%\limsup_{N} N^{\frac{1}{d}}e_q(\Gamma_N,X)< +\infty,
%%\end{equation}
%%or,
%a kind of ($L^r$-$L^s$) Pierce's Lemma, i.e. if there exists some real constant $K$ such that
%\begin{equation}
%\label{pierce}
%\|X-\mbox{Proj}_{\Gamma_N}(X)\|_s \leq K \, \sigma_{r+\delta}(X) N^{-\frac{1}{d}}
%\end{equation} 
%where $\sigma_{r}(X)=\inf_{a \in \R^d}\|X-a\|_{r}$ is the $L^r$-standard deviation of $X$.\medskip \\ 
%For $s <r$, it is clear that an $L^r$-optimal quantizer is $L^s$-rate optimal due to the monotony of $r \rightarrow \|.\|_r$. When $s$ becomes higher than $r$, we do not have such direct results. After dealing, in \cite{mismatch08}, with radial density distributions, a more general result was given in \cite{PaSa18} and summed up in the following Theorem.
\begin{thm}[Extended Pierce's Lemma]
\label{distmismatch}
$(a)$ Let $r>0$ and $X$ be an $\R^d$-valued random vector such that $\E |X|^{r'}<+\infty$ for some $r'>r$. Assume that its distribution $\P_X$ has a non-zero absolutely continuous component and let $(\Gamma_N)_{N \geq 1}$ be a sequence of $L^r$-optimal quantizers of $X$. Then, for every $s\!\in \Big(0, \frac{(d+r)r'}{d+ r'}\Big)$,
\begin{equation}
\label{piercemismatch} e_s(\widehat{X}^{\Gamma_N},X)\leq \widetilde{K}_{d,r,r',s}\, \sigma_{r'}(X)\, N^{-\frac{1}{d}}
\end{equation}
	where $\sigma_{r'}(X)=\inf_{a \in \R^d}\|X-a\|_{r'}$ is the $L^{r'}$-standard deviation of $X$ and $$ \widetilde{K}_{d,r,r',s}\leq \, 2^{\frac{sr'}{(r+d)^2-ds}} V_d^{-\frac{1}{r+d}}\kappa_{X,r}^{\frac{1}{r+d}}\min_{\varepsilon \in (0,\tfrac13)}\Psi_r(\varepsilon) \left(\int\big(1 \vee \|x\| \big)^{-\frac{(d+r-s)r'}{s}}dx \right)^{\frac{1}{r+d}}$$ with $\kappa_{X,r}$ a finite positive constant independent from $N$, $V_d$ the volume of the hyper-unit ball and
	$ \Psi_r(u)=(1+u)\left(\frac{1}{3^r}-u^r\right)^{-\frac{1}{d+r}}u^{-\frac{d}{d+r}}$.\medskip\\
	$(b)$ In particular if $X$ has finite polynomial moments at any order, then $(\ref{piercemismatch})$ is satisfied for every $s \in (r,d+r)$ and $r' > \frac{sd}{d+r-s}$.
\end{thm}
The following lemma is a technical one used repeatedly in the proofs in this paper. Its proof will be postponed to the appendix.
\iffalse
\begin{lem}
	\label{lemme1}
	Let $A \in \mathcal{M}(d,q,\R)$, $a \in \R^d$ and $r \geq 2$. For every random vector $Z$ having values in $\R^q$ such that $\E[Z]=0$ and every $\Delta>0$, 
	\begin{equation}
	\label{lem1}
	\E|a+A\sqrt{\Delta}Z|^r \leq |a|^r\left( 1+2^{(r-3)_+}(r-1)(r-2)\Delta \right)+2^{(r-3)_+}(r-1) \Delta\|A\|^r \E|Z|^r \left(1 +\frac{r}{2} \Delta^{\frac{r}{2}-1} \right)
	\end{equation}
	where $\|A\|^2={\rm Tr}(AA^{*})$ denotes the (squared) Fr\"obenius norm.
\end{lem}
\fi
 \begin{lem}\label{lemme1}
 %[see~\cite{Nmeir}]
  Let $r\!\in [2, +\infty)$ and $h_0>0$. Let $Z\!\in L^r_{\R^q}(\mathbb{P})$ with $\E\, Z=0$ and let $a\!\in \R^d$, $A\!\in {\cal M}(d,q,\R)$. Then for every  $h\!\in (0, h_0)$,
 \begin{equation}\label{lem1}
 \E\, \big| a+\sqrt{h}AZ \big|^r\le |a|^r (1+c^{(1)}_rh ) + c^{(2)}_{r,h_0}\, h \|A\|^r \E\, |Z|^r
  \end{equation}
where $c^{(1)}_r = 2^{(r-3)_+} \frac {(r-1)(r-2)}{2}$, $c^{(2)}_{r,h_0} = 2^{(r-3)_+}(r-1) \big( 1+ \tfrac r2 h_0^{\frac r2-1} \big)$ and $\|A\| $ is the operator norm.
 \end{lem} 
The following lemma is important for the proof of Theorem \ref{erreurXQR}.
\begin{lem}
	\label{Xtildequad} 
	Consider $(\bar X_k)_{0 \leq k \leq n}$ defined by $(\ref{RBSDE:SDEeuler})$ and $(\widehat{X}_k)_{0 \leq k \leq n}$ its recursive quantization sequence defined by $(\ref{RBSDE:quantifrecursive})$. Assume that, for every $k \in \{0,\ldots,n\}$, $\widehat X_k$ is a stationary quadratic optimal quantization of $\widetilde{X}_k$ of size $N_k$ in the sense of $(\ref{RBSDE:stationary})$, with $\widehat{X}_0=\bar{X}_0=x_0 \in \R^d$. For every $r\geq2$ and every $k \in \{1, \ldots, n\}$, 
	\begin{equation}
	\label{momentXtilde}
	\E|\widetilde{X}_k|^r \leq e^{t_k (C_1+C_2)} |x_0|^r+ \frac{C_3}{C_1+C_2} \left(e^{t_{k-1} (C_1+C_2)}-1  \right).
	\end{equation}
	where \\
	 $C_1=rL_{b,\sigma}+(r-1)2^{r-2}+c_r^{(1)},\quad$ 
	$C_2=2^{r-1}L_{b,\sigma}^r \E|Z|^r\Delta_{\max}^{r} c_{r,\Delta_{\max}}^{(2)}:=L_{b,\sigma}^r2^{r-1} \Delta_{\max}^r c^{(3)}_{r,\Delta_{\max},Z}$ 
	and $C_3= C_2+2^{r-2}L_{b,\sigma}^r(1+r\Delta_{\max}^{r-1})(1+c_r^{(1)}\Delta_{\max})$ with $c_r^{(1)}$ and $c^{(2)}_{r,\Delta_{\max}}$ defined in Lemma \ref{lemme1}.
\end{lem}
\begin{proof}
	The starting point is to use inequality $(\ref{lem1})$ with $a=x+\Delta b(t,x)$ and $A=\sigma(t,x)$. On the one hand, we notice that
	$$|a|  \leq |x|+\Delta L_{b,\sigma}(1+|x|) \leq |x|(1+\Delta L_{b,\sigma}) +\Delta L_{b,\sigma}.$$
	Then, using the fact that, for every $\varepsilon >0$,
	\begin{align}
	\label{a+bexpor}
	(\alpha+\beta)^r & \leq \alpha^r + r \beta (\alpha+\beta)^{r-1} \nonumber \\ 
	& \leq \alpha^r +r 2^{r-2} \left( (\varepsilon \alpha)^{r-1} \frac{\beta}{\varepsilon^{r-1}}+\beta^r \right)  \nonumber \\ 
	& \leq \alpha^r +r 2^{r-2} \left(  \beta^r + \frac{\varepsilon^r \alpha^r (r-1)}{r} +\frac{\beta^r}{r \varepsilon^{r(r-1)}}\right) \quad (\mbox{Young's inequality with } \tfrac{r}{r-1} \mbox{ and } r) \nonumber \\ 
	& \leq \alpha^r \Big(1+(r-1) 2^{r-2} \varepsilon^r \Big) +2^{r-2} \beta^r \Big( r +\frac{1}{\varepsilon^{r(r-1)}}\Big),
	\end{align}
	one has, by considering $\alpha=|x|(1+\Delta L_{b,\sigma})
	$ and $\beta=\Delta L_{b,\sigma}$, that
	$$|a|^r \leq |x|^r (1+\Delta L_{b,\sigma})^r \left(1+(r-1) 2^{r-2} \varepsilon^r \right) +2^{r-2} \Delta^r L_{b,\sigma}^r \left( r +\frac{1}{\varepsilon^{r(r-1)}}\right).$$
	On the other hand,  
	\begin{equation*}
	\|A\|  \leq \Delta L_{b,\sigma}(1+|x|) \quad \mbox{ so that } \quad
	\|A\|^r  \leq 2^{r-1} \Delta^r L_{b,\sigma}^r (1+|x|^r).
	\end{equation*}
	Consequently, Lemma $\ref{lemme1}$ yields
	\begin{align*}
	\E|a+A \sqrt{\Delta}Z|^r \leq & |x|^r(1+\Delta L_{b,\sigma})^r\left(1+(r-1) 2^{r-2} \varepsilon^r \right)\left(1+c_r^{(1)}\Delta\right) +L_{b,\sigma}^r 2^{r-1} \E|Z|^r \Delta^{r+1}  c_{r,\Delta_{\max}}^{(2)}|x|^r\\
	&+\left(1+c_r^{(1)}\Delta\right)2^{r-2}L_{b,\sigma}^r\Delta^r \left( r +\frac{1}{\varepsilon^{r(r-1)}}\right) +L_{b,\sigma}^r 2^{r-1} \E|Z|^r \Delta^{r+1}  c_{r,\Delta_{\max}}^{(2)} .
	\end{align*}
	
	At this stage, we are interested in considering a particular value of $\varepsilon$ to avoid any explosion at infinity in the rest of the proof. The best choice (up to a multiplicative constant) is $$\varepsilon =\Delta^{\frac{1}{r}}.$$
	Now, we recall that $\Delta \in [0,\Delta_{\max}), \, \Delta_{\max}>0$ and denote 
\begin{align*}
	C_1:=&C_1(r)=rL_{b,\sigma}+(r-1)2^{r-2}+c_r^{(1)}\\
	C_2:=&C_2(r,L_{b,\sigma},Z,\Delta_{\max})=2^{r-1}L_{b,\sigma}^r \E|Z|^r\Delta_{\max}^{r} c_{r,\Delta_{\max}}^{(2)}:=L_{b,\sigma}^r2^{r-1} \Delta_{\max}^r c^{(3)}_{r,\Delta_{\max},Z}\\
	C_3:=&C_3(r,Z,L_{b,\sigma},\Delta_{\max})= C_2+2^{r-2}L_{b,\sigma}^r(1+r\Delta_{\max}^{r-1})(1+c_r^{(1)}\Delta_{\max})
	\end{align*}
	Having $1+x \leq e^x$ yields
	\begin{align*}
	\E|a+A \sqrt{\Delta}Z|^r \leq  |x|^r e^{C_1\Delta}+\Delta \big( C_2|x|^r+C_3 \big)
	\leq   |x|^r e^{C_1\Delta} \big(1+\Delta C_2 e^{-\Delta C_1} \big) +\Delta C_3
	\leq  e^{\Delta (C_1+C_2)} |x|^r + \Delta C_3.
	\end{align*}
	Thus, since $\E|\widetilde{X}_k|^r =\E|\mathcal{E}_{k-1}(\widehat{X}_{k-1},\ve_k)|^r$, one can write
	$$\E|\widetilde{X}_k|^r \leq   e^{\Delta (C_1+C_2)} \E|\widehat{X}_{k-1}|^r + \Delta C_3. $$
	Using the fact that $\widehat{X}_{k-1}$ is a stationary quadratic optimal quantization of $\widetilde{X}_{k-1}$ and Jensen inequality yield
	$$\E|\widehat{X}_{k-1}|^r = \E|\E(\widetilde{X}_{k-1}|\widehat{X}_{k-1})|^r \leq  \E\big[\E\big(|\widetilde{X}_{k-1}|^r|\widehat{X}_{k-1}\big)\big] \leq \E|\widetilde{X}_{k-1}|^r.$$
	Therefore,
	$$\E|\widetilde{X}_k|^r \leq   e^{\Delta (C_1+C_2)} \E|\widetilde{X}_{k-1}|^r + \Delta C_3. $$
	Finally, it follows by induction that
	\begin{align*}
	\E|\widetilde{X}_k|^r  \leq & e^{k\Delta (C_1+C_2)} \E|\widetilde{X}_0|^r +\Delta C_3 \sum_{j=0}^{k-1}e^{j\Delta (C_1+C_2)}\\
	\leq &  e^{k\Delta (C_1+C_2)} |x_0|^r +\Delta C_3 \frac{e^{(k-1)\Delta (C_1+C_2)}-1}{e^{\Delta (C_1+C_2)}-1}\\
	\leq &  e^{k\Delta (C_1+C_2)} |x_0|^r+ \frac{C_3}{C_1+C_2} \left(e^{(k-1)\Delta (C_1+C_2)}-1  \right).
	\end{align*}
	The result is obtained by noting that $k\Delta=k\frac{T}{n}=t_k$.
	\hfill $\square$
\end{proof}
\bigskip
\begin{refprooferreurXQR}
The first step of the proof is to show that the function $\mathcal{E}_k(.,\ve_{k+1})$ is $L^p$-lipschitz continuous with Lipschitz coefficient $[\mathcal{E}_k]_{\rm Lip}$ for every $k \in \{0,\ldots,n-1\}$. We consider two cases depending on the values of $p$. \\
$\bullet$ If $p \in [2,2+d)$: For every $x,x' \in \R^d$,
		$$\E\big|\mathcal{E}_k(x,\ve_{k+1})-\mathcal{E}_k(x',\ve_{k+1})\big|^p = \E\big|x-x' +\Delta\big(b_k(x)-b_k(x')\big)+\sqrt{\Delta}\ve_{k+1}\big(\sigma_k(x)-\sigma_k(x')\big)\big|^p.$$
		Since $p \geq 2$, one applies Lemma \ref{lemme1} with $a=x-x' +\Delta\big(b_k(x)-b_k(x')\big)$ and $A=\sigma_k(x)-\sigma_k(x')$. We have
		$$
		|a|^p \leq \big(|x-x'|+\Delta [b]_{\rm Lip}|x-x'| \big)^p \leq |x-x'|^p\big(1+\Delta [b]_{\rm Lip}\big)^p \leq |x-x'|^p \,e^{p\Delta[b]_{\rm Lip}}
		$$
		and
		$$\|A\|^p \leq [\sigma]_{\rm Lip}^p |x-x'|^p.$$
		At this stage, reusing the constants $c_p^{(1)}=2^{(p-3)_+} \frac {(p-1)(p-2)}{2}$ and $c^{(3)}_{p,\Delta_{\max},\ve_{k+1}}=2^{(p-3)_+}(p-1)\E |\ve_{k+1}|^p \big( 1+ \tfrac p2 \Delta_{\max}^{\frac p2-1} \big)$ defined in Lemmas $\ref{lemme1}$ and $\ref{Xtildequad}$ yields
		\begin{align*}
		\E|\mathcal{E}_k(x,\ve_{k+1})-\mathcal{E}_k(x',\ve_{k+1})|^p & \leq \left(e^{\Delta (p[b]_{\rm Lip}+c_p^{(1)})} + \Delta [\sigma]_{\rm Lip}^p c^{(3)}_{p,\Delta_{\max},\ve_{k+1}} \right)|x-x'|^p\\
		& \leq |x-x'|^pe^{\Delta (p[b]_{\rm Lip}+c_p^{(1)})} \Big(1+\Delta [\sigma]_{\rm Lip}^p c^{(3)}_{p,\Delta_{\max},\ve_{k+1}}e^{-\Delta (p[b]_{\rm Lip}+c_p^{(1)})}\Big)\\
		& \leq |x-x'|^pe^{\Delta (p[b]_{\rm Lip}+c_p^{(1)})} \Big(1+\Delta [\sigma]_{\rm Lip}^p c^{(3)}_{p,\Delta_{\max},\ve_{k+1}}\Big)\\
		& \leq |x-x'|^pe^{\Delta \big(p[b]_{\rm Lip}+c_p^{(1)}+[\sigma]_{\rm Lip}^p c^{(3)}_{p,\Delta_{\max},\ve_{k+1}}\big)} .
		\end{align*}
		Consequently, $\mathcal{E}_k$ is $L^p$-lipschitz continuous with $[\mathcal{E}_k]_{\rm Lip}=e^{\Delta \big(p[b]_{\rm Lip}+c_p^{(1)}+[\sigma]_{\rm Lip}^p c^{(3)}_{p,\Delta_{\max},\ve_{k+1}}\big)/p}$, for every $k \in \{1,\ldots,n\}$ and $p\in [2,2+d)$.\medskip \\
		\noindent $\bullet$ If $1 <p<2$: Consider $s=p+1>2$ so that $p-s<0$. One has
	$$\E|\mathcal{E}_k(x,\ve_{k+1})-\mathcal{E}_k(x',\ve_{k+1})|^p=\E \left[|\mathcal{E}_k(x,\ve_{k+1})-\mathcal{E}_k(x',\ve_{k+1})|^s\, |\mathcal{E}_k(x,\ve_{k+1})-\mathcal{E}_k(x',\ve_{k+1})|^{p-s} \right].$$
On the one hand, 
\begin{align*}
|\mathcal{E}_k(x,\ve_{k+1})-\mathcal{E}_k(x',\ve_{k+1})|^{p-s} & \leq |x-x'|^{p-s}\left(1+\Delta [b]_{\rm Lip}+\sqrt{\Delta}\frac{|\sigma(x)-\sigma(x')|}{|x-x'|}|\ve_{k+1}| \right)^{p-s}\\
& \leq |x-x'|^{p-s}e^{(p-s)\left(1+\Delta [b]_{\rm Lip}+\sqrt{\Delta}\frac{|\sigma(x)-\sigma(x')|}{|x-x'|}|\ve_{k+1}|\right) } \quad (\mbox{since } 1+x \leq e^x)\\
& \leq |x-x'|^{p-s} \qquad (\mbox{ since } p-s<0).
\end{align*}
On the other hand, one uses inequality $(\ref{inegalitedulemme})$ from the proof of Lemma \ref{lemme1} (see Appendix) and denotes $a=x-x'+\Delta [b]_{\rm Lip}(x-x')$ and $AZ=(\sigma(x)-\sigma(x'))\ve_{k+1}$, to obtain
\begin{align*}
|\mathcal{E}_k(x,\ve_{k+1})-\mathcal{E}_k(x',\ve_{k+1})|^{s}  \leq &|x-x'+\Delta [b]_{\rm Lip}(x-x')+\sqrt{\Delta}(\sigma(x)-\sigma(x'))\ve_{k+1}|^{s}\\
\leq & |a|^s (1+\Delta c_s^{(1)})+ s\left(|a|^{s-1} \frac{a}{|a|} | A \sqrt{\Delta}Z \right)+ \Delta c^{(2)}_{s,\Delta_{\max}} |AZ|^s.
\end{align*}
At this stage, one notices that $|a|^s\leq |x-x'|^s(1+\Delta [b]_{\rm Lip})^s$ and that
$|AZ|^s  = [\sigma]_{\rm Lip}^s |x-x'|^s |\ve_{k+1}|^s$. Then, using $1+x \leq e^x$, one deduces
\begin{align*}
|\mathcal{E}_k(x,\ve_{k+1})-\mathcal{E}_k(x',\ve_{k+1})|^{s} \leq &|x-x'|^s (1+\Delta c_s^{(1)})(1+\Delta [b]_{\rm Lip})^s + s\left(|a|^{s-1} \frac{a}{|a|} | A \sqrt{\Delta}Z \right)\\
&+\Delta c^{(2)}_{s,\Delta_{\max}} [\sigma]_{\rm Lip}^s |x-x'|^s |\ve_{k+1}|^s \\
 \leq & |x-x'|^s e^{\Delta (c_s^{(1)}+s[b]_{\rm Lip})} + s\left(|a|^{s-1} \frac{a}{|a|} | A \sqrt{\Delta}Z \right)+\Delta c^{(2)}_{s,\Delta_{\max}} [\sigma]_{\rm Lip}^s |x-x'|^s |\ve_{k+1}|^s.
\end{align*}
Consequently, applying the expectation and keeping in mind that $\E|AZ|=0$, we obtain
\begin{align*}
\E|\mathcal{E}_k(x,\ve_{k+1})-\mathcal{E}_k(x',\ve_{k+1})|^p & \leq  \;e^{\Delta (c_s^{(1)}+s[b]_{\rm Lip})} |x-x'|^p+\Delta c^{(2)}_{s,\Delta_{\max}}  [\sigma]_{\rm Lip}^s |x-x'|^p \E|\ve_{k+1}|^s\\
&\leq  |x-x'|^p e^{\Delta (c_s^{(1)}+s[b]_{\rm Lip})} \Big(1+\Delta c^{(2)}_{s,\Delta_{\max}}  [\sigma]_{\rm Lip}^s \E|\ve_{k+1}|^s e^{-\Delta (c_s^{(1)}+s[b]_{\rm Lip})}\Big)\\&\leq  |x-x'|^p e^{\Delta (c_s^{(1)}+s[b]_{\rm Lip})} \Big(1+\Delta c^{(2)}_{s,\Delta_{\max}}  [\sigma]_{\rm Lip}^s \E|\ve_{k+1}|^s\Big)\\
& \leq  |x-x'|^p e^{\Delta \big(c_s^{(1)}+s[b]_{\rm Lip}+c^{(2)}_{s,\Delta_{\max}}  [\sigma]_{\rm Lip}^s \E|\ve_{k+1}|^s\big)}.
\end{align*}
		Consequently, $\mathcal{E}_k$ is $L^p$-Lipschitz continuous, for every $k \in \{1,\ldots,n\}$ and $p \in (1,2)$, with $[\mathcal{E}_k]_{\rm Lip}=e^{\Delta \big(c_s^{(1)}+s[b]_{\rm Lip}+c^{(2)}_{s,\Delta_{\max}}  [\sigma]_{\rm Lip}^s \E|\ve_{k+1}|^s\big)/p}$.\\
		
	\noindent For the second step, we first note that
	\begin{align*}
	\|\bar{X}_{k+1}-\widetilde{X}_{k+1}\|_p & = \|\mathcal{E}_k(\bar{X}_k,\ve_{k+1})-\mathcal{E}_k(\widehat{X}_k,\ve_{k+1})\|_p  \\
	&\leq [\mathcal{E}_k]_{\rm Lip} \|\bar{X}_k-\widehat{X}_k\|_p \\
	&\leq [\mathcal{E}_k]_{\rm Lip} \|\bar{X}_k-\widetilde{X}_k\|_p+ [\mathcal{E}_k]_{\rm Lip} \|\widetilde{X}_k-\widehat{X}_k\|_p.
	\end{align*} 
	Then, we show by induction, since $\widehat{X}_0=\widetilde{X}_0$, that
	$$\|\bar{X}_k-\widetilde{X}_k\|_p \leq \sum_{l=1}^{k-1} [\mathcal{E}_k]_{\rm Lip}^{k-l} \|\widetilde{X}_l-\widehat{X}_l\|_p.$$
	Consequently, 
	\begin{equation*}
	\|\bar{X}_k-\widehat{X}_k\|_p  \leq \|\bar{X}_k-\widetilde{X}_k\|_p+\|\widetilde{X}_k-\widehat{X}_k\|_p \leq  \sum_{l=1}^{k} [\mathcal{E}_k]_{\rm Lip}^{k-l} \|\widetilde{X}_l-\widehat{X}_l\|_p.
	\end{equation*}
	Now relying on the fact that $\widehat{X}_l$ is an $L^2$-optimal quantizer of $\widetilde{X}_l$ for every $l \in \{1, \ldots,k\}$, we distinguish two cases: one the one hand, if $p \in (1,2)$, we use the monotony of $p \mapsto \|\cdot\|_p$ and Pierce's Lemma $(\ref{Pierce})$ to conclude that, for every $l \in \{1,\ldots,k\}$, 
	$$\|\widetilde{X}_l-\widehat{X}_l\|_p \leq \|\widetilde{X}_l-\widehat{X}_l\|_2 \leq \kappa_{d,2,\delta} \|\widetilde{X}_l\|_{2+\delta} N_l^{-\frac{1}{d}},$$
	for some $\delta>0,$ and, on the other hand, if $p \in [2,2+d)$, we note that $\widetilde X_l=F_l(\widehat X_{l-1}, \ve_{l})$ has finite polynomial moments at any order since the innovations $(\ve_k)_{0 \le k \le n}$ in the Euler operators are with Gaussian distribution and hence have finite polynomial moments at any order, so one uses section $(b)$ of the distortion mismatch Theorem $\ref{distmismatch}$ to conclude that the quantization $\widehat{X}_l$ of $\widetilde X_l$ is $L^p$-rate optimal for every $p\in [2,2+d)$, in other words, we consider $\delta>0$ such that $r'=2+\delta > \frac{pd}{d+2-p}>2$ so that 
	$$\|\widetilde{X}_l-\widehat{X}_l\|_p \leq  \widetilde K_{d,2,2+\delta,p} \|\widetilde{X}_l\|_{2+\delta} N_l^{-\frac{1}{d}}.$$
	Hence, for every $p \in (1,2+d)$, 	
	$$\|\bar{X}_k-\widehat{X}_k\|_p \leq \big(\widetilde{K}_{d,2,2+\delta,p} \vee \kappa_{d,2,\delta}\big)  \sum_{l=1}^{k} [\mathcal{E}_k]_{\rm Lip}^{k-l} \|\widetilde{X}_l\|_{2+\delta} N_l^{-\frac{1}{d}}.$$
	The result is obtained by plugging $(\ref{momentXtilde})$ for $r=2+\delta>2$ in this last inequality.
	\hfill $\square$\\
\end{refprooferreurXQR}
\begin{rmq}
In higher dimensions, an approach to obtain the quantization grid of a multidimensional random variable is by taking the tensor product of one-dimensional quantization grids, that is the independent marginals of the distribution. The product quantization grid hence obtained by independent optimal one-dimensional quantizers is stationary and so this problem is solved in the multidimensional case. However, in most cases, the components of the diffusion $X_t$ are not independent so this is not a very useful technique in practice.
\end{rmq}
\begin{rmq}
We assume that $\widehat{X}_k$ is an $L^p$-optimal quantizer of $\widetilde{X}_k$ for every $k \in \{1,\ldots,n\}$. What differs from $L^2$-optimal quantizers is that $L^p$-optimal quantizers are not usually stationary, a property that was very useful in the quadratic case. The beginning of the study is exactly similar to the quadratic framework until we obtain 
$$\E|\widetilde{X}_k|^r \leq e^{\Delta (C_1+C_2)} \E|\widehat{X}_{k-1}|^r + \Delta C_3.$$
At this stage, one cannot use the stationarity property. Instead, applying inequality $(\ref{a+bexpor})$ yields 
\begin{align*}
\E|\widehat{X}_{k-1}|^r &  \leq \E\big(|\widehat{X}_{k-1}-\widetilde{X}_{k-1}|+|\widetilde{X}_{k-1}|\big)^r 
%& \leq \E|\widehat{X}_{k-1}-\widetilde{X}_{k-1}|^r (1+(r-1)2^{r-2}\varepsilon^r)+\E|\widetilde{X}_{k-1}|^r 2^{r-2} \Big(r+\frac{1}{\varepsilon^{r(r-1)}}\Big)\\
 \leq \E|\widehat{X}_{k-1}-\widetilde{X}_{k-1}|^r e^{C_4 \Delta } +\E|\widetilde{X}_{k-1}|^r 2^{r-2} \Big(r+\frac{1}{\Delta^{r-1}}\Big)
\end{align*}
where we took $\varepsilon=\Delta^{\frac1r}$ and denoted $C_4=(r-1)2^{r-2}$. Then, 
$$\E|\widetilde{X}_k|^r \leq \E|\widehat{X}_{k-1}-\widetilde{X}_{k-1}|^r e^{(C_1+C_2+C_4) \Delta } +e^{(C_1+C_2) \Delta}\E|\widetilde{X}_{k-1}|^r 2^{r-2} \Big(r+\frac{1}{\Delta^{r-1}}\Big)+\Delta C_3 $$
and an induction yields 
\begin{align*}
\E|\widetilde{X}_k|^r \leq & e^{k\Delta (C_1+C_2)} \left[2^{r-2} \Big(r+\frac{1}{\Delta^{r-1}}\Big) \right]^k \E|X_0|^r \\
&+ \sum_{i=0}^k e^{(k-i)\Delta (C_1+C_2)}\left(\E|\widehat{X}_{k-1}-\widetilde{X}_{k-1}|^r e^{(C_1+C_2+C_4) \Delta }+\Delta C_3 \right) \left[2^{r-2} \Big(r+\frac{1}{\Delta^{r-1}}\Big) \right]^k
\end{align*}
which clearly diverges as $n$ goes to infinity. The fact that it seems impossible to get rid of the factor $\frac{1}{\Delta}$, without the stationarity property, leads to conclude that we do not obtain satisfactory $L^p$-error bounds with a non-stationary $L^p$-optimal quantizer $\widehat{X}_k$ of $\widetilde{X}_k$. However, this is not really problematic since this is a very rare situation in practice because, as mentioned previously,  one usually uses quadratic optimal quantizers for numerical purposes. 
\end{rmq}
\subsection{Hybrid recursive quantization}
\label{hybridQR}
When the dimension becomes greater than $1$, computing the distribution (grids and transition matrices) of $(\widehat{X}_k)_{0 \leq k \leq n}$ via the recursive formulas $(\ref{RBSDE:quantifrecursive})$ cannot be achieved via closed formulas and deterministic optimization procedures. % as developed in $\cite{gilles a mis 11 ds ces corrections}$. 
Multi-dimensional extensions can be found in \cite{FiPaSa19} based on product quantization but this approach becomes computationally demanding when the dimension grows, an alternative being to implement a massive ''embedded'' Monte Carlo simulation.  We propose here a third approach based on the quantization of the white noise (here a Gaussian one). This quantization can be part of a pre-processing  and kept off line. In the case of a Gaussian noise, highly accurate quantization grids of ${\cal N}(0,I_q)$ distribution for dimensions $d =1$ up to $10$ and regularly sampled sizes from $N=1$ to $1\, 000$ can be downloaded from the quantization website 
\href{http://www.quantize.maths-fi.com}{www.quantize.maths-fi.com} (for non-commercial purposes). In other words, we consider,  instead of $(\ref{RBSDE:quantifrecursive})$, the following recursive scheme
\begin{equation}
\label{RBSDE:quantifrecursivehybride}
\left\{
\begin{array}{rl}
\widetilde{X}_{k} &=\mathcal{E}_{k-1}(\widehat{X}_{k-1},\widehat{\ve}_{k}),  \\
\smallskip
\widehat{X}_{k} & =\mbox{Proj}_{\Gamma_k}(\widetilde{X}_k), \qquad \quad \forall k= 1, \ldots, n.
\end{array}
\right .
\end{equation}
where $(\widehat{\ve}_k)_{k}$ is now a sequence of optimal quantizers of the Normal distribution $\mathcal{N}(0,I_q)$, which are already computed and kept off line. The main advantage of this approach is that using quantization grids of small size $N^{\ve}_k$ approaching the Gaussian random vectors $\ve_k$ gives the same precision as a Monte Carlo simulation of much larger size, always having in mind that the optimal quantizers can be computed offline and called when needed. This is a great gain in cost. \\
In the following, we establish $L^p$-error bounds of this hybrid recursive quantization scheme, for $p \in (1,2+d)$, in terms of the error between $\widehat{X}_k$ and $\widetilde{X}_k$ and the quantization error between $\ve_k$ and $\widehat{\ve}_k$ simultaneously. We recall that $\Delta \in [0,\Delta_{\max}), \, \Delta_{\max}>0$. %Note that for the Euler scheme, one can have $\Delta_{\max}=\frac{T}{n_0}$ if we consider schemes with step $\Delta=\frac{T}{n}$ and a number of steps $n>n_0$ for some $n_0>0$. 
\begin{thm}
	\label{hybridrecursive}
		Let $p \in (1,2+d)$ and $\delta>0$. Consider $(\bar X_k)_{0 \leq k \leq n}$ defined by $(\ref{RBSDE:SDEeuler})$ and $(\widehat{X}_k)_{0 \leq k \leq n}$ its hybrid recursive quantization sequence defined by $(\ref{RBSDE:quantifrecursivehybride})$. Assume that, for every $k \in \{0,\ldots,n\}$, $\widehat X_k$ is a stationary $L^2$-optimal quantization of $\widetilde{X}_k$ of size $N_k^X$ in the sense of $(\ref{RBSDE:stationary})$ with $\widehat{X}_0=\bar{X}_0=x_0 \in \R^d$ and $(\widehat \ve_k)_{0 \leq k \leq n}$ an $L^p$-optimal quantization sequence of the Gaussian distributed sequence $(\ve_k)_{0 \leq k \leq n}$ of size $N_k^{\ve}$. For every $k\in \{1, \ldots, n\}$, 
	$$\|\bar{X}_k-\widehat{X}_k\|_p  \leq \big(\widetilde K_{d,2,2+\delta,p} \vee \kappa_{d,2,\delta} \big)\sum_{l=1}^{k} [F^x_k]_{\rm Lip}^{k-l} C_{2+\delta, b,\sigma,T}^{\frac{1}{2+\eta}} (N_l^X)^{-\frac1d}+\sum_{l=1}^{k-1} \kappa_{d,p,\delta}[F^{\ve}_k]_{\rm Lip}^{k-l} \|\ve_l\|_p (N_l^{\ve})^{-\frac1d}$$
	where $\kappa_{d,2,\delta}$ is the constant given by Pierce's Lemma, $\widetilde K_{d,2,2+\delta,p}$ is given in Theorem \ref{distmismatch},
	$$C_{2+\delta, b,\sigma,T}= e^{t_k (C_1+C_2)} |x_0|^{2+\delta}+ \frac{C_3}{C_1+C_2} \left(e^{t_{k-1} (C_1+C_2)}-1  \right)$$
	with $C_1,C_2$ and $C_3$ are defined in Lemma $\ref{Xtildequad}$,
	$$[F^x_k]_{\rm Lip}=
	\left\{
	\begin{array}{ll}
	e^{\frac{\Delta}{p} \Big(c_p^{(1)}+L_{b,\sigma}\big(p+2^{p-1} c^{(2)}_{p,\Delta_{\max}} \big)\Big)} & \quad \mbox{if } p \in[2,2+d)\\
	e^{\frac{\Delta}{p}\Big(c_s^{(1)}+sL_{b,\sigma}+2^{s-1}L_{b,\sigma}^s c^{(2)}_{s,\Delta_{\max}} \big(\E|\ve|^s+\frac{p-s}{p}\big) \Big) } & \quad \mbox{if } p \in(1,2)
	\end{array}
	\right.$$ 
	and $$[F^{\ve}_k]_{\rm Lip}=
	\left\{
	\begin{array}{ll}
	 \Delta^{\frac1p} \Big(2^{p-1}c^{(2)}_{p,\Delta_{\max}}  L_{b,\sigma}\Big)^{\frac1p}& \quad \mbox{if } p \in[2,2+d)\\
	 \Delta^{\frac1p}\Big(\frac sp 2^{s-1}c^{(2)}_{s,\Delta_{\max}}  L_{b,\sigma}^s\Big)^{\frac1p}& \quad \mbox{if } p \in(1,2)
	\end{array}
	\right.$$
	where $s=p+1$, $c_p^{(1)}$ and $c^{(2)}_{p,\Delta_{\max}} $ are defined in Lemma $\ref{lemme1}$.
\end{thm}
\iffalse
\begin{rmq}
	The $L^p$-norms showing in this upper bound are optimal quantization errors induced by the approximation of $\widetilde{X}_k$ by $\widehat{X}_k$ and that of $Z_k$ by $\widehat{Z}_k$. Therefore, they are of the order of $(N^X_k \vee N^Z_k)^{-\frac1d}$ where $N^X_k$ is the size of the optimal quantization grid of $\widetilde X_k$, $N^Z_k$ is the size of the optimal quantization grid of $Z_k$ and $d$ is the dimension of $\widetilde X_k$. We can apply Pierce's Lemma $\ref{Zadoretpierce}(b)$ and inequality $(\ref{momentXtilde})$ if $\widehat{X}_k$ is a quadratic optimal quantizer of $\widetilde{X}_k$ to obtain universal non-asymptotic bounds.\\
\end{rmq}
\fi
\begin{proof}
	We start by showing that $\mathcal{E}_k$ is Lipschitz continuous with respect to its two variables. For every $x,x'  \in \R^d$ and $\R^d$-valued r.v. $\ve$ and $\ve'$ with standard Normal distribution, we consider two cases depending on the values of $p$.\\
	$\bullet$ If $p \in [2,2+d)$: Always keeping in mind that $\Delta< \Delta_{\max}$, Lemma $\ref{lemme1}$ yields
	\begin{align*}
	\E|\mathcal{E}_k(x,\ve)-\mathcal{E}_k(x',\ve')|^p =&\, \E\big|x-x'+\Delta\big(b(x)-b(x')\big)+\sqrt{\Delta} \big(\sigma(x)\ve-\sigma(x')\ve'\big)\big|^p\\
	\leq  & \,\big|x-x'+\Delta\big(b(x)-b(x')\big)\big|^p \big(1+c_p^{(1)} \Delta\big) +\Delta c^{(2)}_{p,\Delta_{\max}}  \E\big|\sigma(x)\ve-\sigma(x')\ve'\big|^p\\
	\leq & \,|x-x'|^p\big(1+\Delta [b]_{\rm Lip}\big)^p\big(1+c_p^{(1)}\Delta\big)+ \Delta c^{(2)}_{p,\Delta_{\max}}  \E\big|\sigma(x)\ve-\sigma(x')\ve'\big|^p
	\end{align*}
	where $c_p^{(1)}$ and $c^{(2)}_{p,\Delta_{\max}} $ are defined in Lemma $\ref{lemme1}$. Now, noticing that $|\sigma(x)\ve-\sigma(x')\ve'|=|\sigma(x)\ve-\sigma(x')\ve+\sigma(x')\ve-\sigma(x')\ve'|$ and using $(a+b)^p\leq 2^{p-1}(a^p+b^p)$ yield
	\begin{align*}
	\E|\mathcal{E}_k(x,\ve)-\mathcal{E}_k(x',\ve')|^p\leq & \, |x-x'|^p(1+\Delta [b]_{\rm Lip})^p(1+c_p^{(1)}\Delta)\\
	&+
	2^{p-1}c^{(2)}_{p,\Delta_{\max}} \Delta \Big(\E|\sigma(x)\ve-\sigma(x)\ve'|^p+ \E|\sigma(x)\ve'-\sigma(x')\ve'|^p\Big)\\
	\leq & \, |x-x'|^p\Big((1+\Delta [b]_{\rm Lip})^p(1+c_p^{(1)}\Delta)+2^{p-1}\Delta c^{(2)}_{p,\Delta_{\max}}  [\sigma]_{\rm Lip} \E|\ve'|^p\Big)\\
	&+ 2^{p-1}\Delta c^{(2)}_{p,\Delta_{\max}}  \|\sigma\|_{\infty}\E|\ve-\ve'|^p.
	\end{align*}
	Now, using the fact that $1+x \leq e^x$ yields
	\[ \E|\mathcal{E}_k(x,\ve)-\mathcal{E}_k(x',\ve')|^p\leq e^{\bar C \Delta}|x-x'|^p+\Delta \widetilde{C} \E|\ve-\ve'|^p  \]
	where $\bar C=p[b]_{\rm Lip}+c_p^{(1)}+2^{(p-3)_++p-1}(p-1)\big(1+\frac p2 \Delta_{\max}^{\frac p2 -1} \big) [\sigma]_{\rm Lip}$ and $\widetilde C= 2^{(p-3)_++p-1}(p-1)\big(1+\frac p2 \Delta_{\max}^{\frac p2 -1} \big) \|\sigma\|_{\infty}$. Then, applying $(a+b)^{\frac1p}\leq a^{\frac1p}+b^{\frac1p}$ for $a,b>0$ and $p>1$ yields
	\[ \|\mathcal{E}_k(x,\ve)-\mathcal{E}_k(x',\ve')\|_p \leq e^{\frac{\bar C \Delta}{p}}\|x-x'\|_p+(\Delta \widetilde{C})^{\frac1p} \|\ve-\ve'\|_p. \]
	Consequently, $\mathcal{E}_k$ is Lipschitz continuous for $k \in \{1, \ldots ,n\}$ and for $p \in [2,2+d)$ with Lipschitz coefficients $[F^x_k]_{\rm Lip}\leq e^{\Delta \bar C/p}$ and $[F^{\ve}_k]_{\rm Lip}\leq (\Delta \widetilde{C})^{\frac1p}$.\medskip \\
	$\bullet$ If $p\in (1,2)$: Consider $s=p+1>2$ so that $p-s<0$. One has
	$$\E|\mathcal{E}_k(x,\ve)-\mathcal{E}_k(x',\ve')|^p=\E \left[|\mathcal{E}_k(x,\ve)-\mathcal{E}_k(x',\ve')|^s\, |\mathcal{E}_k(x,\ve)-\mathcal{E}_k(x',\ve')|^{p-s} \right].$$
On the one hand, 
\begin{align*}
|\mathcal{E}_k(x,\ve)-\mathcal{E}_k(x',\ve')|^{p-s} %& \leq |x-x'+\Delta [b]_{\rm Lip}(x-x')+\sqrt{\Delta}(\sigma(x)\ve-\sigma(x')\ve')|^{p-s}\\
& \leq |x-x'|^{p-s}\left(1+\Delta [b]_{\rm Lip}+\sqrt{\Delta}\frac{|\sigma(x)\ve-\sigma(x')\ve'|}{|x-x'|} \right)^{p-s}\\
& \leq |x-x'|^{p-s}e^{(p-s)\left(1+\Delta [b]_{\rm Lip}+\sqrt{\Delta}\frac{|\sigma(x)\ve-\sigma(x')\ve'|}{|x-x'|}\right) } \quad (\mbox{since } 1+x \leq e^x)\\
& \leq |x-x'|^{p-s} \qquad (\mbox{ since } p-s<0).
\end{align*}
On the other hand, using inequality $(\ref{inegalitedulemme})$ from the proof of Lemma \ref{lemme1} (see Appendix), and noting $a=x-x'+\Delta [b]_{\rm Lip}(x-x')$ and $AZ=\sigma(x)\ve-\sigma(x')\ve'$, yields
\begin{align*}
|\mathcal{E}_k(x,\ve)-\mathcal{E}_k(x',\ve')|^{s}  \leq &|x-x'+\Delta [b]_{\rm Lip}(x-x')+\sqrt{\Delta}(\sigma(x)\ve-\sigma(x')\ve')|^{s}\\
\leq & |a|^s (1+\Delta c_s^{(1)}+ s\left(|a|^{s-1} \frac{a}{|a|} | A \sqrt{\Delta}Z \right)+ \Delta c_{2}^{(s,\Delta_{\max})} |AZ|^s.
\end{align*}
At this stage, one notices that $|a|^s\leq |x-x'|^s(1+\Delta [b]_{\rm Lip})^s$ and that
\begin{align*}
|AZ|  =|\sigma(x)\ve-\sigma(x')\ve'| \leq |\sigma(x)\ve-\sigma(x')\ve|+|\sigma(x')\ve-\sigma(x')\ve'| \leq [\sigma]_{\rm Lip} |x-x'| |\ve| +|\sigma(x')| |\ve-\ve'|,
\end{align*}
so that 
$$|AZ|^s \leq 2^{s-1} \left( [\sigma]_{\rm Lip}^s |x-x'|^s |\ve|^s +\|\sigma\|_{\infty}^s |\ve-\ve'|^s \right).$$
Hence, since $1+x \leq e^x$,
\begin{align*}
|\mathcal{E}_k(x,\ve)-\mathcal{E}_k(x',\ve')|^{s} \leq &|x-x'|^s (1+\Delta c_s^{(1)})(1+\Delta [b]_{\rm Lip})^s + s\left(|a|^{s-1} \frac{a}{|a|} | A \sqrt{\Delta}Z \right)\\
 & +\Delta c^{(2)}_{s,\Delta_{\max}} 2^{s-1} \left( [\sigma]_{\rm Lip}^s |x-x'|^s |\ve|^s +\|\sigma\|_{\infty}^s |\ve-\ve'|^s \right)\\
 \leq & |x-x'|^s e^{\Delta (c_s^{(1)}+s[b]_{\rm Lip})} + s\left(|a|^{s-1} \frac{a}{|a|} | A \sqrt{\Delta}Z \right)\\
 & +\Delta c^{(2)}_{s,\Delta_{\max}} 2^{s-1} \left( [\sigma]_{\rm Lip}^s |x-x'|^s |\ve|^s +\|\sigma\|_{\infty}^s |\ve-\ve'|^s \right).
\end{align*}
Consequently, applying the expectation and keeping in mind that $\E|AZ|=0$, we obtain
\begin{align*}
\E|\mathcal{E}_k(x,\ve)-\mathcal{E}_k(x',\ve')|^p \leq & \;e^{\Delta (c_s^{(1)}+s[b]_{\rm Lip})} \E|x-x'|^p\\
&+\Delta c^{(2)}_{s,\Delta_{\max}} 2^{s-1} \left( [\sigma]_{\rm Lip}^s \E[|x-x'|^p |\ve|^s] +\|\sigma\|_{\infty}^s \E[|\ve-\ve'|^s |x-x'|^{p-s}] \right).
\end{align*}
Using the fact that $\ve$ is independent of $\{x,x'\}$ and applying Young inequality with the conjugate exponents $\frac ps$ and $\frac{p}{p-s}$ to $\E[|\ve-\ve'|^s |x-x'|^{p-s}]$ yields 
\begin{align*}
\E|\mathcal{E}_k(x,\ve)-\mathcal{E}_k(x',\ve')|^p & \leq  \E|x-x'|^p \left(e^{\Delta (c_s^{(1)}+s[b]_{\rm Lip})}+ \Delta c^{(2)}_{s,\Delta_{\max}} 2^{s-1}[\sigma]_{\rm Lip}^s \E|\ve|^s\right) \\
& \quad + \Delta c^{(2)}_{s,\Delta_{\max}} 2^{s-1} \|\sigma\|^s_{\infty} \left(\frac sp \E|\ve-\ve'|^p+ \frac{p-s}{p} \E|x-x'|^p \right)\\
& \leq   \E|x-x'|^p\left(e^{\Delta (c_s^{(1)}+s[b]_{\rm Lip})}+ \Delta \tilde{\kappa}_1\right)+\Delta \tilde{\kappa}_2 \E|\ve-\ve'|^p\\
&\leq  \E|x-x'|^pe^{\Delta (c_s^{(1)}+s[b]_{\rm Lip})} (1+\Delta \tilde{\kappa}_1 e^{-\Delta (c_s^{(1)}+s[b]_{\rm Lip})})+\Delta \tilde{\kappa}_2 \E|\ve-\ve'|^p\\
&\leq  \E|x-x'|^pe^{\Delta (c_s^{(1)}+s[b]_{\rm Lip}+\tilde{\kappa}_1)}+\Delta \tilde{\kappa}_2 \E|\ve-\ve'|^p
\end{align*}
where $\tilde{\kappa}_1=c^{(2)}_{s,\Delta_{\max}}2^{s-1} \left([\sigma]_{\rm Lip}^s \E|\ve|^s+\|\sigma\|^s_{\infty}\frac{p-s}{p}\right)$ and $ \tilde{\kappa}_2=c^{(2)}_{s,\Delta_{\max}}2^{s-1}\|\sigma\|^s_{\infty}\frac sp$.
Then, 
$$\|\mathcal{E}_k(x,\ve)-\mathcal{E}_k(x',\ve')\|_p \leq \|x-x'\|_p e^{\Delta \kappa_1} +\|\ve-\ve'\|_p \Delta^{\frac1p} \kappa_2$$
where $\kappa_1=(c_s^{(1)}+s[b]_{\rm Lip}+\tilde{\kappa}_1)/p$ and $ \kappa_2=\tilde{\kappa}_2^{\frac 1p}.$
Consequently, $\mathcal{E}_k$ is lipschitz continuous for $k \in \{1, \ldots ,n\}$ with Lipschitz coefficients $ \ds [F^x]_{\rm Lip}\leq e^{\Delta \kappa_1}$ and $[F^{\ve}]_{\rm Lip} \leq \Delta^{\frac1p} \kappa_2$, for $p \in (1,2)$.\medskip\\
For the section step, the Lipschitz continuity of $\mathcal{E}_k$ yields 
	\begin{align*}
	\|\bar{X}_{k+1} -\widetilde{X}_{k+1}\|_p \leq & \|\mathcal{E}_k(\bar{X}_k,\ve_k)-\mathcal{E}_k(\widehat{X}_k,\widehat{\ve}_k)\|_p \\
	\leq & [F^x]_{\rm Lip}\|\bar{X}_k-\widehat{X}_k\|_p+[F^{\ve}]_{\rm Lip}\|\ve_k-\widehat{\ve}_k\|_p\\
	\leq & [F^x]_{\rm Lip} \|\bar{X}_k-\widetilde{X}_k\|_p+[F^x]_{\rm Lip} \|\widetilde{X}_k-\widehat{X}_k\|_p+[F^{\ve}]_{\rm Lip} \|\ve_k-\widehat{\ve}_k\|_p.
	\end{align*}
	Then, by induction, one has 
	$$\|\bar{X}_k-\widetilde{X}_k\|_p \leq \sum_{l=1}^{k-1} [F^x]_{\rm Lip}^{k-l} \|\widehat{X}_l-\widetilde{X}_l\|_p+[F^{\ve}]_{\rm Lip}^{k-l}\|\ve_l-\widehat{\ve}_l\|_p$$
	so that 
	\begin{align*}
	\|\bar{X}_k-\widehat{X}_k\|_p & \leq \|\bar{X}_k-\widetilde{X}_k\|_p+\|\widetilde{X}_k-\widehat{X}_k\|_p  \leq \sum_{l=1}^{k} [F^x]_{\rm Lip}^{k-l} \|\widetilde{X}_l-\widehat{X}_l\|_p+\sum_{l=1}^{k-1} [F^{\ve}]_{\rm Lip}^{k-l} \|\ve_l-\widehat{\ve}_l\|_p.
	\end{align*}
	Now, since $\widehat \ve_l$ is an optimal quantization of $\ve_l$ of size $N_l^{\ve}$, then Pierce's Lemma $\ref{Zadoretpierce}(b)$ yields 
	\begin{equation}
	\|\bar{X}_k-\widehat{X}_k\|_p \leq \sum_{l=1}^{k} [F^x]_{\rm Lip}^{k-l} \|\widetilde{X}_l-\widehat{X}_l\|_p + \sum_{l=1}^{k-1} [F^{\ve}]_{\rm Lip}^{k-l} \kappa_{d,p,\eta} \|\ve_l\|_{p+\eta} (N_l^{\ve})^{-\frac 1d}.
	\end{equation} 
	As for the error terms $\|\widetilde{X}_l-\widehat{X}_l\|_p$, one uses the same techniques as in the end of the proof of Theorem \ref{erreurXQR}, namely the distortion mismatch Theorem \ref{distmismatch} and Lemma \ref{Xtildequad}, to deduce the result.
	\hfill $\square$
\end{proof}
\section{Time discretization of the RBSDE}
\label{timedisc}
We consider the reflected backward stochastic differential equation RBSDE $(\ref{BSDE})$ with maturity $T$ given in the introduction and recalled below
\begin{equation*}
Y_t=g(X_T)+\int_t^T f(s,X_s,Y_s,Z_s)ds+K_T-K_t -\int_t^T Z_s . dW_s\, , \qquad t \in [0,T],
\end{equation*}
\begin{equation*}
Y_t \geq h(t,X_t) \quad \mbox {and} \quad \int_0^T (Y_s-h(s,X_s))dK_s =0
\end{equation*}
where $(W_t)_{t \geq 0}$ is a $q$-dimensional Brownian motion independent of $X_0$ and $(X_t)_{t \geq 0}$ is an $\R^d$-valued Brownian diffusion process solution to the SDE $(\ref{SDE})$ given in the introduction and recalled below
\begin{equation*}
X_t=X_0+\int_0^t b(s,X_s)ds +\int_0^t \sigma(s,X_s)dW_s\, , \qquad X_0=x_0 \in \R^d,
\end{equation*}
As explained, we need to approximate the solutions of these equations by discretization schemes. The time and space discretization of the forward process $(X_t)_{t \in [0,T]}$ have already been investigated and detailed in Section $\ref{QR}$. We proceed now with the time discretization of the solution of the RBSDE. Plugging the time-discretized process $(\bar{X}_{t_k})_{0 \leq k \leq n}$ in $(\ref{BSDE})$ will not make it possible to find an exact solution for the RBSDE.
%, that is due to $Z$ which is mostly intractable. 
Another approximation is needed, in which we discretize the term $Z_t$ itself: considering a sequence $(\ve_k)_{0 \leq k \leq n}$ of i.i.d. normally distributed random variables, the time discretization scheme associated to $(Y_t,Z_t)$ is given by the following backward recursion
\begin{align}
\label{YbarT}
\bar{Y}_{T} & = g(\bar{X}_{T})\\
\label{Ytildek}
\widetilde{Y}_{t_k} &=\E\big(\bar{Y}_{t_{k+1}}\,|\, \mathcal{F}_{t_k}\big)+\Delta \mathcal{E}_k\big(\bar{X}_{t_k},\, \E\big(\bar{Y}_{t_{k+1}}\,|\, \mathcal{F}_{t_k}\big),\bar{\zeta}_{t_k}\big)\, , \quad k=0,\ldots,n-1\\
\label{Zetabark}
\bar{\zeta}_{t_{k}}&=\frac{1}{\sqrt{\Delta}}\E\big(\bar{Y}_{t_{k+1}}\ve_{k+1}\,|\,\mathcal{F}_{t_k}\big)\, , \quad k=0,\ldots,n-1,\\
\label{Ybark}
\bar{Y}_{t_{k}}& = \widetilde{Y}_{t_{k}} \vee h_k(\bar{X}_{t_k})\, , \quad k=0,\ldots,n-1.
\end{align} 
As stated previously, this scheme differs from what was previously studied in the literature (see the references in the Introduction) since the conditional expectation is applied directly to $\bar{Y}_{t_{k+1}}$ inside the driver function which depends itself on the discretization $\bar{\zeta}_{t_{k}}$ of $Z_{t_k}.$ That is why it is interesting to establish a priori estimates for the error induced by the approximation with such a time discretization scheme. We note that, among others, time discretization errors for RBSDEs with a driver independent of $Z_t$ were establsihed in $\cite{BaPaSPA}$, errors for BSDEs (without reflection) with a driver depending on $Z_t$ and on the conditional expectation of $\bar Y_t$  in $\cite{PaSa18}$ and those for BSDEs (without reflection) with a driver depending on $Z_t$ but where the conditional expectation is applied to the whole function $f$ were studied in $\cite{Zhang04}$.\\

Since $\bar{X}_{t_k}$ is a Markov chain, one shows that there exists, for every $k \in \{0,\ldots,n\}$, Borel functions $\bar{y}_{t_k}$, $\widetilde{y}_{t_k}$ and $\bar{z}_{t_k}$ such that $\bar{Y}_{t_k}=\bar{y}_{t_k}(\bar{X}_{t_k})$, $\widetilde{Y}_{t_k}=\widetilde{y}_{t_k}(\bar{X}_{t_k})$ and $\bar{\zeta}_{t_k}=\bar{z}_{t_k}(\bar{X}_{t_k})$ and defined by
\begin{align}
\label{ybarT}
\bar{y}_T(x)&=g(x),\\
\label{ytildek}
\widetilde{y}_{t_k}(x)&=\E\,\bar{y}_{t_{k+1}}\big( \mathcal{E}_k(x,\ve_{k+1})\big)+\Delta \mathcal{E}_k\left(x,\E\, \bar{y}_{t_{k+1}}\big( \mathcal{E}_k(x,\ve_{k+1})\big), \bar{z}_{t_k}(x) \right)\\
\label{zbark}
\bar{z}_{t_k}(x)&=\frac{1}{\sqrt{\Delta}} \E\Big( \bar{y}_{t_{k+1}}\big( \mathcal{E}_k(x,\ve_{k+1})\big)\ve_{k+1}\Big)\\
\label{ybark}
\bar{y}_{t_k}(x)&=\widetilde{y}_{t_k}(x) \vee h_k(x).
\end{align}
where $\mathcal{E}_k(x, \ve_{k+1})=x+\Delta b_k(x)+\sqrt{\Delta} \sigma_k(x)\ve_{k+1}$ and $(\ve_{k})_{k\geq 0}$ are i.i.d random variables with distribution $\mathcal{N}(0,I_q)$.\\

In order to establish error bounds between $(Y_t,Z_t)$ and $(\bar{Y}_{t_k},\bar{Z}_{t_k})$, it is useful to introduce a time continuous process which extends $\ds \bar{Y}_{t_k}$. In fact, one notes that since the variable $\sum_{k=1}^{n-1}\bar{Y}_{t_{k+1}}-\E(\bar{Y}_{t_{k+1}}|\mathcal{F}_{t_k})$ is square integrable and measurable with respect to the augmented Brownian filtration $\mathcal{F}_{t_k}$, then, by the martingale representation Theorem, it can be considered as the terminal value of a Brownian martingale $\int_0^T \bar Z_s dW_s$ where the process $\bar{Z}_t$ is such that $\E \sup_{[0,T]} |\bar{Z}_s|^2 \leq \gamma_1< +\infty$ for a finite constant $\gamma_1$. So,
\begin{equation}
\bar{Y}_{t_{k+1}}-\E(\bar{Y}_{t_{k+1}}|\mathcal{F}_{t_k})=\int_{t_k}^{t_{k+1}}\bar{Z}_s dW_s \qquad \mbox{ for } \; k=0,\ldots,n-1.
\end{equation}
We note that 
\begin{equation}
\label{zetabar}
\bar{\zeta}_{t_k}=\frac{1}{\sqrt{\Delta}} \E\left(\bar{Y}_{t_{k+1}}\ve_{k+1}\, |\, \mathcal{F}_{t_k} \right) = \frac{1}{\Delta} \E\left(\int_{t_k}^{t_{k+1}}\bar{Z}_s ds\, |\, \mathcal{F}_{t_k} \right).
\end{equation}
Likewise, we define 
\begin{equation}
\label{zeta}
\zeta_{t_k}=\frac{1}{\Delta} \E\left(\int_{t_k}^{t_{k+1}}Z_s ds|\mathcal{F}_{t_k} \right)
\end{equation}
where $Z_s$ is the solution of the RBSDE $(\ref{BSDE})$ and one checks that $\bar{\zeta}_t$ is the best approximation of $\bar{Z}_t$ and $\zeta_t$ the best approximation of $Z_t$ in $L^2(d\P \times dt)$ among $\mathcal{F}_t$-measurable processes that are piecewise constant on the time intervals $[t_k,t_{k+1}[$.

Consequently, one may define (by a continuous extension) the c\`adl\`ag process $\widetilde{Y}_{t}$ on $[t_k,t_{k+1})$ and the l\`adc\`ag process  $\bar{Y}_{t}$ on $(t_k,t_{k+1}]$, by
\begin{equation}
\label{YbarYtildecont}
\widetilde{Y}_t=\bar{Y}_t=\bar{Y}_{t_{k+1}}-(t_{k+1}-t)\mathcal{E}_k\big(\bar{X}_{t_k},\E(\bar{Y}_{t_{k+1}}\,|\,\mathcal{F}_{t_k}),\bar{\zeta}_{t_k}\big)-\int_{t}^{t_{k+1}} \bar{Z}_s dW_s, 
\end{equation}
and the increasing positive process 
$$\bar{K}_{t_k}=\sum_{j=0}^k \left(h_j(\bar{X}_{t_j})-\widetilde{Y}_{t_k} \right)_+$$
such that $\bar{K}_t=\bar{K}_{t_k}$ for every $t \in (t_k,t_{k+1})$. Finally, we have the following representation
\begin{equation}
\label{Ytildecontgeneral}
\widetilde{Y}_t=\bar{Y}_T+\int_t^T f(\underline{s},\bar{X}_{\underline{s}},\E(\bar{Y}_{\bar{s}}\,|\,\mathcal{F}_{\underline{s}}),\bar{\zeta}_{\underline{s}}) \,ds -\int_{t}^{t_{k+1}} \bar{Z}_s dW_s +\bar{K}_T-\bar{K}_t.
\end{equation}
where $\underline s=t_k$ and $\bar s =t_{k+1}$ if $s \in (t_k,t_{k+1})$. Note that the introduction of $\bar{K}$ is mainly due to the fact that
$$\bar{Y}_{t_k}=\widetilde{Y}_{t_k} \vee h(t_k,\bar{X}_{t_k})=\widetilde{Y}_{t_k}+\left(h(t_j,\bar{X}_{t_j})-\widetilde{Y}_{t_k} \right)_+=\widetilde{Y}_{t_k}+\bar{K}_{t_k}-\bar{K}_{t_{k-1}}.$$

In the following, we will denote $\bar Y_k, \bar \zeta_k,\bar y_k, \bar K_k,$ etc. instead of $\bar Y_{t_k}, \bar \zeta_{t_k}, \bar y_{t_k}, \bar K_{t_k},$ etc. to alleviate notations, as well as $ \E_k(.)$ instead of $\E(.|\mathcal{F}_{t_k}). $  We recall that $\Delta \in [0,\Delta_{\max}), \, \Delta_{\max}>0$.
\begin{thm}
	\label{timeerror}
	Let $Y_t$ be the solution of $(\ref{BSDE})$ and $(\bar Y_k)_{0 \leq k \leq n}$ the corresponding time discretized process defined by $(\ref{Ybark})$. Assume that the functions $f$ and $h$ are Lipschitz continuous. Then, for every $k \in \{1, \ldots,n\}$,
	$$\E|Y_{k}-\bar{Y}_{k}|^2 \leq C_{b,\sigma,f,h,T}\left(\Delta +\int_0^T \E|Z_s-Z_{\underline{s}}|^2 ds \right)$$
	where $\underline{s}=t_k$ if $s \in [t_k,t_{k+1})$ and $C_{b,\sigma,f,h,T}$ is a real positive constant.\\
	Furthermore, there exists a finite constant $C>0$ such that
	$$\int_0^T \E|Z_s-Z_{\underline{s}}|^2 ds \leq C \sqrt{\Delta}.$$
\end{thm}
The second part of the theorem is established in $\cite{MaZhang05}$, see Theorem $6.3$. The proof of the first part is postponed to the appendix (see Appendix B).
\section{Space discretization of the RBSDE}
\label{spacedisc}
After the time discretization, we move to the space discretization schemes to approximate the solution of the RBSDE. We rely on the recursive quantization $(\widehat{X}_{t_k})_{0 \leq k \leq n}$ of the time-discretized scheme $(\bar{X}_{t_k})_{0 \leq k \leq n}$ to obtain the recursive quantization scheme associated to $(\ref{YbarT})$-$(\ref{Ytildek})$-$(\ref{Zetabark})$-$(\ref{Ybark})$. If we consider a sequence $(\ve_k)_{0 \leq k \leq n}$ of i.i.d. random variables with distribution $\mathcal{N}(0,I_q)$, this scheme is defined recursively by
\begin{align}
\label{YchapT}
\widehat{Y}_{T} & = g(\widehat{X}_{T})\\
\label{Zetachapk}
\widehat{\zeta}_{t_{k}}&=\frac{1}{\sqrt{\Delta}}\E_k\, \big(\widehat{Y}_{t_{k+1}} \ve_{k+1}\big)\, , \quad k=0,\ldots,n-1,\\
\label{Ychapk}
\widehat{Y}_{t_{k}}& = \max\Big(h_k(\widehat{X}_{t_k})\, ,\, \E_k\widehat{Y}_{t_{k+1}}+\Delta \mathcal{E}_k\big(\widehat{X}_{t_k},\E_k\widehat{Y}_{t_{k+1}},\widehat{\zeta}_{t_k}\big) \Big)\, , \quad k=0,\ldots,n-1.
\end{align} 
where $(\widehat{X}_{t_k})_{0 \leq k \leq n}$ is the recursively quantized process associated to $(\bar{X}_{t_k})_{0 \leq k \leq n}$ given by $(\ref{RBSDE:quantifrecursive})$ or $(\ref{RBSDE:quantifrecursivehybride})$. This quantization scheme is different than the optimal (or marginal) quantization schemes that were usually applied before in theses situations, in $\cite{BaPaSPA,theseilland, PaSa18}$ for example. The main difference is that since recursive quantization preserve the Markov property, the process $\widehat{Y}_{t_k}$ is $\mathcal{F}_{t_k}$-measurable for every $k \in \{0,\ldots,n\}$ where $\mathcal{F}_{t_k}=\sigma(W_{t_1},\ldots,W_{t_k}, \mathcal{N}_{\P})$ which is not the case for optimal quantization. More details on the utility of this character of recursive quantization will be presented in Section $\ref{algorithmics}$.\\

In the following, we will reconsider the notations with the indices $k$ instead of $t_k$ for every $k \in \{0,\ldots,n\}$, and we establish an upper bound for the quantization error induced by approximating $\bar{Y}_k$ by $\widehat{Y}_k$ in $L^p$ for $p \in (1,2+d)$ and $k \in \{1,\ldots,n\}$.  We recall that $\Delta \in [0,\Delta_{\max}), \, \Delta_{\max}>0$.
\begin{thm}
	Let $(\bar Y_k)_{0 \leq k \leq n}$ be the time-discretized process defined by $(\ref{Ybark})$ and $(\widehat Y_k)_{0 \leq k \leq n}$ the corresponding recursive quantized process defined by $(\ref{Ychapk})$. For every $p \in (1,2+d)$ and every $k \in \{1,\ldots,n\}$,
	\begin{equation}
	\|\bar{Y}_k-\widehat{Y}_k\|_p \leq  \left(\frac{\kappa_2}{\kappa_1}(e^{(T-t_k)\kappa_1}-1)+e^{(T-t_k)\kappa_1}([g]_{\rm Lip}^p\vee [h]_{\rm Lip}^p) \right)\big\|\max_{k \leq l \leq n}\big|\bar{X}_l-\widehat{X}_l| \big\|_p
	\end{equation}
	where $\kappa_1=p\kappa+(p-1)2^{p-2}$, $\kappa_2=2^{p-2}[f]_{\rm Lip}^p(1+p\Delta^{p-1})$ and $\kappa= \frac{c_s^{(1)}+s[f]_{\rm Lip}+[f]_{\rm Lip}^s c^{(3)}_{s,\Delta_{\max},\ve_{k+1}}}{s}$, the positive finite constants $c_s^{(1)}$ and $c^{(3)}_{s,\Delta_{\max},\ve_{k+1}}$ are defined in Lemmas $\ref{lemme1}$ and $\ref{Xtildequad}$.
\end{thm} 
\begin{rmq}
	The norms $\|\bar{X}_l-\widehat{X}_l\|_p$ are recursive quantization errors established in Theorems $\ref{erreurXQR}$ and $\ref{hybridrecursive}$ for $p \in (1,2+d)$. We recall that, for every $l \in \{1, \ldots,n\}$, one has $\|\bar{X}_l-\widehat{X}_l\|_p= \mathcal{O}(N_l^{-\frac1d})$ where $N_l$ is the size of the quantization grid corresponding to $\widehat{X}_l$. 
\end{rmq}
\begin{proof}
	For every $k \in \{1,\ldots,n\}$, we use the inequality $|\max(a,b)-\max(a',b')| \leq \max(|a-a'|,|b-b'|)$ and have 
	\begin{align*}
	|\bar{Y}_k-\widehat{Y}_k| \leq \max \left(|h_k(\bar{X}_k)-h_k(\widehat{X}_k)| ,\Big| \E_k\bar{Y}_{k+1}-\E_k\widehat{Y}_{k+1}+\Delta \big(\mathcal{E}_k(\bar{X}_k,\E_k\bar{Y}_{k+1}, \bar{\xi}_k)-\mathcal{E}_k(\widehat{X}_k,\E_k\widehat{Y}_{k+1}, \widehat{\xi}_k) \big)\Big| \right)
	\end{align*}
	We denote $\beta_k=\E_k (\bar{Y}_{k+1}-\widehat{Y}_{k+1})+\Delta \left(\mathcal{E}_k\big(\bar{X}_k,\E_k\bar{Y}_{k+1}, \bar{\xi}_k\big)-\mathcal{E}_k\big(\widehat{X}_k,\E_k\widehat{Y}_{k+1}, \widehat{\xi}_k\big) \right)$ and we have 
	\begin{align*}
	\beta_k= \E_k (\bar{Y}_{k+1}-\widehat{Y}_{k+1})+\Delta\Big(\widehat{A}_k (\bar{X}_k-\widehat{X}_k)+\widehat{B}_k\E_k(\bar{Y}_{k+1}-\widehat{Y}_{k+1})+\frac{\widehat{C}_k}{\sqrt{\Delta}}\E_k\big((\bar{Y}_{k+1}-\widehat{Y}_{k+1})\ve_{k+1}\big) \Big)
	\end{align*}
	where $$\widehat{A}_k=\frac{\mathcal{E}_k\big(\bar{X}_k,\E_k\bar{Y}_{k+1}, \bar{\xi}_k\big)-\mathcal{E}_k\big(\widehat{X}_k,\E_k\bar{Y}_{k+1}, \bar{\xi}_k\big)}{\bar{X}_k-\widehat{X}_k}\mathds{1}_{\bar{X}_k \neq \widehat{X}_k},$$
	$$ \widehat{B}_k=\frac{\mathcal{E}_k\big(\widehat{X}_k,\E_k\bar{Y}_{k+1}, \bar{\xi}_k\big)-\mathcal{E}_k\big(\widehat{X}_k,\E_k\widehat{Y}_{k+1}, \bar{\xi}_k\big)}{\E_k(\bar{Y}_{k+1}-\widehat{Y}_{k+1})}\mathds{1}_{\E_k\bar{Y}_{k+1}\neq \E\widehat{Y}_{k+1}},$$
	$$ \widehat{C}_k=\frac{\mathcal{E}_k\big(\widehat{X}_k,\E_k\widehat{Y}_{k+1}, \bar{\xi}_k\big)-\mathcal{E}_k\big(\widehat{X}_k,\E_k\widehat{Y}_{k+1}, \widehat{\xi}_k\big)}{\E_k\big((\bar{Y}_{k+1}-\widehat{Y}_{k+1})\ve_{k+1}\big)} \mathds{1}_{\bar{\xi}_k \neq \widehat{\xi}_k}.$$
	It is clear that $\max\Big(|\widehat{A}_k|,\,|\widehat{B}_k|,\,|\widehat{C}_k| \Big) \leq [f]_{\rm Lip}$, so one has 
	$$|\beta_k| \leq \Delta [f]_{\rm Lip} |\bar{X}_k-\widehat{X}_k|+\E_k \Big| (1+\Delta \widehat{B}_k+\sqrt{\Delta}
	\widehat{C}_{k}\ve_{k+1})(\bar{Y}_{k+1}-\widehat{Y}_{k+1}) \Big|.$$
	At this stage, we consider two conjugate exponents $r \in (1, 2 \, \wedge \, p)$ and $s =\frac{r}{r-1} >2$ and we apply conditional H\"older's inequality 
	\begin{align*}
	\E_k \Big| (1+\Delta \widehat{B}_k+\sqrt{\Delta}\widehat{C}_k \ve_{k+1})(\bar{Y}_{k+1}-\widehat{Y}_{k+1}) \Big|\leq & \left(\E_k |1+\Delta \widehat{B}_k+\sqrt{\Delta}\widehat{C}_k \ve_{k+1}|^s\right)^{\frac 1s} \left(\E_k |\bar{Y}_{k+1}-\widehat{Y}_{k+1}|^r \right)^{\frac 1r}.
	\end{align*}
	Since $s>2$, one can apply Lemma $\ref{lemme1}$ with $a=1+\Delta \widehat{B}_k$ and $A=\widehat{C}_k$ and obtains
	\begin{align*}
	\E_k |1+\Delta \widehat{B}_k+\sqrt{\Delta}\widehat{C}_k\ve_{k+1}|^s & \leq (1+\Delta [f]_{\rm Lip})^s (1+c_s^{(1)} \Delta)+\Delta [f]_{\rm Lip}^s c^{(3)}_{s,\Delta_{\max},\ve_{k+1}}\\
	& \leq e^{s\Delta [f]_{\rm Lip}+\Delta c_s^{(1)}}+\Delta [f]_{\rm Lip}^s c^{(3)}_{s,\Delta_{\max},\ve_{k+1}}\\
	& \leq e^{\Delta(c_s^{(1)}+s[f]_{\rm Lip})}(1+\Delta [f]_{\rm Lip}^s c^{(3)}_{s,\Delta_{\max},\ve_{k+1}}e^{-\Delta(c_s^{(1)}+s[f]_{\rm Lip})})\\
	& \leq e^{\Delta (c_s^{(1)}+s[f]_{\rm Lip}+[f]_{\rm Lip}^s c^{(3)}_{s,\Delta_{\max},\ve_{k+1}})}
	\end{align*}
	where $c_s^{(1)} $ and $c^{(3)}_{s,\Delta_{\max},\ve_{k+1}}$ are real constants defined in Lemmas $\ref{lemme1}$ and $\ref{Xtildequad}$. Therefore,
	$$|\beta_k| \leq \Delta [f]_{\rm Lip} |\bar{X}_k-\widehat{X}_k|+e^{\kappa \Delta}\left(\E_k |\bar{Y}_{k+1}-\widehat{Y}_{k+1}|^r \right)^{\frac 1r},$$
	where $\kappa= \frac{c_s^{(1)}+s[f]_{\rm Lip}+[f]_{\rm Lip}^s c^{(3)}_{s,\Delta_{\max},\ve_{k+1}}}{s}$, and
	$$|\bar{Y}_k -\widehat{Y}_k|^p \leq \max \Big([h]_{\rm Lip}^p|\bar{X}_k-\widehat{X}_k|^p\,, \, |\beta_k|^p\Big).$$
	Now, using  inequality $(\ref{a+bexpor})$ yields
	$$|\beta_k|^p\leq e^{p\kappa \Delta}\left(\E_k |\bar{Y}_{k+1}-\widehat{Y}_{k+1}|^r \right)^{\frac pr} \big(1+(p-1)2^{p-2}\varepsilon^p\big) +2^{p-2}[f]_{\rm Lip}^p |\bar{X}_k-\widehat{X}_k|^p \Delta^p \Big(p+\frac{1}{\varepsilon^{p(p-1)}} \Big).$$
	We choose $\varepsilon =\Delta^{\frac1p}$ so that  $\Delta^p \Big(p+\frac{1}{\varepsilon^{p(p-1)}}\Big)= \Delta(1+p\Delta^{p-1})$ and hence
	$$|\beta_k|^p\leq e^{\kappa_1 \Delta}\left(\E_k |\bar{Y}_{k+1}-\widehat{Y}_{k+1}|^r \right)^{\frac pr}+\Delta \kappa_2 |\bar{X}_k-\widehat{X}_k|^p$$ 
	where  $\kappa_1=p\kappa+(p-1)2^{p-2}$ and $\kappa_2=2^{p-2}[f]_{\rm Lip}^p(1+p\Delta^{p-1})$. Moreover, by our choice of $r$, we have that $\frac pr>1$ so we apply Jensen's inequality and obtain 
	$$|\beta_k|^p\leq e^{\kappa_1 \Delta}\E_k |\bar{Y}_{k+1}-\widehat{Y}_{k+1}|^p +\Delta \kappa_2 |\bar{X}_k-\widehat{X}_k|^p.$$
	Hence, having in mind that $\bar{X}_k, \widehat{X}_k, \bar{Y}_k$and $\widehat{Y}_k$ are all $\mathcal{F}_{t_k}$-measurable processes, one has 
	\begin{equation}
	\label{erravrecurr}
	\E_k |\bar{Y}_k-\widehat{Y}_k|^p\leq \max \Big([h]_{\rm Lip}^p\E_k|\bar{X}_k-\widehat{X}_k|^p\,, \, e^{\kappa_1 \Delta}\E_k |\bar{Y}_{k+1}-\widehat{Y}_{k+1}|^p +\Delta \kappa_2 \E_k|\bar{X}_k-\widehat{X}_k|^p\Big).
	\end{equation}
	At this stage, we aim to prove that $\E_k|\bar{Y}_k-\widehat{Y}_k|^p$ satisfies the following backward induction
	\begin{equation}
	\label{recurrence}
	\E_k|\bar{Y}_k-\widehat{Y}_k|^p \leq e^{(n-k) \kappa_1 \Delta}\big([g]_{\rm Lip}^p\vee [h]_{\rm Lip}^p\big) \E_k \max_{k \leq i \leq n} |\bar{X}_i-\widehat{X}_i|^p +\Delta \kappa_2\sum_{i=k}^{n-1} e^{(i-k)\kappa_1 \Delta} \E_k |\bar{X}_i-\widehat{X}_i|^p.	
	\end{equation}
	First, it is clear that $\E_n |\bar{Y}_n-\widehat{Y}_n|^p \leq [g]_{\rm Lip}^p \E_n|\bar{X}_n-\widehat{X}_n|^p$ so the induction is satisfied for $k=n$. We assume that $(\ref{recurrence})$ is true for $k+1$ i.e.
	\begin{align}
	\label{k+1}
	\E_{k+1}|\bar{Y}_{k+1}-\widehat{Y}_{k+1}|^p  &\leq e^{(n-k-1) \kappa_1 \Delta}([g]_{\rm Lip}^p\vee [h]_{\rm Lip}^p) \E_{k+1} \max_{k+1 \leq i\leq n} |\bar{X}_i-\widehat{X}_i|^p \nonumber \\
	&+\Delta \kappa_2\sum_{i=k+1}^{n-1} e^{(i-k-1)\kappa_1 \Delta} \E_{k+1} |\bar{X}_i-\widehat{X}_i|^p
	\end{align}
	and show it for $k$. In fact, since $\E_k\E_{k+1}(\cdot)=\E_k(\cdot)$, one has, by merging $(\ref{erravrecurr})$ with $(\ref{k+1})$, that
	\begin{align*}
	\E_k|\bar{Y}_k-\widehat{Y}_k|^p  \leq & \max \Big([h]_{\rm Lip}^p\E_k|\bar{X}_k-\widehat{X}_k|^p\,, \, e^{\kappa_1 \Delta}\E_k \E_{k+1}|\bar{Y}_{k+1}-\widehat{Y}_{k+1}|^p +\Delta \kappa_2 \E_k|\bar{X}_k-\widehat{X}_k|^p\Big)\\
	 \leq & \max \Big([h]_{\rm Lip}^p\E_k|\bar{X}_k-\widehat{X}_k|^p\,, \, \Delta \kappa_2 \E_k |\bar{X}_k-\widehat{X}_k|^p +\Delta \kappa_2\sum_{i=k+1}^{n-1} e^{(i-k)\kappa_1 \Delta} \E_k \E_{k+1} |\bar{X}_i-\widehat{X}_i|^p\\
	 &+e^{(n-k) \kappa_1 \Delta}([g]_{\rm Lip}^p\vee [h]_{\rm Lip}^p) \E_k\E_{k+1} \max_{k+1 \leq i \leq n} |\bar{X}_i-\widehat{X}_i|^p  \Big)\\
	\leq & \max \Big([h]_{\rm Lip}^p\E_k|\bar{X}_k-\widehat{X}_k|^p\,, e^{(n-k) \kappa_1 \Delta}([g]_{\rm Lip}^p\vee [h]_{\rm Lip}^p) \E_k \max_{k \leq i \leq n} |\bar{X}_i-\widehat{X}_i|^p \\
	&+\Delta \kappa_2\sum_{i=k}^{n-1} e^{(i-k)\kappa_1 \Delta} \E_k |\bar{X}_i-\widehat{X}_i|^p \Big)
	\end{align*}
	since $\max_{ k+1 \leq i \leq n } \alpha_i \leq \max_{k \leq i \leq n} \alpha_i$ for $\alpha_i>0$. Furthermore, noticing that
	\begin{align*}
	[h]_{\rm Lip}^p \E_k |\bar{X}_k-\widehat{X}_k|^p & \leq \big([g]_{\rm Lip}^p\vee [h]_{\rm Lip}^p\big) \E_k \max_{k \leq i \leq n} |\bar{X}_i-\widehat{X}_i|^p \leq e^{(n-k) \kappa_1 \Delta}\big([g]_{\rm Lip}^p\vee [h]_{\rm Lip}^p\big) \E_k \max_{k \leq i \leq n} |\bar{X}_i-\widehat{X}_i|^p
	\end{align*}
	because $e^{(n-k) \kappa_1 \Delta }>1$, one concludes the induction $(\ref{recurrence})$. This yields
	\begin{equation}
	\label{lastineq}
	\E_k|\bar{Y}_k-\widehat{Y}_k|^p \leq e^{(T-t_k) \kappa_1}\big([g]_{\rm Lip}^p\vee [h]_{\rm Lip}^p\big) \E_k \max_{k \leq i \leq n} |\bar{X}_i-\widehat{X}_i|^p +\Delta \kappa_2 \E_k \max_{k \leq i \leq n} |\bar{X}_i-\widehat{X}_i|^p \sum_{i=k}^{n-1} e^{(i-k)\kappa_1 \Delta}.
	\end{equation}
	Finally, since $e^x-1 \geq x$ for $x \ge 0$, one has
	$$\sum_{i=k}^{n-1} e^{(i-k)\kappa_1 \Delta}= \frac{e^{(n-k)\kappa_1 \Delta}-1}{e^{\kappa_1 \Delta}-1} \leq \frac{e^{(T-t_k)\kappa_1 }-1}{\Delta \kappa_1}$$
	and  then deduces the result by taking the expectation in ($\ref{lastineq}$) . 
	\hfill $\square$
\end{proof}
\section{Algorithmics}
\label{algorithmics}
Our aim is to write $(\widehat{Y}_k,\widehat{\zeta}_k)$, which approximates the solution of the RBSDE $(\ref{BSDE})$, in a form that allows us to compute their values. For this, we first note that $(\bar{X}_k)_{0 \leq k \leq n}$ and $(\widehat{X}_k)_{0 \leq k \leq n}$ are both $\mathcal{F}_{t_k}$-Markov chains where $\mathcal{F}_{t_k}=\sigma(W_s, s \leq t_k, \mathcal{N}_{\P})$, for every $k \in \{0,\ldots,n\}$, with respective transitions $P_k(x,dy)=\P(\bar{X}_{k+1}\in dy|\bar{X}_k=x)$ and $\widehat{P}_k(x,dy)=\P(\widehat{X}_{k+1}\in dy|\widehat{X}_k=x)$. The main advantage of recursive quantization is that it preserves the Markovian property of $(\widehat{X}_k)_{0 \leq k \leq n}$ with respect to the filtration $(\mathcal{F}_{t_k})_{0\leq k \leq n}=\big(\sigma(W_s, s \leq t_k, \mathcal{N}_{\P})\big)_{0 \leq k \le n}$. Note that, for optimal quantization, the trick was to force the Markov property by conditioning with respect to the filtration $\widehat{\mathcal{F}}_{t_k}=\sigma(\widehat{X}_0,\ldots,\widehat{X}_k)$ instead of $\mathcal{F}_{t_k}$ in $(\ref{Zetachapk})$-$(\ref{Ychapk})$. The price to pay is that the approximations $\|\bar X_k-\widehat X_k\|_p$, for every $k \in \{1, \ldots,n\}$, are less accurate (but not in a drastic way). This point is discussed in details in \cite{PaSa18}. \\

\noindent For every bounded or non-negative Borel function $f$, one has $\ds P_kf(x)=\int_{\R^d} f(y)P_k(x,dy)$, so that
$$\E \big(f(\bar X_{k+1})\,|\, {\cal F}_{t_k}\big) = P_{k}f(\bar X_k)
   \qquad \mbox{ and } \qquad \E \big(f(\widehat X_{k+1})\,|\, {\cal F}_{t_k}\big) = \widehat P_{k}f(\widehat X_k)
.$$
Moreover, we introduce 
$$Q_{k}f(\bar{X}_k)=\frac{1}{\sqrt{\Delta}}\E \big(f(\bar{X}_{k+1}) \ve_{k+1}\, | \, {\cal F}_{t_k}\big)  \qquad \mbox{ and } \qquad \widehat{Q}_{k}f(\widehat{X}_k)=\frac{1}{\sqrt{\Delta}}\E \big(f(\widehat{X}_{k+1}) \ve_{k+1}\, | \, {\cal F}_{t_k} \big)$$
where $(\ve_k)_{0 \leq k \leq n}$ are i.i.d. with Normal distribution $\mathcal{N}(0,I_q).$\\

Similarly to the functions $(\bar{y}_k)_{0 \leq k \leq n}$ defined by $(\ref{ybark})$, one shows that there exists Borel functions $(\widehat{y}_k)_{0 \leq k \leq n}$ such that $\widehat{Y}_k=\widehat{y}_k(\widehat{X}_k)$ for every $k \in \{0,\ldots,n\}$. They are defined recursively by the following Backward Dynamic Programming Principle (BDPP)
\begin{equation}
\label{ykchapeau}
\left\{
\begin{array}{rl}
\widehat{y}_n &=h_n\\
\widehat{y}_k&=\max \Big(h_k,\,\widehat{P}_{k}\widehat{y}_{k+1}+\Delta \mathcal{E}_k\big(.,\widehat{P}_{k}\widehat{y}_{k+1},\widehat{Q}_{k}\widehat{y}_{k+1}\big)\Big)\, , \quad k =0,\ldots,n-1,
\end{array}
\right.
\end{equation}
This BDPP can also be written in distribution, one can write $(\bar{y}_k)_{0 \leq k \leq n}$ as
\begin{equation*}
\left\{
\begin{array}{rl}
\bar{y}_n &=h_n\\
\bar{y}_k&=\max \Big(h_k,\,P_{k}\bar{y}_{k+1}+\Delta \mathcal{E}_k\big(.,P_{k}\bar{y}_{k+1},Q_{k}\bar{y}_{k+1}\big)\Big)\, , \quad k =0,\ldots,n-1,\\
\end{array}
\right.
\end{equation*}
The fact that $\bar Y_k =\bar y_k(\bar X_k)$ and $\widehat Y_k=\widehat y_k(\widehat X_k)$ can easily be checked by a backward induction relying on $(\ref{YbarT})$-$(\ref{Ytildek})$-$(\ref{Ybark})$ and $(\ref{YchapT})$-$(\ref{Ychapk})$ respectively.
\iffalse: if we proceed by backward induction and rely on $(\ref{YbarT})$-$(\ref{Ytildek})$-$(\ref{Ybark})$ and $(\ref{YchapT})$-$(\ref{Ychapk})$ respectively, one has
\begin{align*}
\bar{Y}_k&=\max \left( h_k(\bar X_k),\, \E_k\big(\bar{y}_{k+1}(\bar{X}_{k+1})\big)+\Delta \mathcal{E}_k\Big(\bar{X}_k,\E_k\,\big(\bar{y}_{k+1}(\bar{X}_{k+1})\big),\frac{1}{\sqrt{\Delta}}\E_k\, \big(\bar{y}_{k+1}(\bar{X}_{k+1})\ve_{k+1}\big)\Big) \right)\\
&= \max \left( h_k(\bar X_k),\, P_{k}\bar{y}_{k+1}(\bar{X}_k)+\Delta \mathcal{E}_k\Big(\bar{X}_k,P_{k}\bar{y}_{k+1}(\bar{X}_k),Q_{k}\bar{y}_{k+1}(\bar{X}_k)\Big) \right)\\
&= \bar{y}_k(\bar{X}_k).
\end{align*} 
and
\begin{align*}
\widehat{Y}_k&=\max \left( h_k(\widehat X_k),\,\E_k\big(\widehat{y}_{k+1}(\widehat{X}_{k+1})\big)+\Delta \mathcal{E}_k\Big(\widehat{X}_k,\E_k\,\big(\widehat{y}_{k+1}(\widehat{X}_{k+1})\big),\frac{1}{\sqrt{\Delta}}\E_k\, \big(\widehat{y}_{k+1}(\widehat{X}_{k+1})\ve_{k+1}\big)\Big) \right)\\
&=\max \left( h_k(\widehat X_k),\,\widehat{P}_{k}\widehat{y}_{k+1}(\widehat{X}_k)+\Delta \mathcal{E}_k\Big(\widehat{X}_k,\widehat{P}_{k}\widehat{y}_{k+1}(\widehat{X}_k),\widehat{Q}_{k}\widehat{y}_{k+1}(\widehat{X}_k)\Big) \right)\\
&= \widehat{y}_k(\widehat{X}_k).
\end{align*} 
\fi
Furthermore, there exists functions $\bar z_k$ and $\widehat z_k$ such that $\bar{\zeta}_k=\bar{z}_k(\bar{X}_k)$ and $\widehat{\zeta}_k=\widehat{z}_k(\widehat{X}_k)$, defined by $$\bar{z}_k=Q_{k}\bar y_{k+1} \qquad \mbox{ and } \qquad  \widehat{z}_k=\widehat{Q}_{k}\widehat{y}_{k+1}.$$

In order to compute $\widehat{Y}_k$ and $\widehat{\zeta}_k$, we first need to compute the optimal (or at least optimized) recursive quantization $\widehat{X}_k$ of $\bar{X}_k$ for every $k \in \{0,\ldots,n\}$ and the corresponding transition weights. We will consider the quadratic case $p=2$ for all numerical aspects. 
\subsection{Computation of the recursive quantizers}
\label{constructionXk}
As defined previously, the recursive quantization of $(\bar{X}_k)_{0 \leq k \leq n}$ is realized via $(\ref{RBSDE:quantifrecursive})$ (or $(\ref{RBSDE:quantifrecursivehybride})$). In a quadratic framework, the computation of the optimal quantization grids $\Gamma_k$ of $\widetilde{X}_k$ of size $N_k$, at each time step $t_k$, is achieved by algorithms such as CLVQ (Competitive Learning Vector Quantization), Lloyd's algorithm or Newton-Raphson. These algorithms are presented in details in \cite{PaPrin03} for example. Here, we expose a variant of Lloyd's algorithm for recursive quantization.\\

For $k \in \{1, \ldots, n\}$, computing an optimal quantizer $\widehat{X}_k^{\Gamma_k}$ of $\widetilde{X}_k$ consists in computing the grid $\Gamma_k$ solution to the minimization problem
$$\Gamma_k \in \mbox{argmin}\Big\{\|\widehat{X}_k^{\Gamma} - \widetilde{X}_k\|_2^2, \, \Gamma \subset \R^d, \, \mbox{card}(\Gamma) \leq N_k\Big\}.$$
The construction of these grids is performed recursively at each step $t_k$ in a forward way. It is somehow an {\em embedded} optimization. We suppose that, at time $t_k$, the grid $\Gamma_k=\{x_1^k, \ldots, x_{N_k}^k\}$ is already computed (optimized) and that $\widetilde{X}_k$ has been quantized by $\widehat{X}_k=\sum_{i=1}^{N_k} x_i^k \mathds{1}_{C_i(\Gamma_k)}$ where $(C_i(\Gamma_k))_{1 \leq i \leq N_k}$ is  the Vorono\"i diagram associated to $\widehat{X}_k$ and defined by $(\ref{Voronoicells})$. Then, at time step $t_{k+1}$, we build the grid $\Gamma_{k+1}$ that minimizes the quadratic distortion ${G}_{k+1}^2(\Gamma)$ defined by $(\ref{distfunction})$ and written as a function of the grid $\Gamma_k=\{x_1^k,\ldots,x_{N_k}^k\}$ computed at the previous step. So, if $\Gamma_{k+1}=\{x_1^{k+1}, \ldots, x_{N_{k+1}}^{k+1}\}$, then one has, for every $j \in \{1, \ldots,N_{k+1}\}$,
%Then, the Lloyd's algorithm is given by
%$$\widehat{X}_k^{[n+1]}=\E\left(\widetilde{X}_k \mid \widehat{X}_k^{[n]}\right)= \frac{\E\left[\widetilde{X}_k \mathds{1}_{\{\widetilde{X}_k \in C_j(\Gamma_k^{[n]})\}}\right]}{\P\left(\widetilde{X}_k \in C_j(\Gamma_k^{[n]})\right)}.$$
%In other words, if $\Gamma_{k+1}=\{x_1^{k+1}, \ldots, x_{N_{k+1}}^{k+1}\}$, then one can have for every $j \in \{1, \ldots,N_{k+1}\}$,
\begin{align}
\label{RBSDE:xj}
x_j^{k+1}& = \E\Big( \widetilde{X}_{k+1}\,|\,\widehat{X}_{k+1} \in C_j(\Gamma_{k+1}) \Big)   = \frac{\sum_{i=1}^{N_k}p_i^k \E\Big(\mathcal{E}_k(x_i^k,\ve_{k+1}) \mathds{1}_{\{\mathcal{E}_k(x_i^k,\ve_{k+1})\in C_j(\Gamma_{k+1})\}} \Big)}{p_j^{k+1}}.
\end{align}
Recalling that $\mathcal{E}_k(x,\ve_{k+1})=x+\Delta b_k(x)+\sqrt{\Delta}\sigma_k(x) \ve_{k+1}$, it is important to notice that, for every $k \in \{1, \ldots,n\}$ and $i \in \{1, \ldots, N_k\}$, $\mathcal{E}_k(x_i^k,\ve_{k+1}) \sim \mathcal{N}(m_i^k, \Sigma_i^k)$ where $m_i^k=x_i^k+\Delta b_k(x_i^k)$ and $\Sigma_i^k= \sqrt{\Delta} \sigma_k(x_i^k)$. \\

\noindent We are interested in more than just computing the distribution of $(\widehat X_k)_{0 \leq k \leq n}$, the computation of the transition matrices $P_k=(p^k_{ij})_{_{ij}}$ is even more fundamental among the companion parameters in view of our applications.
% Since $(\widehat{X}_k)_{1 \leq k \leq n}$ is a Markov chain, it is necessary to compute the corresponding transition weight matrices $P_k=(p^k_{ij})$. 
For every $k \in \{1, \ldots, n\}$ and $i,j \in \{1, \ldots, N_k\}$, the transition probability $p_{ij}^k$ from $x_i^k$ to $x_j^{k+1}$ is given by
\begin{equation}
\label{RBSDE:poidsij}
p_{ij}^k  = \P\left( \widetilde{X}_{k+1} \in C_j(\Gamma_{k+1})\, |\, \widetilde{X}_{k} \in C_i(\Gamma_{k}) \right) = \P\left( \mathcal{E}_k(x_i^k,\ve_{k+1})\in C_j(\Gamma_{k+1})\right).
\end{equation} 
This identity allows the computation of the weights $p_j^{k+1}$ of the Vorono\"i cells $C_j(\Gamma_{k+1})$, for every $j \in \{1, \ldots, N_{k+1}\}$, via the classical (discrete time) forward Kolmogorov equation. They are given by
\begin{align}
\label{RBSDE:poidsj}
p_j^{k+1} & = \P\big(\widetilde{X}^{k+1} \in C_j(\Gamma_{k+1})\big) = \sum_{i=1}^{N_k} p_i^k  \P\Big( \mathcal{E}_k(x_i^k,\ve_{k+1})\in C_j(\Gamma_{k+1})\Big).
\end{align}
\paragraph{One-dimensional setting $q=d=1$:} 
The transition weights $p_{ij}^k$ can be computed in a direct way as follows:  for every $i \in \{1,\ldots,N_k\}$ and $j \in \{1,\ldots,N_{k+1}\}$  
\begin{align*}
p_{ij}^k %& = \P\left( \widehat{X}_{k+1} \in C_j(\Gamma_{k+1}) \,|\, \widehat{X}_{k} \in C_i(\Gamma_{k}) \right)\\
& = \P \Big(\widetilde X_{k+1} \leq x_{j+\frac12}^{{k+1}} \, | \, \widehat X_k =x_i^k \Big)-\P \Big(\widetilde X_{k+1} \leq x_{j-\frac12}^{{k+1}} \, | \, \widehat X_k =x_i^k \Big) = \Phi_0\big( x_{i,j_+}^{k+1}\big)-\Phi_0\big( x_{i,j_-}^{k+1}\big)
\end{align*}  
where $\Phi_0$ is the cumulative distribution function of the standard Normal distribution $\mathcal{N}(0,1)$ and 
$$x_{i,j_+}^{k+1}=\frac{x_{j+\frac12}^{k+1}-x_i^k-\Delta b_k(x_i^k)}{\sqrt{\Delta}\sigma_k(x_i^k)} \quad \mbox{ and } \quad x_{i,j_-}^{k+1}=\frac{x_{j-\frac12}^{k+1}-x_i^k-\Delta b_k(x_i^k)}{\sqrt{\Delta}\sigma_k(x_i^k)}$$
with $x_{j+\frac12}^{k+1}=\frac{x_j^{k+1}+x_{j+1}^{k+1}}{2}$, $x_{\frac12}^{k+1}=-\infty$ and $x_{N_{k+1}-\frac12}^{k+1}=+\infty$.
\paragraph{General setting:}
In order to approximate the transition probabilities and the weights of the Vorono\"i cells when $d>1$, one may proceed with Monte Carlo simulations or rely on Markovian and componentwise product quantization (see $\cite{FiPaSa19}$). A very interesting alternative is the hybrid recursive quantization, studied in Section \ref{hybridQR}, where we replaced the white Gaussian noise by its optimal quantization sequences. The principle on which we rely to design the hybrid recursive quantizers is the same as the one for the standard recursive quantization. The only difference is with the computation of the expectations and probabilities in $(\ref{RBSDE:xj})$,$(\ref{RBSDE:poidsij})$ and $(\ref{RBSDE:poidsj})$. Instead of resorting to large and slow Monte Carlo simulations, we consider sequences of optimal quantizers $(\hat \ve_l^k)_{1\leq l\leq N_{\ve}}$ of size $N_{\ve}$ of the Gaussian distribution $\mathcal{N}(0,I_d)$, available on the quantization website 
\href{http://www.quantize.maths-fi.com}{www.quantize.maths-fi.com}, 
%where, for every $l \in \{1, \ldots, N_Z\}$ and  $k \in \{0, \ldots, n-1\}$, $\hat \ve _l^k=\big(\hat \ve_l^{k, (1)}, \hat \ve_l^{k,(2)}\big)$ 
and compute the sequence and its companion parameters based on the following formulas 
\begin{equation}
\label{espQRH}
\E \Big( \mathcal{E}_k(x_i^k, \ve_{k})\mathds{1}_{\mathcal{E}_k(x_i^k, \ve_{k}) \in C_j(\Gamma_{k+1})}\Big)=\sum_{l=1}^{N_{\ve}} p_{\ve_l}^k \mathcal{E}_k(x_i^k, \hat \ve_l^k ) \mathds{1}_{\mathcal{E}_k(x_i^k, \hat \ve_l^k )  \in C_j(\Gamma_{k+1})}
\end{equation}
and
\begin{equation}
\label{probaQRH}
P\big(\mathcal{E}_k(x_i^k, \ve_{k}) \in C_j(\Gamma_{k+1}) \big)= \sum_{l=1}^{N_{\ve}} p_{\ve_l}^k \mathds{1}_{\mathcal{E}_k(x_i^k, \hat \ve_l^k )  \in C_j(\Gamma_{k+1})}
\end{equation}
where $p_{\ve_l}^k$ is the weight of the Vorono\"i cell of centroid $\hat \ve_{l}^k$, also available on the quantization website.%, for every $l \in \{1, \ldots, N_{\ve}\}$ and every $k \in \{0, \ldots, n-1\}$.
\subsection{Computation of the quantized solution of the RBSDE}
Having already computed the recursive quantization $(\widehat{X}_k)_{0 \leq k \leq n}$ of $(\bar{X}_k)_{0 \leq k \leq n}$ as described in the previous section $\ref{constructionXk}$, as well as the corresponding companion parameters (the weights $(p_i^k)_{1 \leq i \leq N_k}$ of Vorono\"i cells and the transition weights $(p_{ij}^k)_{1 \leq i\leq N_k, 1 \leq j \leq N_{k+1}}$), we proceed with the computation of $(\widehat{Y}_k)_{0 \leq k \leq n}$ and rely on the BDPP $(\ref{ykchapeau})$ allowing us to compute $\widehat{Y}_k=\widehat y_k(\widehat X_k)$ as a function of the quantizer $\Gamma_k=\{x_1^k,\ldots,x_{N_k}^k\}$. For every $k \in \{0,\ldots,n-1\}$ and $i \in \{1,\ldots,N_k\}$, we denote 
$$\widehat{\alpha}_k(x_i^k)=\sum_{j=1}^{N_{k+1}}\widehat{y}_{k+1}(x_j^{k+1}) p_{ij}^k \qquad \mbox{ and } \qquad \widehat{\beta}_k(x_i^k)=\frac{1}{\Delta} \sum_{j=1}^{N_{k+1}}\widehat{y}_{k+1}(x_j^{k+1}) \pi_{ij}^k$$
where
\begin{equation}
\label{piijk}
\pi_{ij}^k =\frac{\sqrt{\Delta}}{p_i^k}\, \E\Big(\ve_{k+1} \mathds{1}_{\{\widehat{X}_{k+1}=x_j^{k+1},\widehat{X}_k=x_i^k\}} \Big) =\sqrt{\Delta}\E\Big(\ve_{k+1} \mathds{1}_{\mathcal{E}_k(x_i^k,\ve_{k+1}) \in C_j(\Gamma_{k+1})} \Big)
\end{equation}
and $\mathcal{E}_k(x,\ve_{k+1})=x+\Delta b_k(x)+\sqrt{\Delta}\sigma_k(x) \ve_{k+1}.$ 
Note that the quantities $(\pi_{ij}^k)_{_{1\leq i,j \leq N_k}}$ are computed online at the same time as the transition weight matrices $(p_{ij}^k)_{_{1 \leq i,j \leq N_k}}$ for every $k \in \{0,\ldots,n-1\}$, so that they can be stored and used instantly in the computations of the solution of the RBSDE. \\

Therefore, the solution $Y_0$ of the RBSDE is approximated by the value $\widehat y_0$ at time $t_0$ of the following recursive quantized scheme 
\begin{equation}
\label{BDPP}
\left\{
\begin{array}{rl}
\widehat{y}_n(x_i^n) &=h_n(x_i^n)\, , \qquad i =1,\ldots,N_n,\\
\widehat{y}_k(x_i^k)&=\max \Big( h_k(x_i^k), \, \widehat{\alpha}_{k}(x_i^k)+\Delta \mathcal{E}_k\big(x_i^k,\widehat{\alpha}_{k}(x_i^k),\widehat{\beta}_{k}(x_i^k)\big) \Big)\, , \qquad i =1,\ldots,N_k,
\end{array}
\right.
\end{equation}
And, the function $\hat z_k$ used to approximate $\hat \zeta_k$ is computed via the following sum
$$\hat z_k (x_i^k)= \frac{1}{\Delta}\sum_{j=1}^{N_{k+1}} \hat y_{k+1}(x_j^{k+1})\pi_{ij}^k. $$
\begin{rmq}
	One should mention that, once the recursive quantization grids and the corresponding companion parameters are computed, the computation of the solution of the RBSDE is almost instantaneous, we can even say that its computational cost is negligible. %It is the recursive quantization procedure that can be expensive, especially in higher dimensions, due to all the integrals that need to be computed.  
\end{rmq}
\section{Numerical examples}
\label{examples}
We carry out some numerical experiments to illustrate the rate of convergence of the recursive quantization-based discretized scheme and to compare its performances with other schemes based on optimal quantization, greedy quantization and greedy recursive quantization. We start by explaining how to obtain the quantizers and their companions parameters (Vorono\"i and transition weights) by optimal, greedy and recursive greedy quantization. Concerning the time discretization, we  consider the Euler scheme of the forward diffusion $(X_t)_{0 \leq t \leq T}$ defined by $(\ref{RBSDE:SDEeuler})$. 
%. Let $(X_t)_{t\in [0,T]}$ be solution to the stochastic differential equation i.e. $(\bar X_k)_{0 \leq k \leq n}$ is the discretized Euler scheme of $(X_t)_{0 \leq t \leq T}$ solution to
%$$dX_t=b(X_t)dt+\sigma(X_t)dW_t$$
%where $W_t$ is a standard Brownian motion. Hence, $\bar X_{k+1}$ is given by
%$$\bar X_{k+1}=\bar X_k+ \Delta b_k(\bar X_k) +\sqrt{\Delta} \sigma_k(\bar X_k) \ve_{k+1}$$
%where $(\ve_{k})_{1 \leq k \leq n}$ is a sequence of i.i.d. r.v. with distribution $\mathcal{N}(0,I_q)$ and $\Delta=\frac{T}{n}$ is the time discretization step. 
\subsection{Various quantization methods}
\subsubsection{Quanization tree with optimal marginal quantization}
\label{QO}
In this section, we aim to build optimal quantizers $\widehat{X}_k^{\Gamma_k}$ of $\bar X_k$ for every $k \in \{0,\ldots,n\}$. At time $t_0$, we start with $\widehat X_0=X_0=x_0 \in \R^d$. Then, at each time step $t_k$, we rely on a sequence of optimal quantizers $(z^k_i)_{1 \leq i \leq N_k}$ of size $N_k$ of the Normal distribution $\mathcal{N}(0,I_d)$ and we compute the quantizer $\Gamma_k=(x_1^k,\ldots,x_{N_k}^k)$ via  
$$x_i^k=x_0+ t_k b(x_0) +\sqrt{t_k} \sigma(x_0) z_i^k\; , \qquad i \in \{1, \ldots,N_k\}.$$
In particular, if $(\bar X_k)_{0, \leq k \leq n}$ evolves following a Black-Scholes model with interest rate $r$ and volatility $\sigma$, then the quantizers are computed as follows
$$x_i^k=x_0\exp\Big((r-\tfrac{\sigma^2}{2})t_k+\sigma \sqrt{t_k}z_i^k\Big).$$
The weights of the Vorono\"i cells are obtained by the forward Kolmogorov equation $(\ref{RBSDE:poidsj})$. In the one-dimensional case, they are easily computed relying on the c.d.f. of the Gaussian distribution.\\

The challenge in this method is the computation of the transition weights $p_{ij}^k$, which are mandatory for our cause. By optimal quantization, $(\widehat X_k)_{0 \leq k \leq n}$ is not a Markov chain so one cannot use its distribution to compute $p_{ij}^k$ like for recursive quanization. One usually compute them by Monte Carlo simulations, but, in the one-dimensional case, there exist some closed formulas. In the following, we present such closed formulas in the case of a Black-Scholes model (the case that interests us the most for our numerical examples), i.e. a case where, for an the interest rate $r$ and a volatility $\sigma$, the process is given by
$$\widehat X_k=\widehat X_0\exp\Big((r-\frac{\sigma^2}{2})t_k+\sigma \sqrt{t_k} \ve_k \Big)$$
where $(\ve_k)_{1 \leq k \leq n}$ is an i.i.d. sequence of random variables with distribution $\mathcal{N}(0,1)$.
\paragraph{Exact computation of the transition weights}
\label{transexact}
Assume that the quantizers $\Gamma_k=(x_i^k)_{1 \leq i \leq N_k}$ of size $N_k$ of $\bar X_k$ are already computed for every $k \in \{1,\ldots,n\}$ and that the sizes of the grids $N_k$, $k=1,\ldots,n$, are all equal to $N\in \N$. Note that this hypothesis is not optimal but turns out to be  optimal in terms of complexity for a given budget $N_1+\cdots+N_n$. It is not sharp in terms of error estimates (up to a multiplicative constant) but remains a good compromise which is convenient in practice for the implementation. The goal is to compute the transition weights
$$p_{ij}^k=\mathbb{P}\Big( \widehat{X}_{{k+1}} =x_{j}^{k+1} \, |\, \widehat{X}_{k}=x_i^k \Big) =\frac{\bar p_{ij}^k}{p_i^k}$$
where $$\bar p_{ij}^k=\mathbb{P}\Big( \widehat{X}_{{k+1}} =x_{j}^{k+1} ,\,  \widehat{X}_{k}=x_i^k \Big) \qquad \mbox{and} \qquad p_i^k=\mathbb{P}\big(\widehat{X}_{k}=x_i^k\big).$$
The weights $p_i^k$ are computed via the forward Kolmogorov equation, using  the transition weights $p_{ij}^k$, as follows 
$$p_j^{k+1}=\sum_{i=1}^{N_k}p_{ij}^k \, p_i^k=\sum_{i=1}^{N_k} \bar p_{ij}^k,$$
keeping in mind that the Vorono\"i weight at time $t_0$ (i.e. $k=0$) is equal to $1$ since $\widehat X_0=X_0=x_0$ is deterministic. So, our main concern is the computation of $\bar p_{ij}^k$ for every $k \in \{1,\ldots,n\}$ and $ i,j \in \{1,\ldots,N\}$.
We start by noticing that
$$\widehat X_{k+1}=\widehat X_k\Big(1+rh+\sigma \sqrt{h} \ve_k \Big)$$
where $h =\frac{T}{n}$ is the time step of the discretization scheme. Note  that highly accurate quantization grids of $\mathcal{N}(0,1)$ for regularly sampled sizes
from $N = 1$ to $1\, 000$ are available and can be downloaded from the quantization website \href{http://www.quantize.maths-fi.com}{www.quantize.maths-fi.com} 
(for non-commercial purposes). Then, considering two independent random variables $z_1$ and $z_2$ with distribution $\mathcal{N}(0,1)$, one has
\begin{align*}
\bar p_{ij}^k &= \mathbb{P}\Big(\widehat X_{k+1} \in \big[x_{j-\frac{1}{2}}^{k+1},x_{j+\frac{1}{2}}^{k+1} \big] \; , \;\widehat X_{k} \in \big[x_{i-\frac{1}{2}}^{k},x_{i+\frac{1}{2}}^{k} \big] \Big) \\
& = \mathbb{P} \left(\widehat X_k(1+rh+\sigma\sqrt{h}z_2) \in C_j(\Gamma_{k+1}), \, z_1 \in \big[\underline{x}_i^k,\, \overline{x}_i^k\big]\right) 
\end{align*}
where \begin{equation}
 \underline{x}_i^k=\frac{\ln\big( x_{i-\frac{1}{2}}^k \big)  +\big(\frac{\sigma^2}{2}-r\big)t_k  -\ln(x_0)}{\sigma \sqrt{t_k}} \quad \mbox { and } \quad  \overline{x}_i^k=\frac{\ln\big( x_{i+\frac{1}{2}}^k \big) + \big(\frac{\sigma^2}{2}-r\big)t_k  -\ln(x_0)}{\sigma \sqrt{t_k}},
 \end{equation}
Then, the independence of $z_1$ and $z_2$ yields
\begin{align}
\label{RBSDE:pijk}
\bar p_{ij}^k &= \int_{\underline{x}_i^k}^{\overline{x}_i^k} \mathbb{P} \Big( x_0(1+rh+\sigma\sqrt{h}z_2) \, {\rm exp}\Big((r-\tfrac{\sigma^2}{2})t_k+\sigma \sqrt{t_k}z\Big) \in \Big[x_{j-\frac{1}{2}}^{k+1},x_{j+\frac{1}{2}}^{k+1} \Big]  \Big)\, e^{-\frac{z^2}{2}} \frac{dz}{\sqrt{2\pi}} \nonumber \\
&= \int_{\underline{x}_i^k}^{\overline{x}_i^k} \mathbb{P} \left( z_2 \in \left[  \frac{x_{j-\frac{1}{2}}^{k+1} e^{(\frac{\sigma^2}{2}-r)t_k-\sigma\sqrt{t_k}z } -x_0-rhx_0}{\sigma x_0 \sqrt{h}} \; , \; \frac{x_{j+\frac{1}{2}}^{k+1}e^{(\frac{\sigma^2}{2}-r)t_k-\sigma\sqrt{t_k}z } -x_0-rhx_0}{\sigma x_0 \sqrt{h}} \right]\right) e^{-\frac{z^2}{2}} \frac{dz}{\sqrt{2\pi}} \nonumber \\
&= \int_{\underline{x}_i^k}^{\overline{x}_i^k} \left(\Phi_0(\overline{x}_j^{k+1})-\Phi_0(\underline{x}_j^{k+1}) \right)e^{-\frac{z^2}{2}} \frac{dz}{\sqrt{2\pi}},
\end{align}
where
 \begin{equation}
 \label{xdandxu}
 \underline{x}_j^{k+1}= \frac{x_{j-\frac{1}{2}}^{k+1} e^{(\frac{\sigma^2}{2}-r)t_k-\sigma\sqrt{t_k}z } -x_0-rhx_0}{\sigma x_0\sqrt{h}}\quad \mbox{ and} \quad \overline{x}_j^{k+1}= \frac{x_{j+\frac{1}{2}}^{k+1} e^{(\frac{\sigma^2}{2}-r)t_k-\sigma\sqrt{t_k}z } -x_0-rhx_0}{\sigma x_0\sqrt{h}}.
 \end{equation}
These integrals can be computed via Gaussian quadrature formulas, mainly Gauss-Legendre quadrature formulas for integrals on closed intervals and Gauss-Laguerre quadrature formulas for integrals on semi-closed intervals. So, if $i=1$ or $i=N$, one uses Gauss-Laguerre formulas since the Vorono\"i cells (over which we are integrating) are of the form $(-\infty,a)$ or $(a, +\infty)$ for some $a \in \R$. Otherwise, the Vorono\"i cells are closed intervals so one relies on Gauss-Legendre quadrature formula. 
Let us detail these computations.\medskip \\
$\rhd$ {\sc Integration on a closed interval $[a,b]$: Gauss Legendre fomula}\\
	Considering $\ds f(z)=\Big(\Phi_0(\overline{x}_j^{k+1})-\Phi_0(\underline{x}_j^{k+1})
	\Big) \frac{e^{-\frac{z^2}{2}}}{\sqrt{2\pi}}$, $a=\underline{x}_i^k$ and $b=\overline{x}_i^k$, the goal is to compute $I=\int_a^b f(z)dz$. Applying the change of variables $z=\frac{b-a}{2}x+\frac{a+b}{2}$, $I$ can be written and computed as follows 
	$$I=\frac{b-a}{2}\int_{-1}^1 f\left(\frac{b-a}{2}x+\frac{a+b}{2}\right)dx=\frac{b-a}{2}\sum_{i=1}^n w_i f\left(\frac{b-a}{2}x_i+\frac{a+b}{2}\right)$$
	where $(x_i)_{1 \leq i \leq n}$ are the roots of the $n^{\rm th}$ Legendre polynomial $ P_n(x)=\frac{1}{2^n} \sum_{k=0}^{\lfloor\frac{n}{2} \rfloor} (-1)^k \frac{(2n-2k)!}{k! (n-k)! (n-2k)!}x^{n-2k} $ and the weights $(w_i)_{1 \leq i \leq n}$ are given by
	$$w_i=\frac{2}{(1-x_i^2)P'_n(x_i)^2}=\frac{2(1-x_i^2)}{(n+1)^2P_{n+1}(x_i)^2}.$$
	$\rhd$ {\sc Integration on intervals of the form $[a,+\infty)$ or $(-\infty, a]$: Gauss Laguerre quadrature}\smallskip \\
	We consider $f(z)=\Phi_0(\overline{x}_j^{k+1})-\Phi_0(\underline{x}_j^{k+1})$ and distinguish two cases.\smallskip\\
	$\bullet$ {\it Integration on $[a,+\infty)$}\\
		The goal is to compute $I=\int_a^{+\infty} f(z) e^{-\frac{z^2}{2}}dz$
		where $a=\underline{x}_i^k$. Applying the change of variables $x=\frac{z^2}{2}$ and denoting $g(x)=\frac{f(x)}{x}$ yield
		\begin{align*}
		I = \int_{\frac{a^2}{2}}^{+\infty} \frac{f(\sqrt{2x})}{\sqrt{2x}} e^{-x}dx = \int_{\frac{a^2}{2}}^{+\infty} g(\sqrt{2x}) e^{-x}dx = e^{-\frac{a^2}{2}} \int_0^{+\infty} g\left(\sqrt{2x+a^2} \right)e^{-x}dx
		\end{align*}
		where we applied in the last equality the change of variables $y=x-\frac{a^2}{2}$. Hence, we use Gauss-Legendre quadrature formula to obtain
		$$I=e^{-\frac{a^2}{2}} \sum_{i=1}^N  w_i g\left(\sqrt{2x_i+a^2} \right) $$
		where $(x_i)_{1 \leq i \leq n}$ are the roots of the $n^{\rm th}$ Laguerre polynomial
		%\begin{equation}
		%\label{laguerrepoly}
		$L_n(x)=\sum_{k=0}^n (-1)^k \frac{n!}{k! (n-k)!^2}x^k$
		%\end{equation}
		and the weights $(w_i)_{ 1 \leq i \leq n}$ are given by
		\begin{equation}
		\label{weightslaguerre}
		w_i=\frac{1}{(n+1)L'_n(x_i) L_{n+1}(x_i)}=\frac{x_i}{(n+1)^2 L_{n+1}(x_i)^2}.
		\end{equation}
		$\bullet$ {\it Integration on $(-\infty, a]$}\\
		The goal is to compute $I=\int_{-\infty}^a f(x) e^{-\frac{x^2}{2}}dx$
		where $a=\overline{x}_i^k$.
		Similarly to the previous case, $I$ can be written as follows
		\begin{align*}
		I=\int_{-a}^{+\infty} f(-x) e^{-\frac{x^2}{2}}dx= \int_{\frac{a^2}{2}}^{+\infty} \frac{f(-\sqrt{2z})}{\sqrt{2z}} e^{-z}dz = \int_{\frac{a^2}{2}}^{+\infty} g(\sqrt{2z}) e^{-z}dz= e^{-\frac{a^2}{2}} \int_0^{+\infty} g\left(\sqrt{2z+a^2} \right)e^{-z}dz
		\end{align*}
		where $g(x)=\frac{f(-x)}{x}$. Hence, Gauss-Legendre quadrature formula yields
		$$I=e^{-\frac{a^2}{2}} \sum_{i=1}^N  w_i g\left(\sqrt{2x_i+a^2} \right) $$
		where $(x_i)_{ 1 \leq i \leq n}$ are the roots of $L_n(x)$ and $(w_i)_{ 1 \leq i \leq n}$ are given by $(\ref{weightslaguerre})$.
\paragraph{Approximation of the transition weights}
If the goal is not necessarily the highest level of precision, then one approximates the transition weights $p_{ij}^k$ by $g_j(z_i^k)$ where the function $g_j(z)$ is defined by
\begin{equation}
\label{transBSapprox}
g_j(z)=\Phi_0(\overline{x}_j^{k+1})-\Phi_0(\underline{x}_j^{k+1}).
\end{equation}
and 
$\overline{x}_j^{k+1}$ and $\underline{x}_j^{k+1}$ are given by $(\ref{xdandxu})$.
In fact, based on $(\ref{RBSDE:pijk})$ and then applying Taylor-Lagrange formula, one has 
\begin{align*}
\bar p_{ij}^k &= \int_{z_{i-\frac{1}{2}}^k}^{z_{i+\frac{1}{2}}^k} g_j(z)e^{-\frac{z^2}{2}}\frac{dz}{\sqrt{2\pi}}\\
&=g_j(z_i^k)p_i^k+g'_j(z_i^k )\int_{z_{i-\frac{1}{2}}^k}^{z_{i+\frac{1}{2}}^k} (z-z_i^k)e^{-\frac{z^2}{2}}\frac{dz}{\sqrt{2\pi}} + \int_{z_{i-\frac{1}{2}}^k}^{z_{i+\frac{1}{2}}^k} g_j''(\xi(z))\frac{(z-z_i^k)^2}{2}e^{-\frac{z^2}{2}}\frac{dz}{\sqrt{2\pi}}.
\end{align*}
Since $(z_i^k)_{1 \leq i\leq N}$ is a quadratic optimal quantization sequence of the standard Normal distribution, then it is stationary and the second term of the above inequality is equal to $0$. Moreover,
$$g_j'(z)=\frac{k}{x_0\sqrt{2\pi}} \Big( x_{j+\frac12}e^{-\sigma\sqrt{t_k}z-\frac12\bar x_j^2}-x_{j-\frac12} e^{-\sigma\sqrt{t_k}z-\frac12\underline x_j^2}\Big) $$
and
\begin{align*}
g_j''(z)=&\frac{k}{x_0\sqrt{2\pi}} \left[ \sigma \sqrt{t_k}e^{-\sigma \sqrt{t_k}z}\Big(x_{j-\frac12}e^{-\frac12\underline x_j^2}-x_{j+\frac12}e^{-\frac12\overline x_j^2}\Big)+\frac{k}{x_0}e^{-2\sigma \sqrt{t_k}z}\Big(x_{j+\frac12}^2\overline x_j^2e^{-\frac12\overline x_j^2}-x_{j-\frac12}^2\underline x_j^2e^{-\frac12\underline x_j^2}\Big)\right]. 
\end{align*}
At this stage, one notices that $ \gamma(z):=\exp(-2z-\frac12e^{-2z}) \leq \kappa$ for every $z \in \R$ for some finite positive constant $\kappa$ and that 
%$C_{\gamma}=\sup_{z \in \R} |\gamma'(z)| <+\infty$ so that
%$$|g_j''(z)| \leq \frac{k}{\sqrt{2\pi}} \left(C_{\gamma} \sqrt{k+1} \big(z_{j+\frac{1}{2}}^{k+1}-z_{j-\frac{1}{2}}^{k+1}\big) \wedge \frac{2}{\sqrt{e}}\right).$$
$|g_j''(z)| \leq \bar \kappa$ for a finite positive constant $\bar \kappa$. Consequently, $\left| p_{ij}^k-g_j(z_i^k)  \right|$ is bounded.\\
%$$\left| p_{ij}^k-g_j(z_i^k)  \right| \leq \frac{k}{\sqrt{2\pi}} \left(C_{\gamma} \sqrt{k+1} \frac{z_{j+1}^{k+1}-z_{j-1}^{k+1}}{2} \wedge \frac{1}{\sqrt{e}}\right) .$$

It is important to note that when we estimate the transition weight by $g_j(z_i^k)$, we formally get the transition weight from $x_i^k$ to $x_{j}^{k+1}$ obtained by recursive quantization, even though they are not the same grids.
\begin{rmq}
For local volatility models (CEV models for example), it becomes more complicated to establish such closed formulas for the computations of the transition matrix. One tends to approximate them by Monte Carlo simulations, for example. 
\end{rmq}
\subsubsection{Greedy quantization}
\label{greedy}
Another technique is greedy vector quantization introduced in \cite{LuPa15} and developed in \cite{papiergreedy}. It consists in building a {\it sequence} of points $(a_n)_{n \geq 1}$ in $\R^d$ recursively optimal step by step, in the following {\em greedy} sense: having computed the first $n$ points $a_1, \ldots, a_n$ of the sequence and defining the resulting grid $a^{(n)}=\{a_1,\ldots,a_n\}$ for $n\geq 1$, we compute the $(n+1)$-th point as a solution to the minimization problem
\begin{equation}
\label{greedydef}
\qquad a_{n+1} \in \mbox{argmin}_{\xi \in \R^d} \, e_p(a^{(n)} \cup \{\xi\}, X),
\end{equation}  
with the convention $a^{(0)} = \varnothing$. Quadratic greedy quantization sequences are obtained by implementing "freezing" avatars of usual stochastic optimization algorithms used for optimal quantization, these variants are exposed in details in \cite{LuPa215}. In this paragraph, we give a quick idea on the computation of the greedy quantization sequence of $(\bar X_k)_{0 \leq k \leq n}$. 
Starting at $\widehat X_0=\bar X_0=x_0$, the process $\bar X_k$ can be written, for every $k \in \{1,\ldots,n\}$, as follows
$$\bar X_k = x_0+t_k b(x_0)+ \sqrt{t_k} \sigma(x_0) \ve_{k}$$
where $\ve_k$ is a random variable with distribution $\mathcal{N}(0,I_q)$.  So $\bar X_k$ is with Normal distribution $ \mathcal{N}(m_k,\Sigma_k)$
where $m_k=x_0+t_k b(x_0)$ and $\Sigma_k=\sqrt{t_k} \sigma(x_0)$ and hence this is the distribution that needs to be discretized by greedy quantization. The transition weights in the one-dimensional case are computed via Gaussian quadrature formula like explained for the optimal quantization, and the weights of the Vorono\"i cells by the forward Kolmogorov equation.\\

In the high-dimensional framework ($d>1$), the computations become too demanding. So, instead of designing pure greedy quantization sequences, one tends to build greedy product quantization sequences which are obtained as a result of the tensor product of one-dimensional sequences, when the target law is a tensor product of its independent marginal laws. We refer to \cite{papiergreedy} for further details.
\subsubsection{Greedy recursive quantization}
\label{GRQ}
In the algorithm described in Section $\ref{algorithmics}$, the recursive quantization scheme $(\ref{RBSDE:quantifrecursive})$ is based on an optimal quantization of the sequences $(\widetilde X_k)_{0 \leq k \leq n}$ at each time step $t_k$. Here, we consider, as an alternative, greedy optimal quantization grids $\widehat X_k$ of $\widetilde X_k$. They are designed as follows: At time $t_{k+1}$, assuming that the $N_k$-tuple $(x_1^k, \ldots,x_{N_k}^k)$ and its companion parameters are already computed, one needs to build, step by step by greedy quantization, the $N_{k+1}$-tuple $(x_1^{k+1},\ldots, x_{N_{k+1}}^{k+1})$ which approaches best $\widetilde X_{k+1}=\mathcal{E}_k(\widehat X_k, \ve_{k+1})$. Since $\mathcal{E}_{k}(x_i^k,\ve_{k+1}) \sim \mathcal{N}(m_i^k,\Sigma_i^k)$ with $m_i^k=x_i^k+\Delta b_k(x_i^k)$ and $\Sigma_i^k=\sqrt{\Delta} \sigma_k(x_i^k)$, the first point of the sequence is $x_1^{k+1}= \E\big[ \widehat X_k+\Delta b_k(\widehat X_k)\big]=\sum_{i=1}^{N_k}p_i^k \big(x_i^k+\Delta b_k(x_i^k)\big)$ and then, at each iteration $N$, $N\in \{2,\ldots,N_{k+1}\}$, one adds one point $x_{N}^{k+1}$ following the steps of the greedy variant of Lloyd's algorithm detailed in $\cite{LuPa215}$. One should take in consideration that the local interpoint inertia are computed, at each time step $t_{k+1}$, by 
\begin{equation}
	\label{inertie}
	\sigma_j^2=\sum_{i=1}^{N_k}p_i^k \left( \int_{x_{j}^{k+1,N}}^{x_{j+\frac{1}{2}}^{k+1,N}} \big(\xi-x_j^{k+1,N}\big)^2 P(d\xi) +\int_{x_{j+\frac{1}{2}}^{k+1,N}}^{x_{j+1}^{k+1,N}} \big(\xi-x_{j+1}^{k+1,N}\big)^2 P(d\xi) \right) \; := \sum_{i=1}^{N_k}p_i^k s_{ij}
	\end{equation}
	where $x_{j+\frac{1}{2}}^{k+1,N}=\frac{x_j^{k+1,N}+x_{j+1}^{k+1,N}}{2}$ with $x_0^{k+1,N}=x_{\frac{1}{2}}^{k+1,N}=-\infty$ and $x_N^{k+1,N}=x_{N-\frac{1}{2}}^{k+1,N}=+\infty$. Likewise, the recurrence of the algorithm is given by
	\begin{equation}
	\label{suitegreedy}
	x_{\ell+1}=\frac{\sum_{i=1}^{N_k}p_i^k \E\Big(\mathcal{E}_k(x_i^k,\ve_{k+1}) \mathds{1}_{\big\{\mathcal{E}_k(x_i^k,\ve_{k+1})\in C_j(\Gamma_{k+1})\big\}} \Big)}{\sum_{i=1}^{N_k} p_i^k  \P\Big( \mathcal{E}_k(x_i^k,\ve_{k+1})\in C_j(\Gamma_{k+1})\Big)},
	\end{equation} 
	The companion parameters are computed following the same principle as for the standard recrusive quantization.
\subsection{Examples}
\subsubsection{American call option in a market with bid-ask spread on interest rates}
\label{bidask}
We are interested in the valuation of an American call option with maturity $T$ in a market with a bid-ask spread on interest rates with a borrowing rate $R$ and a lending rate $r \leq R$. The stock price is represented by the process $(X_t)_{t \in [0,T]}$ given by the SDE $(\ref{SDE})$
and the dynamics of the portfolio are given by 
$$-dY_t=\left(-rY_t-\frac{b_t(X_t)-r}{\sigma_t(X_t)}Z_t -(R-r)\min\Big(Y_t-\frac{Z_t}{\sigma_t(X_t)},0 \Big) \right)dt-Z_tdW_t$$
$$Y_T=h(X_T) \qquad \mbox{ and } \qquad Y_t\geq g(X_t)$$
where $h(x)=g(x)=\max (x-K,0)$, $K$ being the strike price. 
\paragraph{Black-Scholes model}
We consider that $(X_t)_{t \in [0,T]}$ evolves following the Black-Scholes dynamics and is time discretized following the Euler scheme, i.e. for every $k \in \{0,\ldots,n-1\}$,
\begin{equation}
\label{BSeuler}
\bar X_{k+1}=\bar X_k +\mu \Delta \bar X_k +\sigma \sqrt{\Delta} \bar X_k\, \ve_{k+1}
\end{equation}
where $\mu$ is the drift and $\sigma$ is the volatility. The space discretization is established via recursive quantization (RQ), optimal quantization (OQ), greedy quantization (GQ) and greedy recursive quantization (GRQ). We consider $n=20$ time steps and build corresponding quantization grids of size $N=100$ and their companion parameters as explained in the different sections previously in the paper. %The transition weights $p_{ij}^k$ are given by the distribution of $(\widehat X_k)_{1 \leq k \leq n}$ for the recursive and the greedy recursive quantization and computed exactly via Gauss-Laguerre and Gauss-Legendre quadrature formulas for optimal and greedy quantization, as detailed in $\ref{transexact}$. 
Then, we rely on the backward recursion $(\ref{BDPP})$ to compute the value $Y_0$ of the underlying option. Note that the quantities $\pi_{ij}^k$ are computed, for every $k \in \{1, \ldots,n\}$, as a companion parameter with the diffusion $\widehat X_k$ via a Monte Carlo simulation of size $10^6$. We consider the following parameters
$$X_0=100 \, , \quad T=0.25\, , \quad \sigma =0.2\, , \quad \mu=0.05\, , \quad r=0.01\, , \quad R=0.06$$
and we compare the values obtained by the different methods for different values of $K$ varying between $100$ and $120$. As a benchmark, we will assume that the optimal quantization converges to the exact value and, under this hypothesis, we consider the fastest and most accurate version of optimal quantization, which is the quantization-based Richardson-Romberg extrapolation. The idea is the following:\\
If the goal is to approximate $\E f(X)$ for a function $f$ and a random variable $X$, one considers two optimal quantization sequences $\widehat{X}^{N_1}$ of size $N_1$ and $\widehat{X}^{N_2}$ of size $N_2$ of the random variable $X$ and hence $\E f(X)$ is given by
\begin{equation}
\label{romberg}\E f(X)=\frac{N_2^2 \E f(\widehat{X}^{N_2})-N_1^2 \E f(\widehat{X}^{N_1})}{N_2^2-N_1^2}.
\end{equation}
From a practical point of view, one usually considers $N_1=N$ and $N_2=\frac{N}{2}$. Furthermore, when the dimension $d=1$, the standard quantization error is of the form 
$$e_2(X,\mu) \approx c_1 \sqrt{n} +c_2 \sqrt{n} N^{-1}$$
and the Romberg-quantization error is of the form
$$e_2(X,\mu)\approx c_2 \sqrt{n} \left(\frac{1}{N_1}-\frac{1}{N_2} \right) \approx \frac{c_1\sqrt{n}}{2N_1}.$$
So, by studying the values of this error for different values of $n$ and $N_1$, we realize that the best technique is to consider a small number of time steps $n$ and a large size $N$ of the quantizer.\\

In our example, we consider an optimal quantization-based Richardson Romberg extrapolation with $n=5$ and $N=1\, 000$. We observe in Table $\ref{bidaskBSeulerprice}$ the results and the errors obtained by the various methods.
\begin{table}[h]
	\begin{center}
		\begin{tabular}{l|llllllll|l}
			\hline
			{\bf $K$} & {\bf RQ}  & & {\bf GRQ}  & & {\bf OQ} & & {\bf GQ} & & {\bf Romberg}  \\      
			\hline
			& { Value} & { Error} & Value & Error & Value & Error & Value & Error& \\
			\hline
			$100$   & $4.719$ & $0.026$ & $4.728$ & $0.017$ &  $4.747$ &  $0.002$ & $4.704 $& $0.041 $  & $4.745$    \\
			%\hline
			$105$   & $2.538$ & $0.012$ & $2.548$ & $0.002$ &  $2.561$  &  $0.011$ &   $2.529$ &   $0.021$ & $2.55 $ \\
			%\hline
			$110$  & $1.222$ & $0.003$  & $1.225$ & $0.006$ & $1.234$  & $0.015$ &   $1.212$ &   $0.007$ & $1.219$\\
			%\hline 
			$115$  & $0.526$ & $0.008$ & $0.526$ & $0.008$ & $0.532$ & $0.014$  & $0.518$ & $0$& $0.518$\\
			%\hline 
			$120$  &  $0.203$ &  $0.007$ &  $0.202$ &  $0.006$ &  $0.206$ &  $0.01$ & $0.198$ & $0.002$ & $0.196$\\
			\hline 
			{Average} & & $0.0112$ & & $0.0078$ & & $0.0104$ & & $0.0142$ &\\
			\hline
		\end{tabular}
	\end{center}
	\vspace{-0.5cm}
	\caption{Pricing of an American call option in a market with bid-ask spread for interest rates in a Black-Scholes model by recursive (RQ), greedy recursive (GRQ), optimal (OQ) and greedy (GQ) quantization.} 
	\label{bidaskBSeulerprice}
\end{table}
Here, we emphasize on the computational time of these simulations which are performed on a CPU $2.7$ GHz and $8$ GB memory computer. The optimal quantizer and its companion parameters are obtained in about 40 seconds while the greedy quantization sequence and its companions in about 30 seconds. This is approximately a $25 \%$ gain in time in favor of greedy quantization whose results are comparable (a little less precise) than optimal quantization. As for the recursive quantization, the standard simulations (RQ) are obtained in about $2.3$ minutes and the greedy simulations (GRQ) in about $2$ minutes. Hence, the greedy character introduced in the recursive algorithm brings a $13 \%$ gain in time. The additional cost in time is compensated by the preservation of the Markovian property and the precision of the results.\\% Note that similar  observations and conclusions are made in the following one-dimensional examples and models.\\
 Figure $\ref{ODCbidaskBS}$ depicts the convergence of the error induced by the approximation of $Y_0$ based on a recursive quantization of the forward process $\bar X_k$. For this illustration, we consider a strike $K=100$ and we make the size $N$ of the grids vary between $10$ and $100$. The graph is represented in a $\log$-$\log$-scale scale and an $\mathcal{O}(N^{-1})$ rate of convergence is clearly observed. 
 \begin{figure}
		\begin{center}
			\includegraphics[width=9cm,height=7cm,angle=0]{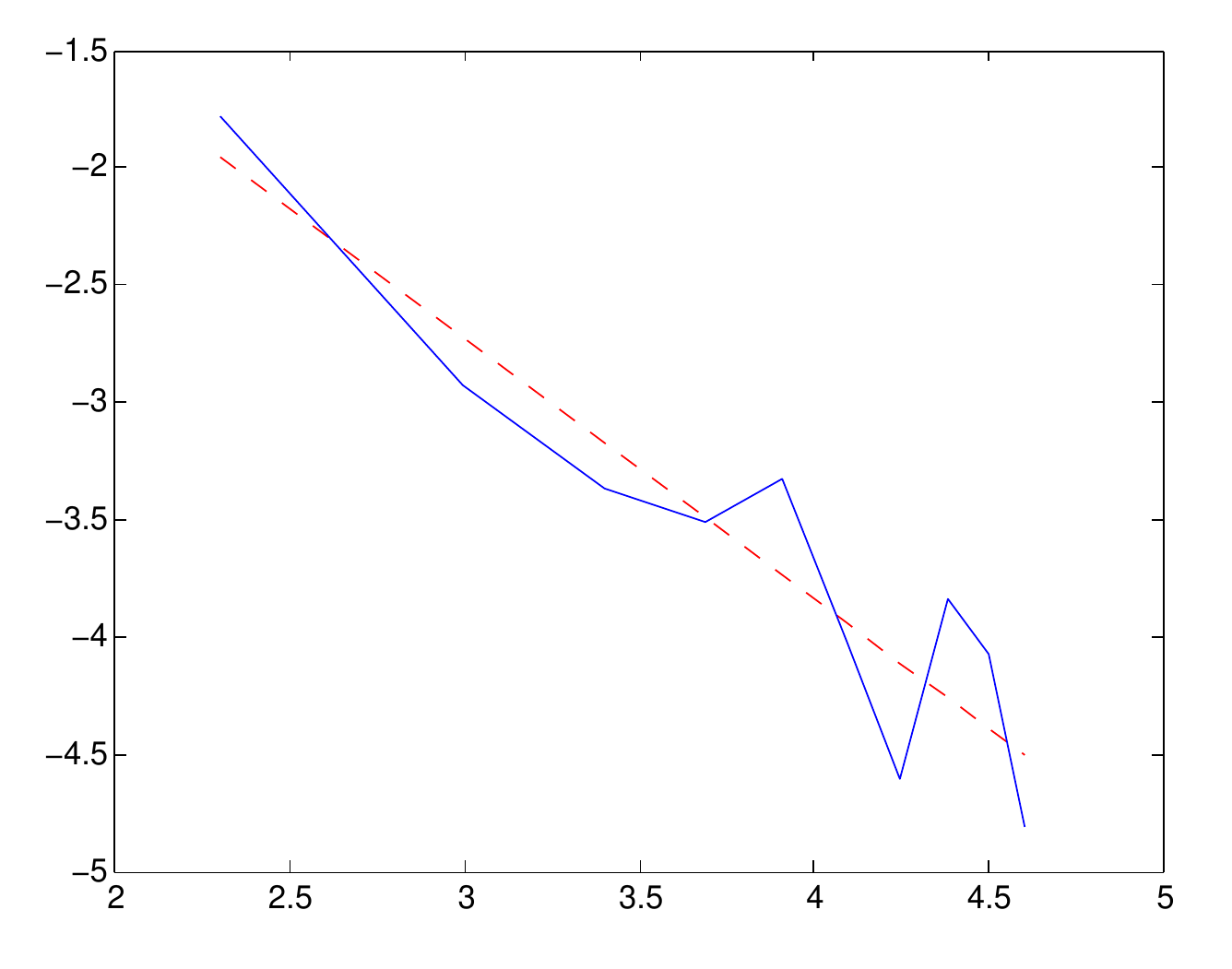}
		\end{center} 
		\vspace{-1cm}
		\caption{Convergence rate of the error induced by the approximation of the Bid-ask spread Call option in a Black-Scholes model discretized by recursive quantization for different sizes $N=10,\ldots,100$. (logarithmic scale)}
		\label{ODCbidaskBS}
	\end{figure}
	\paragraph{CEV model} 
We consider a local volatility model, the CEV model, in which $(X_t)_{0 \leq t \leq T}$ evolves following
\begin{equation}
\label{pCEV}
%dX_t=\mu X_t dt+\vartheta \frac{X_t^{1+\delta}}{\sqrt{\varepsilon+X_t^2}}dW_t, \qquad X_0=x_0,
dX_t=\mu X_t dt+\vartheta X_t^{\delta}dW_t, \qquad X_0=x_0,
\end{equation}
for some $\delta \in (0,1)$ and $\vartheta \in (0,\overline{\vartheta}]$ with $\overline{\vartheta}>0$. $\sigma(x)=\vartheta x^{\delta}$ is the local volatility function. 
%Note that this model becomes close to the CEV model for certain values of $\varepsilon$ and $X_t$. 
The discretized Euler scheme associated to $(X_t)_{t \in [0,T]}$ is given, for every $k \in \{0,\ldots,n-1\}$, by
\begin{equation}
\label{pCEVeuler}
%\bar X_{k+1}=\bar X_{k}+\mu \Delta \bar X_{k}+\vartheta \frac{\bar X_{k}^{1+\delta}}{\sqrt{\varepsilon+\bar X_{k}^2}} \sqrt{\Delta}\, \ve_k
\bar X_{k+1}=\bar X_{k}+\mu \Delta \bar X_{k}+\vartheta \bar X_{k}^{\delta} \sqrt{\Delta}\, \ve_k
\end{equation}
where $(\ve_k)_{1 \leq k \leq n}$ is an i.i.d sequence of random variables with distribution $\mathcal{N}(0,1)$. \\
%In practice, if $X_t$ is large enough, then one can choose $\varepsilon=1$ since $ \frac{X_t}{\sqrt{1+X_t^2}} \rightarrow 1$ which is very close to the CEV model.\\
The construction of the quantizers and the computation of the companion parameters by recursive and greedy recursive quantization is similar to what was done for the Black-Scholes model. As for optimal and greedy quantization, closed forms for the companion parameters are no longer available in this model, we estimate them by Monte Carlo simulations of size $10^5$ coupled with a nearest neighbor search. We build corresponding quantization grids of size $N=150$ and consider $n=15$ time steps. The parameters are the following
$$X_0=100 \, , \quad T=0.25\, , \quad \vartheta =4\, , \quad \delta=0.5\, , \quad \varepsilon=1\, , \quad \mu=0.05\, , \quad r=0.01\, , \quad R=0.06$$
and we compare the values obtained by the different methods for different values of $K$ between $100$ and $120$. The benchmark is given by an optimal quantization-based Richardson-Romberg extrapolation $(\ref{romberg})$. We observe in Table $\ref{bidaskpCEVeulerprice}$ the results and errors obtained by such comparisons. As for the computation time, we note that the optimal quantizer and its companion parameters are obtained in about $100$ seconds while the greedy quantization sequence and its companions in about $70$ seconds. The fact that these computations take more time for the CEV model than for the Black-Scholes model is due to the non-existence of closed formulas for the computation of the companion parameters in the CEV model, the computation of the quantizers themselves is almost instantaneous. Moreover, the recursive quantizer and its companions are computed in about $3.5$ minutes while the greedy recursive quantizers in about $3$ minutes.
\begin{table}[h]
	\begin{center}
		\begin{tabular}{l|llllllll|l}
			\hline
			{\bf $K$} & {\bf RQ}  & & {\bf GRQ}  & & {\bf OQ} & & {\bf GQ} & & {\bf Romberg}  \\      
			\hline
			& { Value} & { Error} & Value & Error & Value & Error & Value & Error& \\
			\hline
			$100$   & $8.517$ & $0.074$ & $8.524$ & $0.067$ &  $8.536$ &  $0.055$ & $8.593 $& $0.002 $  & $8.591$    \\
			%\hline
			$105$   & $6.262$ & $0.049$ & $6.272$ & $0.039$ &  $6.288$  &  $0.023$ &   $6.321$ &   $0.01$ & $6.311 $ \\
			%\hline
			$110$  & $4.479$ & $0.023$  & $4.483$ & $0.019$ & $4.498$  & $0.004$ &   $4.522$ &   $0.02$ & $4.502$\\
			%\hline 
			$115$  & $3.11$ & $0.006$ & $3.113$ & $0.003$ & $3.125$ & $0.009$  & $3.128$ & $0.012$& $3.116$\\
			%\hline 
			$120$  &  $2.094$ &  $0.003$ &  $2.1$ &  $0.009$ &  $2.109$ &  $0.018$ & $2.103$ & $0.012$ & $2.091$\\
			\hline 
			{Average} & & $0.031$ & & $0.0274$ & & $0.0218$ & & $0.0112$ &\\
			\hline
		\end{tabular}
	\end{center}
	\vspace{-0.5cm}
	\caption{Pricing of an American call option in a market with bid-ask spread for interest rates in a CEV model by recursive (RQ), greedy recursive (GRQ), optimal (OQ) and greedy (GQ) quantization.}
	\label{bidaskpCEVeulerprice}
\end{table}
\subsubsection{Two-dimensional American exchange options}
\label{exchange}
We are interested in pricing an American exchange option with exchange rate $\mu$ and maturity $T$. This price is given by the value $Y_0$ at time $t_0$ of the solution of the RBSDE $(\ref{BSDE})$ with driver $f=0$ and $h_t(x)=g_t(x)=\max\big(e^{-\lambda t} X^1_t -\mu X^2_t ,0\big)$. $X^1_t$ and $X^2_t$ are two assets, such that $X^1_t$ is with a geometric dividend rate $\lambda$ and $X^2_t$ is without dividend, both following a Black-Scholes model. The discretized Euler scheme $(\bar X_k^1,\bar X_k^2)$ is given, for every $k \in \{0, \ldots, n-1 \}$, by
\[
\begin{array}{rl}
\bar X_{k+1}^1&=\bar X_k^1 e^{(r-\frac{\sigma^2}{2})\Delta +\sigma \sqrt{\Delta} \ve_k^1 }\\
\bar X_{k+1}^2 &=\bar X_k^2 e^{(r-\frac{\sigma^2}{2})\Delta +\sigma \sqrt{\Delta} (\rho \ve_k^1+\sqrt{1-\rho^2} \ve_k^2)}
\end{array}  
\]
where $r$ is the interest rate, $\sigma$ the volatility, $\rho$ is a correlation coefficient and $(\ve^1_k,\ve^2_k)_{1 \leq k \leq n}$ is a sequence of i.i.d. random variables with distribution $\mathcal{N}(0,I_2)$. \\
From a numerical point of view, we discretize in $n=10$ time steps, build quantizers of size $N_X=100$ and consider the following parameters
$$X_0^1=40\, , \quad T=1\, , \quad r=0\, , \quad \sigma =0.2\, , \quad \lambda =0.05 \, , \quad \mu=1 \, .$$
In high dimensions ($d>1$),  the implementation of the recursive quantization algorithm is too expensive and its cost in time is very high. We consider, instead, the hybrid recursive quantization, introduced in Section \ref{hybridQR} and use sequences of optimal quantizers $(\hat \ve_l^k)_{1\leq l\leq N_{\ve}}$ of size $N^{\ve}= 1000$ to compute the sequence and the companion parameters as detailed in Section \ref{algorithmics}. We also build optimal quantizers and greedy product quantization sequences (see Section \ref{greedy}). We compute the price of the option by these methods for $X_0^2\in \{36;\,44\}$ and $\rho\in \{-0.8; \, 0;\, 0.8\}$ and compare the results obtained to those computed by a finite difference algorithm in \cite{ViZa02} and expose the errors hence induced in Table \ref{exchange2}.\\
\begin{table}[h]
	\begin{center}
	\begin{tabular}{l|l|llllll|l}
		\hline
		{\bf  $X_0^2$} &  {$ \rho$}  &    {\bf OQ}   &&  {\bf HRQ} & & {\bf GPQ} & & {\bf Benchmark}  \\      
		\hline
		&&  Value & Error & Value & Error & Value & Error &\\
		\hline
		$36$   & $-0.8$   & $7.062$  & $0.087$  &  $6.979$  & $0.004 $  &  $6.926$ & $0.049$ & $6.975$    \\
		%\hline
		$36$   & $0$   & $5.832$ & $0.186$  &  $5.706$   &    $0.06$  &  $5.763$ & $0.117$ & $5.646 $ \\
		%\hline
		$36$  & $0.8$  &  $4.076$   & $0.076$ & $4.008$  &   $0.008$ & $4$ & $0$ & $4$\\
		\hline 
		Average error & & & $0.116$ & & $0.024$ & & $0.055$ & \\
		\hline
		$44$  & $-0.8$ & $3.834$ & $ 0.065$  & $3.741$    & $0.028$ &  $3.609$ & $0.16$ & $3.769$\\
		%\hline 
		$44$  &  $0$  &  $2.453$ & $ 0.117$  &  $2.329$   &  $0.007$ &  $2.042$ & $0.294$ & $2.336 $\\
		%\hline 
		$44$  & $0.8$   &$0.426$ & $0.067$  & $0.282$ &   $0.077$ &  $0.401$ & $0.042$ & $0.359$ \\
		\hline
		Average error & & & $0.083$ & & $0.037$ &  & $0.165$ & \\
		\hline
	\end{tabular}
	\end{center}
	\vspace{-0.5cm}
	\caption{Pricing of an American exchange option for $d=2$ in a BS model by hybrid recursive (HRQ), optimal (OQ) and greedy product quantization (GPQ).} 
	\label{exchange2}
\end{table} 
Similarly to the one-dimensional Example \ref{bidask}, a gain in the computation time appears in favor of the greedy quantization. In fact, greedy product quantization sequences are obtained in about $55$ seconds whereas optimal  and hybrid recursive quantizers in about $70$ seconds and $3.75$ minutes respectively, and hence the gain is about $20 \%$ compared to optimal quantization and  $75 \%$ compared to hybrid recursive quantization. Moreover, we remark that hybrid recursive quantization gives the most precise results while an expected gain in precision for optimal quantization compared to greedy quantization is observed.

\section{Appendix }
\subsection{Appendix A: The proof of Lemma $\ref{lemme1}$}
First note that the function $f: u\mapsto |u|^r$ satisfies (since $r\ge 2$)
\[
\nabla |u|^r = r  |u|^{r-1} \frac{u}{|u|}\quad \mbox{ and }\quad  \nabla^2|u|^r =  r|u|^{r-2}  \left((r-2)  \frac{u}{|u|}  \frac{u^*}{|u|}  +I_d\right)
\]
(convention $\frac{0}{|0|}=0$). %so that
%\[
%\left | |u|^{r-1}\frac{u}{|u|}- |v|^{r-1}\frac{v}{|v|}  \right|\le (r-1) \big( |u|\vee |v|\big)^{r-2} |u-v|
%\]
Consequently, Taylor's Theorem with Lagrange remainder applied to $f$ reads
\[
f(u+v) = f(u) + \langle \nabla f(u),v\rangle +\frac 12 v^*\nabla^2f(\xi_{u,v})v
\]
for some  $\xi_{u,v}= \lambda_{u,v} u +(1-\lambda_{u,v})(u+v)$, $\lambda_{u,v}\!\in(0,1)$. Note that
\[
 v^*\nabla^2f(\xi_{u,v})v= r|\xi_{u,v}|^{r-2}\Big((r-2)   \frac{\langle v,\xi_{u,v}\rangle^2}{|\xi_{u,v}|^2}  +|v|^2\Big)\le r |\xi_{u,v}|^{r-2}(r-1)  |v|^2
\]
owing to Cauchy-Schwartz inequality. Then, noting that $|\xi_{u,v}|\le |u|\vee  |u+v|\le |u|+|v|$, we obtain
\begin{align*}
|u+v|^r & \le |u|^r + \big\langle r  |u|^{r-1} \frac{u}{|u|}, v \big\rangle +\frac{r(r-1)}{2}\big( |u|+|v|\big)^{r-2} |v|^2\\
           % & \le |u|^r + \big\langle r  |u|^{r-1} \frac{u}{|u|}, v \big\rangle +\frac{r(r-1)}{2}2^{(r-3)^+}\big( |u| + |v|\big)^{r-2} |v|^2\\
           & \le |u|^r + \big\langle r  |u|^{r-1} \frac{u}{|u|}, v \big\rangle +\frac{r(r-1)}{2}2^{(r-3)_+}\big( |u|^{r-2} + |v|^{r-2}\big) |v|^2\\
           & = |u|^r + \big\langle r  |u|^{r-1} \frac{u}{|u|}, v \big\rangle +\frac{r(r-1)}{2}2^{(r-3)_+}\big( |u|^{r-2} |v|^2 + |v|^{r}\big).
\end{align*}
%Young's inequality applied with conjugate exponents yields $u|^{r-2} |v|^2$
Applying the above inequality  to $u= a$ and $v =  \sqrt{h} \,A Z$ yields
\[
\big|a+A\sqrt{h} Z \big|^r\le |a|^r + r \Big\langle |a|^{r-1} \frac{a}{|a|} ,  A\sqrt{h} \,Z\Big\rangle +2^{(r-3)_+} \frac{r(r-1)}{2}\big(h |a|^{r-2} |AZ|^2 +h^{\frac r2} |AZ|^{r}\big).
\]
Applying Young's inequality  (when $r>2$) to the product $|a|^{r-2}|AZ|^2$ with conjugate exponents $r'= \frac{r}{r-2}$ and $s'= \frac r2$ yields
\begin{align}
\label{inegalitedulemme}
\big|a+A\sqrt{h} Z \big|^r & \le |a|^r + r \Big\langle |a|^{r-1} \frac{a}{|a|} ,  A\sqrt{h} \,Z \Big\rangle +2^{(r-3)_+} \frac{r(r-1)}{2}\Big(\frac h r \big((r-2) |a|^{r}  + 2  |AZ|^r\big) + h^{\frac r2} |AZ|^{r}\Big)\nonumber \\
                                     & \le |a|^r \left( 1+2^{(r-3)_+} \frac{(r-1)(r-2)}{2}h\right)    + r \Big\langle |a|^{r-1} \frac{a}{|a|} ,  A\sqrt{h} \,Z \Big\rangle \nonumber \\
                                     & \quad+ 2^{(r-3)_+}  (r-1)h \|A\|^r |Z|^r\Big(1 + \frac r2\, h^{\frac{r-2}{2}}\Big).
\end{align}
Finally taking expectation and using that $\E\, Z=0$ and $h < h_0$ yields the announced result.
\subsection{Appendix B: Proof of Theorem \ref{timeerror}}
To get into the core of the proof of the first part of Theorem \ref{timeerror}, we need to show some properties of the functions $\bar{y}_k$ and $\bar{z}_k$.
\begin{lem}
\label{Lipandlg}
%	The functions $\bar{y}_{k}$, $\widetilde{y}_k$ and $\bar{z}_{k}$ have linear growth, which means that there exists a constant $C_0$ such that 
%	\begin{equation}
%	\label{lineargrowth}
%	\sup\{|\bar{y}_{k}(x)|, |\widetilde{y}_k(x)|, |\bar{z}_{k}(x)|\} \leq C_{0}(1+|x|).
%	\end{equation}
%	Furthermore, 
	The functions $\bar{y}_{k}$ and $\bar{z}_{k}$ defined by $(\ref{zbark})$-$(\ref{ybark})$ are Lipschitz continuous with $[\bar{y}_{k}]_{\rm Lip}$ and $[\bar{z}_{k}]_{\rm Lip}$ their respective Lipschitz coefficients  given by 
	$$[\bar{y}_k]_{\rm Lip}\leq [h]_{\rm Lip}+\Delta_{\max} (1+\Delta_{\max})[f]_{\rm Lip}+e^{(1+C_f+C_{b,\sigma})\Delta_{\max}}[\bar{y}_{k+1}]_{\rm Lip} $$
	and
	$$[\bar{z}_{k}]_{\rm Lip} \leq  \frac{1}{\sqrt{\Delta}} [\bar{y}_{k+1}]_{\rm Lip} e^{C_{b,\sigma}\Delta}$$
	where $C_{b,\sigma}=1+\Delta_{\max}(2[b_k]_{\rm Lip}+[\sigma_k]_{\rm Lip})+\Delta_{\max}^2[b_k]_{\rm Lip}^2$ and $C_f=2 [f]_{\rm Lip} + [f]_{\rm Lip}^2$.
\end{lem}
\begin{proof}
{\bf STEP $1$:} We show that $\bar{y}_k$ and $\widetilde{y}_k$ are Lipschitz continuous. We rely on a backward induction. In this part, we denote $\mathcal{E}_k^x=\mathcal{E}_k(x,\ve_{k+1})$ for every $x$ to alleviate notations. It is clear that $[\bar y_n]_{\rm Lip}=[g]_{\rm Lip}$. We assume that $\bar y_{k+1}$ is $[\bar y_{k+1}]_{\rm Lip}$-Lipschitz continuous and show the Lipschitz continuity of $\bar{y}_k$. For every $x,x'$, we start by noticing that
	\begin{align*}
	|\widetilde{y}_k(x)-\widetilde{y}_k(x')|&= \Big|\E_k \bar{y}_{k+1}\big( \mathcal{E}_k^x\big)-\E_k \bar{y}_{k+1}\big( \mathcal{E}_k^{x'}\big)  +\Delta \Big( A_k (x-x') +B_k \E_k\left( \bar{y}_{k+1}\big( \mathcal{E}_k^x\big)-\bar{y}_{k+1}\big( \mathcal{E}_k^{x'}\big)\right)\\
	& \quad +\frac{C_k}{\sqrt{\Delta}} \E_k\left( \bar{y}_{k+1}\big( \mathcal{E}_k^x\big)-\bar{y}_{k+1}\big( \mathcal{E}_k^{x'}\big)\right) \ve_{k+1} \Big) \Big|
	\end{align*}
	where 
	\begin{align*}
	A_k & =\frac{\mathcal{E}_k\big(x,\E_k\,\bar{y}_{k+1}(\mathcal{E}_k^x), \bar{z}_k(x)\big)-\mathcal{E}_k\big(x',\E_k\,\bar{y}_{k+1}(\mathcal{E}_k^{x}),\bar{z}_k(x)\big)}{x-x'}\; \mathds{1}_{x \,\neq \, x'},\\
	B_k&=\frac{\mathcal{E}_k\big(x',\E_k\,\bar{y}_{k+1}(\mathcal{E}_k^x),\bar{z}_k(x)\big)-\mathcal{E}_k\big(x',\E_k\,\bar{y}_{k+1}(\mathcal{E}_k^{x'}),\bar{z}_k(x)\big)}{\E_k\left( \bar{y}_{k+1}\big( \mathcal{E}_k^x\big)-\bar{y}_{k+1}\big( \mathcal{E}_k^{x'}\big)\right)} \; \mathds{1}_{\E_k\, \bar{y}_{k+1}\big( \mathcal{E}_k^x\big)\, \neq \, \E\, \bar{y}_{k+1}\big( \mathcal{E}_k^{x'}\big)},\\	
	C_k&=\frac{\mathcal{E}_k\left(x',\E_k\, \bar{y}_{k+1}(\mathcal{E}_k^{x'}),\bar{z}_k(x)\right)-\mathcal{E}_k\big(x',\E_k\, \bar{y}_{k+1}(\mathcal{E}_k^{x'}),\bar{z}_k(x')\big)}{\E_k \big( \bar{y}_{k+1}\big( \mathcal{E}_k^x\big)-\bar{y}_{k+1}\big( \mathcal{E}_k^{x'}\big) \ve_{k+1}\big)} \; \mathds{1}_{\bar{z}_k(x) \, \neq \,\bar{z}_k(x')}.
	\end{align*}
	It is clear that these quantities are $\mathcal{F}_{t_k}$-measurable and that $\max\big(|A_k|,|B_k|,|C_k|\big) \leq [f]_{\rm Lip}$ so
	$$|\widetilde{y}_k(x)-\widetilde{y}_k(x')| \leq \Delta [f]_{\rm Lip} |x-x'|+\E_k \left|\left(\bar{y}_{k+1}(\mathcal{E}_k^x)-\bar{y}_{k+1}(\mathcal{E}_k^{x'})\right)\left(1+\Delta B_k+C_k\sqrt{\Delta}\ve_{k+1}\right) \right|.$$
	Now, using the inequality $(a+b)^2\leq a^2(1+\Delta)+b^2(1+\frac{1}{\Delta})$, one obtains 
	\begin{align*}
	\big|\widetilde{y}_k(x)-\widetilde{y}_k(x')\big|^2 \leq & \, \Delta^2 [f]_{\rm Lip}^2 |x-x'|^2 (1+\tfrac{1}{\Delta})+(1+\Delta)\E_k \left|\big(\bar{y}_{k+1}(\mathcal{E}_k^x)-\bar{y}_{k+1}(\mathcal{E}_k^{x'})\big)\big(1+\Delta B_k+C_k\sqrt{\Delta}\ve_{k+1}\big) \right|^2\\
	\leq &\; \Delta (1+\Delta) [f]_{\rm Lip}^2 |x-x'|^2 +(1+\Delta)\E_k \left|\bar{y}_{k+1}(\mathcal{E}_k^x)-\bar{y}_{k+1}(\mathcal{E}_k^{x'})\right|^2 \E_k\big(1+\Delta B_k+C_k\sqrt{\Delta}\ve_{k+1}\big)^2.
	\end{align*}
	Since, $(\ve_k)_{k \geq 0}$ is a sequence of i.i.d. random variables, then \begin{align*}
	\E_k\left(1+\Delta B_k+C_k\sqrt{\Delta}\ve_{k+1}\right)^2&= (1+[f]_{\rm Lip}\Delta)^2+\Delta [f]_{\rm Lip}^2 \E|\ve_{k+1}|^2 \leq 1+2\Delta [f]_{\rm Lip} +\Delta [f]_{\rm Lip}^2 \leq e^{C_f \Delta},
	\end{align*}
	so that
	$$|\widetilde{y}_k(x)-\widetilde{y}_k(x')|^2 \leq \Delta (1+\Delta) [f]_{\rm Lip}^2 |x-x'|^2 + e^{(C_f+1)\Delta} \E_k \left|\bar{y}_{k+1}(\mathcal{E}_k^x)-\bar{y}_{k+1}(\mathcal{E}_k^{x'})\right|^2.$$
	At this stage, one notes that if $a,b \geq 0$, then $\max(a,b)^2\leq \max(a^2,b^2)$ so 
	\begin{align*}
	|\bar{y}_{k}(x)-\bar{y}_k(x')|^2 &\leq \max \Big(|h_k(x)-h_k(x')|^2, |\widetilde{y}_k(x)-\widetilde{y}_k(x')|^2\Big)\\
	& \leq \max\Big([h]_{\rm Lip}^2|x-x'|^2,\Delta (1+\Delta)[f]_{\rm Lip}^2 |x-x'|^2+e^{\Delta (1+C_f)}\E_k \left|\bar{y}_{k+1}(\mathcal{E}_k^x)-\bar{y}_{k+1}(\mathcal{E}_k^{x'})\right|^2\Big)
	\end{align*}
	We use the fact that $\bar{y}_{k+1}$ is Lipschitz continuous and write
	\begin{align}
	\label{przlip}
	\E_k \left|\bar{y}_{k+1}(\mathcal{E}_k^x)-\bar{y}_{k+1}(\mathcal{E}_k^{x'})\right|^2 & \leq [\bar{y}_{k+1}]_{\rm Lip} \E|x-x'+\Delta (b_k(x)-b_k(x'))+\sqrt{\Delta}(\sigma_k(x)-\sigma_k(x'))\ve_{k+1}|^2\nonumber \\
	& \leq [\bar{y}_{k+1}]_{\rm Lip} |x-x'|^2 (1+\Delta(2[b_k]_{\rm Lip}+[\sigma_k]_{\rm Lip})+\Delta^2[b_k]_{\rm Lip}^2)\nonumber \\
	&\leq  [\bar{y}_{k+1}]_{\text{Lip }}e^{C_{b,\sigma}\Delta} |x-x'|^2
	\end{align}
	where $C_{b,\sigma}=2[b_k]_{\rm Lip}+[\sigma_k]_{\rm Lip} +\Delta_{\max}[b_k]_{\rm Lip}^2$. Therefore, one has 
	\begin{align*}
	|\bar{y}_{k}(x)-\bar{y}_k(x')|^2 &\leq \max \Big([h]_{\rm Lip}^2|x-x'|^2, \Delta (1+\Delta)[f]_{\rm Lip} |x-x'|^2+e^{(1+C_f+C_{b,\sigma})\Delta}[\bar{y}_{k+1}]_{\rm Lip} |x-x'|^2\Big). 
	\end{align*}
	Now, since $\Delta \leq\Delta_{\max}$, one deduces that $\bar{y}_k$ is $[\bar{y}_k]_{\rm Lip}$-Lipschitz continuous with 
	$$[\bar{y}_k]_{\rm Lip}\leq [h]_{\rm Lip}+\Delta_{\max} (1+\Delta_{\max})[f]_{\rm Lip}+e^{(1+C_f+C_{b,\sigma})\Delta_{\max}}[\bar{y}_{k+1}]_{\rm Lip} .$$
{\bf STEP $2$:} %The properties of the function $\bar z_k$.	\\
\iffalse
As for $\bar{z}_k$, we use the fact that $\bar{y}_{k+1}$ is Lipschitz. In fact, since $\E \, \ve_{k+1}=0$,
	\begin{align*}
	\bar{z}_{k}(x)&=\frac{1}{\sqrt{\Delta}} \E\left( \bar{y}_{{k+1}}\big(\mathcal{E}_k(x,\ve_{k+1})\big)\ve_{k+1}\right)= \frac{1}{\sqrt{\Delta}} \E\Big(\Big( \bar{y}_{{k+1}}\big(\mathcal{E}_k(x,\ve_{k+1})\big)-\bar{y}_{k+1}(x+\Delta b_k(x))\Big)\ve_{k+1}\Big)
	\end{align*}
	Then, 
	\begin{align*}
	|\bar{z}_{k}(x)| &\leq \frac{[\bar{y}_{k+1}]_{\rm Lip}}{\sqrt{\Delta}} \E \big(|\sqrt{\Delta} \sigma(x) \ve_{k+1}|\ve_{k+1} \big) \leq  [\bar{y}_{k+1}]_{\rm Lip} \|\sigma(x)\| \E|\ve_{k+1}|^2 \leq  [\bar{y}_{k+1}]_{\rm Lip}L_{b,\sigma}(1+|x|)
	\end{align*}
	since $\sigma$ has itself linear growth. \\
	\fi
	For the Lipschitz continuity of $\bar z_k$, we will use the same property of $\bar y_{k+1}$, more precisely inequality $(\ref{przlip})$. For every $x,x' \in \R^d$, 
	\begin{align*}
	|\bar{z}_k(x)-\bar{z}_k(x')|^2 \leq & \frac{1}{\sqrt{\Delta}} \E\left|(\bar{y}_{k+1}(\mathcal{E}_k^x)-\bar{y}_{k+1}(\mathcal{E}_k^{x'}))\ve_{k+1}\right|\leq \frac{1}{\sqrt{\Delta}} [\bar{y}_{k+1}]_{\rm Lip} e^{C_{b,\sigma}\Delta} |x-x'|^2.
	\end{align*}
	\hfill $\square$
	\end{proof}
\begin{refproof}% Proof of Theorem \ref{timeerror}}
	We denote $\delta V_t =V_t-\bar{V}_t$ for any process $V$. We consider the following stopping times 
	\begin{equation}
	\label{tauc}
	\tau^c=\inf\left\{u \geq t ; \; \int_t^u \mathds{1}_{\delta Y_s>0}dK_s>0 \right\} \wedge T,
	\end{equation}
	\begin{equation}
	\label{taud}
	\tau^d=\min \left\{t_j \geq t; \; \mathds{1}_{\delta Y_i <0} \big( h_i(\bar{X}_i)-\widetilde{Y}_i\big)_+>0\right\}\wedge T
	\end{equation}
	and
	$$\tau=\tau^c \wedge \tau^d.$$
	Keeping in mind that $(\bar{Y}_t)_t$ is a c\`agl\`ad process (see $(\ref{YbarYtildecont})$), we use It\^o's formula between $t$ and $\tau$ to write
	\begin{align*}
	|\delta Y_{\tau}|^2&=|\delta Y_{t}|^2+2 \int_{[t,\tau)}\delta Y_s d\delta Y_s+\int_{[t,\tau)}|\delta Z_s^2|ds+\sum_{t \leq s < \tau} ( \delta Y_s-\delta Y_{s^-})^2\\
	&= |\delta Y_{t}|^2 -2 \int_{[t,\tau)}\delta Y_s \big( f(\Theta_s)-f(\bar{\Theta}_{\underline{s}})\big) ds -2\int_{[t,\tau)}\delta Y_s  dK_s +2 \int_{[t,\tau)}\delta Y_s d\bar{K}_s\\
	& \quad +\int_{[t,\tau)}(Z_s- \bar{Z}_s)dW_s +\int_{[t,\tau)}|\delta Z_s^2|ds+\sum_{t \leq s < \tau} ( \delta Y_s-\delta Y_{s^-})^2
	\end{align*}
	where $\Theta_s=(X_s,Y_s,Z_s)$, $\bar{\Theta}_{\underline{s}}=(\bar{X}_{\underline{s}}, \E_{\underline{s}}\bar{Y}_{\bar{s}}, \bar{\zeta}_{\underline{s}})$, $\underline{s}=t_i$ and $\bar{s}=t_{i+1}$ if $s\in (t_i,t_{i+1})$. One notes that $( \delta Y_s-\delta Y_{s^-})^2= (\bar{Y}_s-\widetilde{Y}_s)^2$ so that, by the definition of the process $\bar{K}_s$, one has 
	\begin{align}
	|\delta Y_t|^2=& |\delta Y_{\tau}|^2+2 \int _t^{\tau} \delta Y_s \Big( f\big(\Theta_s\big)-f\big(\bar{\Theta}_{\underline{s}}\big)\Big) ds +2\int_{[t,\tau)}\delta Y_s  dK_s- \int_{[t,\tau)}(Z_s- \bar{Z}_s)dW_s  \nonumber \\ 
	&-\int_{[t,\tau)}|\delta Z_s^2|ds -\sum_{t \leq t_i < \tau}\left(2 \delta Y_i (h_i(\bar{X}_i)-\widetilde{Y}_i)_++(\bar{Y}_i-\widetilde{Y}_i)^2 \right).
	\end{align}
	For every $t_i < \tau$, we set $\alpha_i = 2 \delta Y_i\,  (h_i(\bar{X}_i)-\widetilde{Y}_i)_++(\bar{Y}_i-\widetilde{Y}_i)^2$ for convenience. It can be written as follows: 
	\begin{align*}
	\alpha_i &=  2(Y_i-\bar{Y}_i)\big(h_i(\bar{X}_i) \vee \widetilde{Y}_i-\widetilde{Y}_i\big)+(\bar{Y}_i-\widetilde{Y}_i)^2 \\
	&= 2(Y_i-\bar{Y}_i)(\bar{Y}_i-\widetilde{Y}_i)+(\bar{Y}_i-\widetilde{Y}_i)^2 \\
	&= (Y_i-\widetilde{Y}_i)^2-(Y_i-\bar{Y}_i)^2\\
	&= (Y_i-\widetilde{Y}_i)^2-(\delta Y_i)^2.
	\end{align*}
	where we used, in the third line, the equality $2(a-b)(b-c)+(b-c)^2=(a-c)^2-(a-b)^2$. \\
	
	Let us evaluate this term $\alpha_i$. For every $t_i < \tau \leq \tau^d$, we have, by $(\ref{taud})$, two choices: Either $h_i(\bar{X}_i)<\widetilde{Y}_i$ so that $\bar{Y}_i=\widetilde{Y}_i$ and hence, $\delta Y_i =Y_i- \widetilde{Y}_i$ and $(\delta Y_i)^2 =(Y_i- \widetilde{Y}_i)^2$, or, $\delta Y_i >0 $ so, since $\widetilde{Y}_t < \bar{Y}_t$ for every $t$, we have  $\widetilde{Y}_i -Y_i< \bar{Y}_i-Y_i<0$ and then, $(\delta Y_i)^2< (Y_i-\widetilde{Y}_i)^2$.\\
	Consequently, for every $t_i \in [t, \tau[$, $$\alpha_i =(Y_i-\widetilde{Y}_i)^2-(\delta Y_i)^2\geq 0.$$ 
	Moreover, for $s \in [t,\tau[,$ $s < \tau^c$ so that, by $(\ref{tauc})$, we have $\delta Y_s<0$ $dK_s$-a.e. Hence,
	$$\int_{[t,\tau)} \delta Y_s dK_s<0.$$
	This yields 
	$$|\delta Y_t|^2\leq  |\delta Y_{\tau}|^2+2 \int _t^{\tau} \delta Y_s \big( f(\Theta_s)-f(\bar{\Theta}_{\underline{s}})\big) ds - \int_{[t,\tau)}(Z_s- \bar{Z}_s)dW_s -\int_{[t,\tau)}|\delta Z_s^2|ds .$$ 
	Now, we evaluate $|\delta Y_{\tau}|^2$ depending on the value of $\tau$. \\
	\smallskip
	$\bullet$ If $\tau=\tau^d$, then, by $(\ref{taud})$, $\delta Y_{\tau}<0$ and $h_{\tau}(\bar{X}_{\tau})> \widetilde{Y}_{\tau}$. This means $\bar{Y}_{\tau}=h_{\tau}(\bar{X}_{\tau})$ and, since $Y_t \geq h_t(X_t)$ for every $t \in [0,T]$, 
		$$0 \leq |\delta Y_{\tau}|=\bar{Y}_{\tau}-Y_{\tau}=h_{\tau}(\bar{X}_{\tau})-Y_{\tau}\leq h_{\tau}(\bar{X}_{\tau})- h_{\tau}(X_{\tau}) .$$
		Hence, $|\delta Y_{\tau}|^2\leq [h]_{\rm Lip}^2|X_{\tau}-\bar{X}_{\tau}|^2.$\\
		\smallskip
		$\bullet$ If $\tau=\tau^c$ , then, by $(\ref{tauc})$, $\delta Y_{\tau}> 0$ and $K_s$ changes its value so $Y_{\tau}=h_{\tau}(X_{\tau})$. Consequently, 
		$$0\leq \delta Y_{\tau}=Y_{\tau}-\bar{Y}_{\tau}= h_{\tau}(X_{\tau})-\bar{Y}_{\tau} \leq h_{\tau}(X_{\tau})-h_{\tau}(\bar{X}_{\tau})$$
		since $\bar{Y}_t\geq h_t(\bar{X}_t)$ for every $t \in [0,T]$. So, $|\delta Y_{\tau}|^2\leq [h]_{\rm Lip}^2|X_{\tau}-\bar{X}_{\tau}|^2.$\\
		\smallskip
		$\bullet$ If $\tau=T$, $\delta Y_{T}=g(X_T) -g(\bar{X}_T)$ so 
		$|\delta Y_{\tau}|^2\leq [g]_{\rm Lip}^2|X_{\tau}-\bar{X}_{\tau}|^2.$ 
	Consequently, for all the possible values of $\tau$, we have 
	$$|\delta Y_{\tau}|^2\leq C_{h,g,b,T,\sigma} \Delta$$
	where $C_{h,g,b,T,\sigma}$ is a constant related to the Euler discretization error and depending on $h$ and $g$. Thus, taking the conditional expectation with respect to $t$ leads to 
	\begin{equation}
	\E_t \left(|\delta Y_t|^2 +\int_t^{\tau} |\delta Z_s|^2 ds \right) \leq C_{h,g,b,T,\sigma} \Delta + 2 \E_t \int_t^{\tau} \delta Y_s \big(f_s(\Theta_s)-f_{\underline{s}}(\bar{\Theta}_{\underline{s}}) \big)-\E_t\int_{[t,\tau)}(Z_s- \bar{Z}_s)dW_s.
	\end{equation}
	
	It remains to study the term $2 \E_t \int_t^{\tau} \delta Y_s \big(f_s(\Theta_s)-f_{\underline{s}}(\bar{\Theta}_{\underline{s}}) \big)$. As $f$ is Lipschitz continuous, we use Young's inequality $ab \leq \frac{a^2}{2\alpha}+\frac{\alpha b^2}{2}$ and the inequality $(a+b+c)^2 \leq 3(a^2+b^2+c^2)$ to write
	\begin{align}
	\label{termedef1}
	2 \E_t \int_t^{\tau} \delta Y_s \big(f_s(\Theta_s)-f_{\underline{s}}(\bar{\Theta}_{\underline{s}}) \big) \leq &\, \frac{3[f]_{\rm Lip}}{\alpha} \left( \int_t^{\tau} \E_t|X_s-\bar{X}_{\underline{s}}|^2ds +\int_t^{\tau} \E_t|Y_s-\E_{\underline{s}}\bar{Y}_{\bar{s}}|^2ds \right.\nonumber \\
	& \left.+\E_t\int_t^{\tau} |Z_s-\bar{\zeta}_{\underline{s}}|^2ds \right) +\alpha [f]_{\rm Lip} \E_t\int_t^{\tau} |\delta Y_s|^2ds.
	\end{align}
	On the one hand, 
	$$\E_t|X_s-\bar{X}_{\underline{s}}|^2 \leq 2 \E_t|X_s-X_{\underline{s}}|^2+2 \E_t|X_{\underline{s}}-\bar{X}_{\underline{s}}|^2$$
	where 
	%$$\E_t|X_{\underline{s}}-\bar{X}_{\underline{s}}|^2 \leq C\Delta(1+|x|)^2$$ since it is an Euler discretization error and 
	$\E_t|X_s-X_{\underline{s}}|^2$ is bounded as follows: from $(\ref{SDE})$ taken between $\underline{s}$ and $s$, we have 
	\begin{align*}
	\E_t|X_s-X_{\underline{s}}|^2   &\leq 2\E_t \int_{\underline{s}}^s b_u(X_u)^2du + 2\E_t \int_{\underline{s}}^s \sigma_u(X_u)^2du\\
	& \leq 4 L_{b, \sigma}^2 \E_t \int_{\underline{s}}^s (1+|X_u|)^2du \\
	&\leq 4 L_{b,\sigma}^2 \Delta \E_t \sup_{ \underline{s}\leq u \leq s}(1+|X_u|)^2 .
	\end{align*} 
	Hence, denoting $C_X=4 L_{b,\sigma}^2 (\tau-t)$, 
	\begin{equation}
	\label{borneX}
	\int_t^{\tau} \E_t|X_s-\bar{X}_{\underline{s}}|^2ds \leq C_X \Delta \E_t \sup_{ \underline{s}\leq u \leq s}(1+|X_u|)^2+2\int_t^{\tau} \E_t|X_{\underline{s}}-\bar{X}_{\underline{s}}|^2ds.
	\end{equation}
	On the other hand, 
	\begin{equation}
	\label{termedey}
	\E_t|Y_s-\E_{\underline{s}}\bar{Y}_{\bar{s}}|^2 \leq 2 \E_t|Y_s-\bar{Y}_s|^2 +4\E_t|\bar{Y}_s-\widetilde{Y}_{\underline{s}}|^2+4\E_t\E_{\underline{s}}|\widetilde{Y}_{\underline{s}}-\bar{Y}_{\bar{s}}|^2.
	\end{equation}
	For every $v, v'$ such that $v<v'$ and $|v-v'|\leq \Delta$, $(\ref{YbarYtildecont})$ at $v$ and $v'$ yields
	$$\widetilde{Y}_v-\bar{Y}_{v'}=(v'-v)f\big(\underline{v}, \, \bar{X}_{\underline{v}},\,  \E_{\underline{v}} \bar{Y}_{\bar{v}},\, \bar{\zeta}_{\underline{v}}\big)-\int_v^{v'} \bar{Z}_s dW_s+\bar{K}_{v'}-\bar{K}_v$$
	so that taking the conditional expectations w.r.t. $t$ yields
	\begin{align*}
	\E_t|\widetilde{Y}_v-\bar{Y}_{v'}|^2 \leq & 2(v'-v)^2 \E_t f(\bar{\theta}_{\underline{v}})^2 +2 \E_t\left(\int_v^{v'} \bar{Z}_s dWs \right)^2 +2 \E_t(\bar{K}_{v'}-\bar{K}_v)^2\\
	\leq & 2 \Delta^2 \E_t f(\bar{\theta}_{\underline{v}})^2 +2\E_t\left(\int_v^{v'} \bar{Z}_s dWs \right)^2 +2 \E_t(\bar{K}_{v'}-\bar{K}_v)^2.
	\end{align*}
%	Noting $ f(\bar{\Theta}_{\underline{s}})=f(\bar{X}_{\underline{s}}, \E_{\underline{s}}\bar{Y}_{\bar{s}}, \bar{\zeta}_{\underline{s}})$, is a composition of the functions $f$, $\bar{y}_{\underline{s}}$ and $\bar{z}_{\underline{s}} $ which all have linear growth according to Lemma $\ref{croissancelineaire}$. Moreover, by its definition, $\bar{Z}_s$ is $L^2$-integrable, so we can deduce that
%	$$\E|\widetilde{Y}_v-\bar{Y}_{v'}|^2 \leq 2\Delta^2 C_{0}C_f \E_t(1+\sup_{v \leq s \leq v'} |\bar{X}_s|)^2 +2\E_t\left(\int_v^{v'} \bar{Z}_s dWs \right)^2  +2 \E(\bar{K}_{v'}-\bar{K}_v)^2.$$
	Since $\bar{K}_v\geq 0$ for every $v \in [0,T]$, we have $-\bar{K}_v < \bar{K}_v$ so that $\bar{K}_{v'}-\bar{K}_v< \bar{K}_{v'}+\bar{K}_v$. Then, owing the fact that $\bar{K}_v$ is non decreasing, $\bar{K}_{v'}-\bar{K}_v\geq 0$ so 
	$$(\bar{K}_{v'}-\bar{K}_v)^2 \leq (\bar{K}_{v'}-\bar{K}_v)(\bar{K}_{v'}+\bar{K}_v) = \bar{K}_{v'}^2-\bar{K}_v^2.$$
	Hence, noting that $ f(\bar{\Theta}_{\underline{s}})=f(\bar{X}_{\underline{s}}, \E_{\underline{s}}\bar{Y}_{\bar{s}}, \bar{\zeta}_{\underline{s}})$ is a composition of the functions $f$, $\bar{y}_{\underline{s}}$ and $\bar{z}_{\underline{s}} $ which are all Lipschitz continuous according to Lemma $\ref{Lipandlg}$ and recalling that if a function $g$ is Lipschitz continuous then it has linear growth i.e. there exists a finite constant $C_0$ such that $g(x) \leq C(1+|x|)$, one has
	$$\E_t|\widetilde{Y}_v-\bar{Y}_{v'}|^2 \leq 2\Delta^2 C_{0} \E_t(1+\sup_{v \leq s \leq v'} |\bar{X}_s|)^2 +2\E_t\left(\int_v^{v'} \bar{Z}_s dWs \right)^2 +2 \E_t(\bar{K}_{v'}^2-\bar{K}_v^2).$$
	Combining this with $(\ref{termedey})$ twice yields
	\begin{align}
	\label{borneY}
	\int_t^{\tau} \E_t |Y_s-\E_{\underline{s}}\bar{Y}_{\bar{s}}|^2 ds \leq & \, 2 \int_t^{\tau} \E_t |\delta Y_s|^2 ds +4\sum_{i=\underline{t}/\Delta}^{(\bar{\tau}/\Delta)-1} \int_{t_i}^{t_{i+1}} \E_t |\bar{Y}_s -\widetilde{Y}_{\underline{s}}|^2ds +4\sum_{i=\underline{t}/\Delta}^{(\bar{\tau}/\Delta)-1} \int_{t_i}^{t_{i+1}} \E_t |\bar{Y}_{\bar{s}}- \widetilde{Y}_{\underline{s}}|^2ds\nonumber \\
	\leq & \, 2 \int_t^{\tau} \E_t |\delta Y_s|^2 ds +8\Delta^2 C_{0} (\tau-t) \E_t(1+\sup_{\underline{s} \leq u \leq \bar{s}} |\bar{X}_u|)^2\nonumber \\
	&+  8 \sum_{i=\underline{t}/\Delta}^{(\bar{\tau}/\Delta)-1} \int_{t_i}^{t_{i+1}} \E_t(\bar{K}_{s}^2-\bar{K}_{\underline{s}}^2)+\E_t(\bar{K}_{\bar{s}}^2-\bar{K}_{\underline{s}}^2)\nonumber \\
	&+ 8\sum_{i=\underline{t}/\Delta}^{(\bar{\tau}/\Delta)-1}\int_{t_i}^{t_{i+1}}\E_t\left(\int_{\underline{s}}^{s} \bar{Z}_u dWu \right)^2ds +8\sum_{i=\underline{t}/\Delta}^{(\bar{\tau}/\Delta)-1}\int_{t_i}^{t_{i+1}}\E_t\left(\int_{\underline{s}}^{\bar s} \bar{Z}_u dWu \right)^2ds\nonumber  \\
	\leq & \, 2 \int_t^{\tau} \E_t |\delta Y_s|^2 ds +8\Delta(\bar{\tau}-\underline{t})\E_t|\bar K_T|^2+8\Delta^2 C_{0}C_f (\tau-t) \E_t(1+\sup_{\underline{s} \leq u \leq \bar{s}} |\bar{X}_u|)^2\nonumber \\
	&+ 8\sum_{i=\underline{t}/\Delta}^{(\bar{\tau}/\Delta)-1}\int_{t_i}^{t_{i+1}}\left(\E_t\Big(\int_{\underline{s}}^{s} \bar{Z}_u dWu \Big)^2+\E_t\Big(\int_{\underline{s}}^{\bar s} \bar{Z}_u dWu \Big)^2\right)ds 
	\end{align}
	where we used the fact that $\bar{K}_t$ is a non-decreasing positive process so for every $t\in [0,T]$, $\bar{K}_t< \bar{K}_T$ and the fact that $\sup_{\underline{s} \leq u \leq s} \alpha_u \leq \sup_{\underline{s} \leq u \leq \bar{s}} \alpha_u$. \\
	
	%As stated in $(\ref{assumptionRBSDE})$, $\E_t |\bar{K}_T|^2< \gamma_0$ so denoting $C_Y =\bar{C}_Y+8\gamma_0(\bar{\tau}-\underline{t})$, one concludes with 
%	\begin{equation}
%	\label{borneY}
%	\int_t^{\tau} \E_t |Y_s-\E_{\underline{s}}\bar{Y}_{\bar{s}}|^2 ds \leq C_Y \Delta +2 \int_t^{\tau} \E_t |\delta Y_s|^2 ds.
%	\end{equation}
	Thirdly, 
	\begin{align*}
	\E_t\int_t^{\tau} |Z_s -\bar{\zeta}_{\underline{s}}|^2ds \, &\leq \, \sum_{i=\underline{t}/\Delta}^{(\bar{\tau}/\Delta)-1} \E_t \int_{t_i}^{t_{i+1}}|Z_s -\bar{\zeta}_{\underline{s}}|^2ds  \\
	& \leq \, \sum_{i=\underline{t}/\Delta}^{(\bar{\tau}/\Delta)-1} \E_t \left(4\int_{t_i}^{t_{i+1}}|Z_s -Z_{\underline{s}}|^2ds +4\int_{t_i}^{t_{i+1}}|Z_{\underline{s}} -{\zeta}_{\underline{s}}|^2ds +2\int_{t_i}^{t_{i+1}}|{\zeta}_{\underline{s}} -\bar{\zeta}_{\underline{s}}|^2ds   \right).
	\end{align*}
	By the definitions $(\ref{zeta})$ and $(\ref{zetabar})$ of $\zeta_s$ and $\bar{\zeta}_s$, we have 
	\begin{align*}
	|Z_{\underline{s}} -{\zeta}_{\underline{s}}|^2 =\Big|Z_{\underline{s}} -\frac{1}{\Delta}\E_{\underline{s}} \int_{\underline{s}}^{\bar{s}} Z_s ds \Big|^2=\frac{1}{\Delta^2} \Big|\E_{\underline{s}}\int_{\underline{s}}^{\bar{s}} (Z_s -Z_{\underline{s}}) \Big|^2 \leq \frac{1}{\Delta} \E_{\underline{s}}\int_{\underline{s}}^{\bar{s}}|Z_s -Z_{\underline{s}}|^2 ds
	\end{align*}
	where the last inequality was obtained by using Cauchy-Schwarz inequality. Hence, we use Fubini's Theorem to deduce 
	$$\int_{t_i}^{t_{i+1}} |Z_{\underline{s}} -{\zeta}_{\underline{s}}|^2ds\leq  \E_{\underline{s}} \int_{\underline{s}}^{\bar{s}}|Z_s -Z_{\underline{s}}|^2 ds.$$ 
	Likewise, $$|{\zeta}_{\underline{s}} -\bar{\zeta}_{\underline{s}}|^2 =\big|  \frac{1}{\Delta} \E_{\underline{s}} \int_{\underline{s}}^{\bar{s}} (Z_s- \bar{Z}_s)\big|^2 \leq \frac{1}{\Delta} \E_{\underline{s}}\int_{\underline{s}}^{\bar{s}}|Z_s -\bar{Z}_s|^2 ds $$so that
	$$ \int_{t_i}^{t_{i+1}}|{\zeta}_{\underline{s}} -\bar{\zeta}_{\underline{s}}|^2 ds  \leq \E_{\underline{s}} \int_{\underline{s}}^{\bar{s}}|Z_s -\bar{Z}_s|^2 ds.$$
	Consequently, 
	\begin{equation}
	\label{borneZ}
	\E_t\int_t^{\tau} |Z_s -\bar{\zeta}_{\underline{s}}|^2ds \leq 8\E_t \int_{\underline{t}}^{\bar{\tau}} |Z_s -Z_{\underline{s}}|^2ds +2 \E_t \int_{\underline{t}}^{\bar{\tau}} |Z_s -\bar{Z}_s|^2ds.
	\end{equation}
	
	At this stage, we merge the 3 equations $(\ref{borneX})$, $(\ref{borneY})$ and $(\ref{borneZ})$ with $(\ref{termedef1})$ and take the expectation to obtain 
	\begin{align*}
	\E\Big(|\delta Y_t|^2+\int_t^{\tau} |\delta Z_t|^2 \Big) \leq & \Delta \Big(C_{h,g,b,T,\sigma}+\frac{6[f]}{\alpha}\big( C_X+C_{0} (\bar{\tau}-\underline{t})\big)\E(1+\sup_{\underline{s} \leq u \leq \bar{s}} |\bar{X}_u|)^2+\frac{12[f]}{\alpha}(\bar{\tau}-\underline{t}) \E|\bar K_T|^2\Big)\\
	&+\left(2\alpha[f]_{\rm Lip}+\frac{6[f]_{\rm Lip}}{\alpha}\right)\int_t^{\tau} \E|\delta Y_s|^2 ds + \frac{6[f]_{\rm Lip}}{\alpha} \int_t^{\tau} \E|X_{\underline{s}}-\bar{X}_{\underline{s}}|^2ds\\
	&+\frac{24[f]_{\rm Lip}}{\alpha}\E \int_{\underline{t}}^{\bar{\tau}} |Z_s -Z_{\underline{s}}|^2ds+\frac{6[f]_{\rm Lip}}{\alpha}\E \int_{\underline{t}}^{\bar{\tau}} |Z_s -\bar{Z}_s|^2ds \\
	&+ \frac{24[f]_{\rm Lip}}{\alpha}\sum_{i=\underline{t}/\Delta}^{(\bar{\tau}/\Delta)-1}\int_{t_i}^{t_{i+1}}\left(\E\int_{\underline{s}}^{s} |\bar{Z}_u|^2 du +\E\int_{\underline{s}}^{\bar s} |\bar{Z}_u|^2 du\right)ds .
\end{align*}
	As stated in $(\ref{assumptionRBSDE})$, $\E|\bar K_T|^2 \leq \gamma_0$ and by the classical properties of the Euler scheme, we have 
	$$\E(1+\sup_u|\bar X_u|)^2 \leq C_{b,T,\sigma}(1+|x_0|)^2\qquad \mbox{and} \qquad \E|X_{\underline{s}}-\bar{X}_{\underline{s}}|^2 \leq C_{b,T,\sigma}\Delta(1+|x_0|)^2.$$
	Moreover, 
	$$\sum_{i=\underline{t}/\Delta}^{(\bar{\tau}/\Delta)-1}\int_{t_i}^{t_{i+1}}\left(\E\int_{\underline{s}}^{s} |\bar{Z}_u|^2 du +\E\int_{\underline{s}}^{\bar s} |\bar{Z}_u|^2 du\right)ds \leq \sum_{i=\underline{t}/\Delta}^{(\bar{\tau}/\Delta)-1}\int_{t_i}^{t_{i+1}} 2\Delta \E \sup_{\underline{s}\leq u \leq \overline{s}}|\overline{Z}_u|^2 \leq 2\Delta^2 (\bar{\tau}-\underline{t})\gamma_1.$$
	Hence, if we consider $\alpha=6[f]_{\rm Lip}$ and denote $\bar{C}=C_{h,g,b,T,\sigma}+C_X+(\bar{\tau}-\underline{t})\Big(2\gamma_0+4\gamma_1 \Delta_{\max}+(1+C_{0})C_{b,T,\sigma}(1+|x_0|)^2 \Big)$ and $\widetilde{C}=1+12[f]_{\rm Lip}^2$, then we obtain
	\begin{align*}
	\E\left(|\delta Y_t|^2+\int_t^{\tau} |\delta Z_s|^2 \right) \leq & \Delta \bar{C} +\widetilde{C} \int_t^{\tau} \E|\delta Y_s|^2 ds+4\E \int_{\underline{t}}^{\bar{\tau}} |Z_s -Z_{\underline{s}}|^2ds\\
	&+ \E \int_{\underline{t}}^{t} |Z_s -\bar{Z}_s|^2ds+\E \int_{t}^{\tau} |Z_s -\bar{Z}_s|^2ds+\E \int_{\tau}^{\bar{\tau}} |Z_s -\bar{Z}_s|^2ds .
	\end{align*}
	Consequently, 
	$$\E|\delta Y_t|^2 \leq \widetilde{C} \int_t^{\tau} \E|\delta Y_s|^2 ds +K$$
	where $K=\Delta \bar{C}+4\E \int_{\underline{t}}^{\bar{\tau}} |Z_s -Z_{\underline{s}}|^2ds+ \E\int_{\underline{t}}^{t} |Z_s -\bar{Z}_s|^2ds+\E \int_{\tau}^{\bar{\tau}} |Z_s -\bar{Z}_s|^2ds .$ Let us denote $f(t)=\E|\delta Y_t|^2$. This function satisfies 
	$$f(t) \leq \widetilde{C} \int_t^{\tau} f(s)ds +K.$$
	We consider $g(t)=f(T-t)$ which satisfies also 
	$$g(t) \leq \widetilde{C} \int_0^t g(s)ds +K.$$
	Hence, Gronwall's Lemma yields $g(t) \leq e^{\widetilde{C}t} K$ so that 
	$$f(t) \leq e^{\widetilde{C}(T-t)}K.$$
	Consequently, 
	$$\E|Y_t -\bar{Y}_t|^2 \leq e^{\widetilde{C}(T-t)} \left(\Delta \bar{C}+4\E\int_{0}^{T} |Z_s -Z_{\underline{s}}|^2ds+ \E \int_{\underline{t}}^{t} |Z_s -\bar{Z}_s|^2ds+\E \int_{\tau}^{\bar{\tau}} |Z_s -\bar{Z}_s|^2ds \right).$$
	In particular, if $t=t_k$ and $\tau=t_{k'}$, $k, k' \in \{1,\ldots,n\}$, then $\underline{t}=t$ and $\bar{\tau}=\tau$ so 
	$$\E|Y_k -\bar{Y}_k|^2 \leq e^{\widetilde{C}(T-t_k)} \left(\Delta \bar{C}+4\int_{0}^{T} \E|Z_s -Z_{\underline{s}}|^2ds+ 0+0 \right).$$
	This completes the proof.
	\hfill $\square$
	\end{refproof}
\end{document}